\documentclass[10pt,a4paper,draft]{article}
\usepackage{bbm}
\usepackage{pifont}
\usepackage{bbding}
\usepackage{amsmath}
\usepackage{mathrsfs}
\usepackage{amsfonts}
\usepackage{amssymb}
\usepackage{amsfonts,amssymb,amsmath,indentfirst,amsthm}
\usepackage{epsfig}

\numberwithin{equation}{section}
\newenvironment{abs}{\textbf{Abstract}\mbox{  }}{ }
\newenvironment{key words}{\emph{\texttt{Keywords}}\mbox{  }}{ }
\newtheorem{theorem}{Theorem}[section]
\newtheorem{lemma}[theorem]{Lemma}

\newtheorem{proposition}[theorem]{Proposition}

\renewenvironment{proof}{\noindent{\textbf{Proof.}}}{\hfill$\Box$}
\newenvironment{altproof}[1]
{\addvspace{0.3cm}\noindent
{\textbf{Proof of Theorem {#1}}.}}
{\nopagebreak\mbox{}\hfill $\Box$\par\addvspace{0.5cm}}
\theoremstyle{remark}
\newtheorem{remark}[theorem]{Remark}

\theoremstyle{plain}

\makeatletter
    
    \newcommand{\Rmnum}[1]{\expandafter\@slowromancap\romannumeral #1@}
\makeatother

\newcommand{\wt}{\widetilde}
\newcommand{\wh}{\widehat}
\newcommand{\ov}{\overline}

   \def\dist{\mathrm{dist}}
   \def\span{\mathrm{span}}

   \def\R{\mathbb{R}}

\newcommand{\cO}{{\mathcal O}}

\newcommand{\vphi}{\varphi}
\newcommand{\al}{\alpha}
\newcommand{\be}{\beta}
\newcommand{\ga}{\gamma}
\newcommand{\de}{\delta}
\newcommand{\ze}{\zeta}
\newcommand{\la}{\lambda}
\newcommand{\si}{\sigma}

\newcommand{\De}{\Delta}

\newcommand{\La}{\Lambda}
\newcommand{\Om}{\Omega}

\newcommand{\eps}{\varepsilon}

\newcommand{\weakto}{\rightharpoonup}
\newcommand{\pa}{\partial}

\def\Ker{\mathrm{Ker}}

\begin{document}

\title{\textbf{Multi-bubble nodal solutions to slightly subcritical elliptic problems with Hardy terms}}
\author{Thomas Bartsch, Qianqiao Guo\thanks{Q.G. was supported by the National Natural Science Foundation of China (Grant Nos.11001221,11271299) and the Fundamental Research Funds for the Central Universities (3102015ZY069).}}
\date{}
\maketitle

% ----------------------------------------------------------------
\noindent
\begin{abs}
The paper is concerned with the slightly subcritical elliptic problem with Hardy term
\[
\left\{
\begin{aligned}
-\De u-\mu\frac{u}{|x|^2} &= |u|^{2^{\ast}-2-\eps}u &&\quad \text{in } \Om, \\\
u &= 0&&\quad \text{on } \pa\Om,
\end{aligned}
\right.
\]
in a bounded domain $\Om\subset\R^N$ with $0\in\Om$, in dimensions $N\ge7$. We prove the existence of multi-bubble nodal solutions that blow up positively at the origin and negatively at a different point as $\eps\to0$ and $\mu=\eps^\al$ with $\al>\frac{N-4}{N-2}$. In the case of $\Om$ being a ball centered at the origin we can obtain solutions with up to $5$ bubbles of different signs. We also obtain nodal bubble tower solutions, i.e.~superpositions of bubbles of different signs, all blowing up at the origin but with different blow-up order. The asymptotic shape of the solutions is determined in detail.
\end{abs}

\vspace{.2cm}

\vspace{.2cm}

%---------------------------------------------------------------------------------
\section{\textbf{Introduction}\label{Section 1}}
The paper is concerned with the semilinear singular problem
\begin{equation}\label{pro}
\left\{
\begin{aligned}
-\De u-\mu\frac{u}{|x|^2} &= |u|^{2^{\ast}-2-\eps}u &&\quad \text{in }\Om, \\\
u &= 0 &&\quad \text{on }\pa \Om,
\end{aligned}
\right.
\end{equation}
where $\Om\subset\R^N$, $N\ge7$, is a smooth bounded domain with $0\in\Om$; $2^*:=\frac{2N}{N-2}$ is the critical Sobolev exponent. We study the existence of nodal (i.e.\ sign changing) solutions that have multiple blow ups as $0<\eps\to0$ and $\mu=\mu_0\eps^\al$, with $\mu_0>0$ and $\al>0$ constants.

The blow-up phenomenon for positive and for nodal solutions to problem \eqref{pro} has been studied extensively in the case $\mu=0$. It was proved in \cite{BrePe89PNDEA, FluWeip7MM, Han91AIHPAN, Rey89MM, Rey90JFA} that as $\eps\to0^+$, the positive solution $u_\eps$ blows up and concentrates at a critical point of the Robin's function of $\Om$. In \cite{BaLiRey95CVPDE, Rey91DIE}, the existence of positive solutions with multiple bubbles was considered. In convex domains a positive solution cannot have multiple bubbles, see \cite{GroTaka10JFA}. The existence of nodal solutions with $k$ bubbles at $k$ different points was proved in \cite{BarMiPis06CVPDE} in the case $k=2$, in \cite{BarDPis-prepr1} in the case $k=4$ when $\Om$ is convex and satisfies a certain symmetry, and in \cite{BarDPis-prepr2} in the case $k=3$ when $\Om$ is a ball. On the other hand, nodal solutions with tower of bubbles were obtained in \cite{MussoPis10JMPA, PisWe07AIHPA}. All these papers only treat the regular case $\mu=0$.

When $\mu\neq0$, the  Hardy potential $\frac{1}{|x|^2}$ cannot be regarded as a lower order perturbation because it has the same homogeneity as the Laplace operator and does not belong to the Kato class. This makes the analysis much more complicated compared with the case $\mu=0$. For the problem with Hardy type potentials and critical exponents, much attention has been paid to the existence of positive and nodal solutions, see e.g.\
\cite{CaoH04JDE, CaoP03JDE, EkeGhou02, FG, GY, GuoNiu08JDE, Jan99, RuWill03, Smets05TAMS, Terr96}. However, few results are known about the existence of positive or nodal solutions with multiple bubbles to the problem involving Hardy type potentials and critical exponents. We are only aware of the papers \cite{FelliPis06CPDE, FelliTer05CCM}, dealing with the problem
\begin{equation}\label{Pro-Felli}
\left\{
\begin{aligned}
&-\De u-\frac{\mu}{|x|^2}u = k(x)u^{2^*-1},\\
&u\in D^{1,2}(\R^N),\ \ u>0 \text{ in } \R^N\setminus\{0\};
\end{aligned}
\right.
\end{equation}
here $D^{1,2}(\R^N):=\{u\in L^{2^*}(\R^N)|~|\nabla u|\in L^2(\R^N)\}$. In \cite{FelliTer05CCM} the existence of positive bubble tower solutions to \eqref{Pro-Felli}, blowing up at the origin, was proved as $\epsilon\to0$, when $k(x)=1+\eps K(x)$ with $K(x)$ a continuous bounded function. These solutions, called fountain-like in \cite{FelliTer05CCM}, are superpositions of positive bubbles. In \cite{FelliPis06CPDE} the existence of a positive solution to \eqref{Pro-Felli} blowing up at a critical point of $k(x)$ was obtained as $\mu\to0^+$. In \cite{CaoPeng06AdM} Cao and Peng investigated the asymptotic behavior of positive solutions to \eqref{pro} in a ball.

In this paper we are interested in the existence of multi-bubble nodal solutions to problem \eqref{pro} as $\eps,\mu\to0$. Compared with \cite{FelliPis06CPDE, FelliTer05CCM} the location of the bubbles does not depend on the shape of a coefficient function $k(x)$ but on the subtle influence of the geometry of the domain. We obtain two types of solutions depending on the exponent $\al$ in the relation $\mu=\mu_0\eps^\al$ where $\mu_0>0$. If $\al>\frac{N-4}{N-2}$ we prove the existence of nodal solutions blowing up at different points, positively at the origin and negatively at other points. The bubble tower solutions exist for $\al=1$, that is, when $\mu$ has the same order as $\eps$ when $\eps\to0^+$. In that case, on any smooth bounded domain, we obtain nodal solutions that are superpositons of bubbles with different signs, all blowing up at the origin. The proofs are based on the Lyapunov-Schmidt reduction scheme.

In order to state our results we introduce some notation. By Hardy's inequality, the norm
\[%begin{equation}\label{norm}
\|u\|_\mu:=\left(\int_{\Om}(|\nabla u|^2-\mu\frac{u^2}{|x|^2})dx\right)^{\frac12}
\]%end{equation}
is equivalent to the norm $\|u\|_0=\left(\int_{\Om}|\nabla u|^2dx\right)^{1/2}$ on $H_0^1(\Om)$ provided $0\le\mu<\ov{\mu}$. This will of course be the case for $\mu=\mu_0\eps^\al$ with $\eps>0$ small. As in \cite{FG} we write $H_\mu(\Om)$ for the Hilbert space consisting of $H^1_0(\Omega)$ functions with the inner product
\[%begin{equation}\label{inner product}
(u,v):=\int_{\Om}\left(\nabla u\nabla v-\mu\frac{uv}{|x|^2}\right)dx.
\]%end{equation}
It is known that the nonzero critical points of the energy functional
\[%begin{eqnarray}\label{energy functional}
J_\eps(u)
 := \frac12\int_{\Om}\left(|\nabla u|^2-\mu\frac{u^2}{|x|^2}\right)dx
     - \frac{1}{2^*-\eps}\int_{\Om}|u|^{2^*-\eps}dx
\]%end{eqnarray}
defined on $H_\mu(\Om)$ are precisely the nontrivial weak solutions to problem \eqref{pro}.

Next we introduce two limiting problems. The first one is
\begin{equation}\label{limit-pro1}
\left\{
\begin{aligned}
-\De u &= |u|^{2^{\ast}-2}u &&\quad \text{in } \R^N, \\\
u &\to 0 &&\quad \text{as } |x|\to\infty.
\end{aligned}
\right.
\end{equation}
It is well known that the nontrivial least energy (positive) solutions to $(\ref{limit-pro1})$ are the instantons
\[
U_{\de,\ze} := C_0\left(\frac{\de}{\de^2+|x-\ze|^2}\right)^{\frac{N-2}{2}}
\]
with $\de>0$, $\ze\in \R^N$ and $C_0:=(N(N-2))^{\frac{N-2}{4}}$, cf.~\cite{Aub76,Tal76}. These
solutions minimize
\[
S_0
 := \min_{u\in D^{1,2}(\R^N)\setminus \{0\}}
         \frac{\int_{\R^N} |\nabla u|^2 dx}{(\int_{\R^N}|u|^{2^*} dx)^{2/{2^*}}}.
\]
Moreover there holds
$$
\int_{\R^N} |\nabla U_{\de,\ze}|^2dx = \int_{\R^N} |U_{\de,\ze}|^{2^*}dx = S_0^{\frac{N}{2}}.
$$

The second limiting problem is
\begin{equation}\label{limit-pro2}
\left\{
\begin{aligned}
-\De u-\mu\frac{u}{|x|^{2}} &= |u|^{2^{\ast}-2}u &&\quad \text{in } \R^N, \\
u &\to 0 &&\quad \text{as } |x|\to\infty.
\end{aligned}
\right.
\end{equation}
For $0<\mu<\ov{\mu}$ we know from \cite{CatW00, Terr96} that all positive solutions to
\eqref{limit-pro2} are given by
\[
V_\si = C_\mu\left(\frac{\si}{\si^2|x|^{\be_1}+|x|^{\be_2}}\right)^{\frac{N-2}{2}}
\]
with $\si>0$, $\be_1:=(\sqrt{\ov{\mu}}-\sqrt{\ov{\mu}-\mu})/\sqrt{\ov{\mu}}$, $\be_2:=(\sqrt{\ov{\mu}}+\sqrt{\ov{\mu}-\mu})/\sqrt{\ov{\mu}}$, and
$C_\mu:=\left(\frac{4N(\ov{\mu}-\mu)}{N-2}\right)^{\frac{N-2}{4}}$.
%Of course, translations of $V_\si$ are also solutions.
These solutions are minimizers of
\[
S_\mu
 := \min_{u\in D^{1,2}(\R^N)\setminus \{0\}}
     \frac{\int_{\R^N}(|\nabla u|^2-\mu\frac{u^2}{|x|^2})dx}
          {(\int_{\R^N}|u|^{2^*}dx)^{2/{2^*}}},
\]
and they satisfy
$$
\int_{\R^N} \left(|\nabla V_\si|^2-\mu\frac{|V_\si|^2}{|x|^2}\right)dx
 = \int_{\R^N} |V_\si|^{2^*}dx = S_\mu^{\frac{N}{2}}.
$$

The Green's function of the Dirichlet Laplacian can be written as
$G(x,y)=\frac{1}{|x-y|^{N-2}}-H(x,y)$, for $x,y\in\Om$, where $H$ is the regular part. These functions are symmetric: $G(x,y)=G(y,x)$ and $H(x,y)=H(y,x)$. We need the map
\[
\vphi(x):=H^{\frac12}(0,0)H^{\frac12}(x,x)+G(x,0).
\]
If the domain satisfies the symmetry condition
\begin{itemize}
\item[$(S_1)$] $\Om$ is invariant under the reflection $(x_1,x')\mapsto(x_1,-x')$, where $x_1\in\R$, $x'\in\R^{N-1}$,
\end{itemize}
then we define $a<0<b$ by $I:=\{(t,0,\ldots,0): a<t<b\}\subset\Om$ and $\pa I=\{(a,0,\ldots,0), (b,0,\ldots,0)\}\subset\pa\Om$. For $a\le s,t\le b$ we set
\[
g(t,s):=G((t,0,\ldots,0),(s,0,\ldots,0)),\quad h(t,s):=H((t,0,\ldots,0),(s,0,\ldots,0)),
\]
and
\[
\vphi(t):=\vphi(t,0,\ldots,0)=h^{\frac12}(0,0)h^{\frac12}(t,t)+g(t,0).
\]
Observe that $\vphi(t)\to\infty$ as $t\to a$ or $t\to b$. In particular, $\vphi$ has a minimum in $(a,b)$.

Now we state the main results of the paper. Throughout the paper let $\Om\subset\R^N$, $N\ge7$, be a smooth bounded domain. We begin with the existence of nodal solutions with bubbles concentrating at different points, one being the origin.

\begin{theorem}\label{theorem-2 bubbles}
Let $\mu=\mu_0\eps^\al$ with $\mu_0>0$ and $\al>\frac{N-4}{N-2}$ fixed.
\begin{itemize}
\item[a)] There exists $\eps_0>0$ such that for any $\eps\in(0,\eps_0)$, there exist a pair of solutions $\pm u_\eps$ to problem \eqref{pro} satisfying
 \begin{equation}\label{shape of solution-2 bubbles}
  u_\eps(x)
    = C_\mu\left(\frac{\si^\eps}{(\si^\eps)^2|x|^{\be_1}+|x|^{\be_2}}\right)^{\frac{N-2}{2}}-C_0\left(\frac{\de^\eps}{(\de^\eps)^2+|x-\xi^\eps|^2}\right)^{\frac{N-2}{2}}
   + o(1),
 \end{equation}
    where $\de^\eps=\la^\eps\eps^{\frac{1}{N-2}}$, $\si^\eps=\ov{\la}^\eps\eps^{\frac{1}{N-2}}$, and for some $\eta>0$ small enough, $|\xi^\eps|>\eta$, $\dist(\xi^\eps,\pa\Om)>\eta$,
    $\la^\eps$, $\ov{\la}^\eps\in(\eta,\frac1\eta)$. Moreover, $\xi^\eps\to \xi^*$ with
    $\vphi(\xi^*)=\min\limits_{\Om}\vphi$.
\item[b)] If $(S_1)$ holds there exist solutions as in a) with $\xi^\eps=(t^\eps,0,\cdots,0)$ and $t^\eps\to t^*$ as $\eps\to0$ where $t^*$ is a (local) minimum of $\vphi(t)$ in $(a,b)\setminus\{0\}$.
\end{itemize}
\end{theorem}

We can obtain more solutions if $\Om=B(0,1)$. More precisely, for $k=2,3$ we obtain the existence of solutions with $k+1$ bubbles, one positive and $k$ negative.

\begin{theorem}\label{theorem-ball-k bubbles}
Let $\Om=B(0,1)\subset\R^N$, $\mu=\mu_0\eps^\al$ with $\mu_0>0$ and $\al>\frac{N-4}{N-2}$ fixed. For $k=2$ and $N\ge7$, or $k=3$ and $N$ large enough, there exists $\eps_0>0$ such that for every $\eps\in(0,\eps_0)$, there exist two pairs of solutions $\pm u^1_\eps$,
$\pm u^2_\eps$ of problem \eqref{pro} satisfying
\begin{equation}\label{ball-shape of solution-k bubbles}
u^j_\eps(x)
 = C_\mu\left(\frac{\si_j^\eps}{(\si_j^\eps)^2|x|^{\be_1}
       + |x|^{\be_2}}\right)^{\frac{N-2}{2}}
    - C_0\sum_{i=1}^k\left(\frac{\de_j^\eps}{(\de_j^\eps)^2
       + |x-(\xi^\eps_j)_i|^2}\right)^{\frac{N-2}{2}}
    + o(1),
\end{equation}
where $\de_j^\eps = \la_j^\eps\eps^{\frac{1}{N-2}}$,
$(\xi^\eps_j)_i = (e^{2\pi i\sqrt{-1}/k}\wt{\xi}^\eps_j,0)$,
$\wt{\xi}^\eps_j \in B^{(2)} := \{x=(x_1,x_2,0,\cdots,0)\in B(0,1)\}$,
$\si_j^\eps = \ov{\la}_j^\eps\eps^{\frac{1}{N-2}}$,
and for some $\eta>0$ small enough, $\eta<|\wt{\xi}^\eps_j|<1-\eta$, $\la_j^\eps$,
$\ov{\la}_j^\eps\in(\eta,\frac{1}{\eta})$, $i=1,2,\dots,k$, $j=1,2$.
\end{theorem}

We would like to point out that the idea of Theorem~\ref{theorem-ball-k bubbles} is not applicable for $k=4$; see Remark~\ref{rem:nonexist} and Proposition~\ref{k=5-impossible-proposition}. It seems very difficult to prove the existence of nodal solutions with the shape as in \eqref{ball-shape of solution-k bubbles} to problem \eqref{pro} for $k\ge4$. However, for $\Om=B(0,1)$, the existence of solutions with $5$ bubbles, $3$ being positive and $2$ being negative, can be proved.

\begin{theorem}\label{theorem-ball-k bubbles-add}
Let $\Om=B(0,1)\subset\R^N$, $N\ge7$, $\mu=\mu_0\eps^\al$ with $\mu_0>0$ and $\al>\frac{N-4}{N-2}$ fixed. Then there exists $\eps_0>0$ such that for any $\eps\in(0,\eps_0)$, there exist one pairs of $5$-bubble solutions $\pm u_\eps$ to problem \eqref{pro} satisfying
\[%begin{equation}%\label{ball-shape of solution-k bubbles-add}
u_\eps(x)
 = C_\mu\left(\frac{\si^\eps}{(\si^\eps)^2|x|^{\be_1}+|x|^{\be_2}}\right)^{\frac{N-2}{2}}
    + C_0\sum_i^{4}(-1)^i\left(\frac{\de^{i,\eps}}{(\de^{i,\eps})^2
       +|x-(\xi^\eps)_i|^2}\right)^{\frac{N-2}{2}}
    + o(1),
\]%end{equation}
where $\de^{i,\eps} = \la^{i,\eps}\eps^{\frac{1}{N-2}}$, $\la^{1,\eps} = \la^{3,\eps}$, $\la^{2,\eps} = \la^{4,\eps}$, $(\xi^\eps)_i = (e^{2\pi i\sqrt{-1}/k}\wt{\xi}^\eps,0)$, $\wt{\xi}^\eps \in B^{(2)}$, $\si^\eps = \ov{\la}^\eps\eps^{\frac{1}{N-2}}$,
and for some $\eta>0$ small enough, $\eta<|\wt{\xi}^\eps|<1-\eta$, $\la^{i,\eps},\ov{\la}^\eps\in(\eta,\frac{1}{\eta})$, $i=1,\dots,4$.
\end{theorem}

The assumption that $\al>\frac{N-4}{N-2}$ is critical to obtain the existence of nodal solutions with multiple bubbles concentrating at different points since it can be seen from the reduction procedure in Section~\ref{Section 3} that the reduced function has no critical point if $0\le\al\le\frac{N-4}{N-2}$. It is natural to ask, whether there exists $\wt{\mu}(\eps)>0$ for $\eps>0$ small, such that problem \eqref{pro} admits positive or nodal solutions with multiple bubbles concentrating at different points when $0\le\mu\le\wt{\mu}(\eps)$.

Now we state a result about the existence of nodal solutions that are towers of bubbles concentrating at the origin.

\begin{theorem}\label{theorem-tower of bubble}
Let $\mu=\mu_0\eps$ with $\mu_0>0$ fixed. For any given integer $k\ge0$ there exists $\eps_0>0$ such that for any $\eps\in(0,\eps_0)$, there exist a pair of solutions $\pm u_\eps$ to problem \eqref{pro} satisfying that, as $\eps\to0^+$,
\begin{eqnarray*}%\label{shape of solution-tower of bubble}
u_\eps(x)
 = C_\mu(-1)^k\left(\frac{\si^\eps}{(\si^\eps)^2|x|^{\be_1}
     + |x|^{\be_2}}\right)^{\frac{N-2}{2}}
   + C_0\sum_{i=1}^{k}(-1)^{i-1}\left(\frac{\de_i^\eps}{(\de_i^\eps)^2
     +|x-\xi_i^\eps|^2}\right)^{\frac{N-2}{2}}
   + o(1)
\end{eqnarray*}
where $\de_i^\eps = \la_i^\eps\eps^{\frac{2i-1}{N-2}}$, $\xi_i^\eps = \de_i^\eps\ze_i^\eps$, $\ze_i^\eps \in \R^N$, $i=1,2,\dots,k$,
$\si^\eps = \ov{\la}^\eps\eps^{\frac{2(k+1)-1}{N-2}}$, and for some $\eta>0$ small enough,
$\la_i^\eps,\ov{\la}^\eps\in\left(\eta,\frac{1}{\eta}\right)$ and $|\ze_i^\eps|\le\frac{1}{\eta}$ for $i=1,\ldots,k$.
\end{theorem}

We assume $N\ge7$ in this paper for technical reasons and in order to not to make the presentation too heavy. The results can be extended to the case $N=6$. For $N\le5$, there would be technical difficulties.

The paper is organized as follows. In Section~2, we give some notations and preliminary results. Section~3 is devoted to the proofs of Theorems~\ref{theorem-2 bubbles} and \ref{theorem-ball-k bubbles}, that is, the existence of nodal solutions with multiple bubbles blowing up at different points. The proof of Theorem~\ref{theorem-tower of bubble}, the existence of nodal bubble tower solutions, is given in Section 4. At last, some useful technical lemmas are collected in the appendices.

%%---------------------------------------------------------------------------------
\section{\textbf{Notations and preliminary results}\label{Section 2}}

Throughout this paper, positive constants will be denoted by $C, c$.

As in \cite{FelliPis06CPDE} let $\iota^*:L^{2N/(N+2)}(\Om)\to H_\mu(\Om)$ be the adjoint operator of the inclusion $\iota:H_\mu(\Om)\to L^{2N/(N-2)}(\Om)$, that is,
\begin{equation}\label{adjoint operator}
\iota^*(u) = v\qquad\Longleftrightarrow\qquad
 (v,\phi)=\int_{\Om}u(x)\phi(x)dx,\quad\text{for all }\phi\in H_\mu(\Om).
\end{equation}
This is continuous, i.e., there exists $c>0$ such that
\begin{equation}\label{ineq-adjoint operator}
\|\iota^*(u)\|_{\mu}\le c\|u\|_{2N/(N+2)}.
\end{equation}
Then problem \eqref{pro} is equivalent to the fixed point problem
\[%\begin{equation}\label{fixed point-pro}
u=\iota^*(f_\eps(u)), u\in H_\mu(\Om),
\]%\end{equation}
where $f_\eps(s)=|s|^{2^*-2-\eps}s$.

To continue, we show an eigenvalue problem first.
\begin{proposition}\label{proposition-eigenvalue}
Let $\La_i$, $i=1,2,\dots$, be the eigenvalues of
\begin{equation}\label{pro-eigenvalue}
\begin{cases}
-\De u-\mu\frac{u}{|x|^2} = \La|V_\si|^{2^{\ast}-2}u &\quad\text{in } \R^N,\\
|u|\to0 &\quad\text{as }|x|\to+\infty
\end{cases}
\end{equation}
in increasing order. Then $\La_1=1$ with eigenfunction $V_\si$, $\La_2=2^*-1$ with eigenfunction $\frac{\pa V_\si}{\pa \si}$.
\end{proposition}

\begin{proof}
Direct computations give that $V_\si$ and $\frac{\pa V_\si}{\pa \si}$ are eigenfunctions corresponding to $1$ and $2^*-1$, respectively. Now as in \cite{WangWill03JFA}, it is enough to prove that the eigenfunction $u$ corresponding to the eigenvalue $\La\le 2^*-1$ has to be radial.

Denote by $\psi_i$, $i=0,1,2,\dots$, the sequence of spherical harmonics, which are eigenfunctions of the Laplace-Beltrami operator on $S^{N-1}$:
$$
-\De_{S^{N-1}}\psi_i=\tau_i\psi_i.
$$
It is well known that $\tau_0=0$, $\tau_1,\dots,\tau_N=N-1$, $\tau_{N+1}>\tau_N$. We prove that for every $i\ge1$,
$$
\int_{S^{N-1}}u(r,\theta)\psi_i(\theta)d\theta=0.
$$

Setting $\vphi_i(r) = \int_{S^{N-1}}u(r,\theta)\psi_i(\theta)d\theta$ we have:
\begin{eqnarray*}
\De\vphi_i&=&\De_r \vphi_i=\int_{S^{N-1}}\De_r u(r,\theta)\psi_i(\theta)d\theta\\
&=& -\int_{S^{N-1}}\frac{\De_\theta u(r,\theta)}{r^2}\psi_i(\theta)d\theta
    - \int_{S^{N-1}}\left(\frac{\mu u(r,\theta)}{r^2}
       + \La V_\si^{2^*-2}u(r,\theta)\right)\psi_i(\theta)d\theta\\
&=& \int_{S^{N-1}}\frac{\tau_i u(r,\theta)}{r^2}\psi_i(\theta)d\theta
    - \int_{S^{N-1}}\left(\frac{\mu}{r^2}
      + \La V_\si^{2^*-2}\right)u(r,\theta)\psi_i(\theta) d\theta\\
&=& \left(\frac{\tau_i}{r^2}-(\frac{\mu}{r^2}+\La V_\si^{2^*-2})\right)\vphi_i(r).
\end{eqnarray*}
This implies for any $R>0$:
\begin{eqnarray*}
0
 &=& \int_{B_R(0)}\De\vphi_i\frac{\pa V_\si}{\pa r}
      +\left(\left(\frac{\mu}{r^2}+\La V_\si^{2^*-2}\right)
         -\frac{\tau_i}{r^2}\right)\vphi_i\frac{\pa V_\si}{\pa r}\\
 &=& \int_{B_R(0)}\vphi_i\De\left(\frac{\pa V_\si}{\pa r}\right)
      + \left(\left(\frac{\mu}{r^2}+\La V_\si^{2^*-2}\right)
          -\frac{\tau_i}{r^2}\right)\vphi_i\frac{\pa V_\si}{\pa r}
      + \int_{\pa B_R(0)}\left(\frac{\pa V_\si}{\pa r}\cdot\frac{\pa \vphi_i}{\pa r}
      - \vphi_i\frac{\pa^2 V_\si}{\pa r^2}\right)\\
&=& \int_{B_R(0)}\frac{N-1}{r^2}\vphi_i\frac{\pa V_\si}{\pa r}
      + \vphi_i\frac{\pa}{\pa r}\left(-\mu\frac{V_\si}{r^2}-V_\si^{2^*-1}\right)
      + \left(\left(\frac{\mu}{r^2}+\La V_\si^{2^*-2}\right)
         - \frac{\tau_i}{r^2}\right)\vphi_i\frac{\pa V_\si}{\pa r}\\
&& +\int_{\pa B_R(0)}\left(\frac{\pa V_\si}{\pa r}\cdot\frac{\pa\vphi_i}{\pa r}
      -\vphi_i\frac{\pa^2 V_\si}{\pa r^2}\right)\\
&=& \int_{B_R(0)}\frac{N-1-\tau_i}{r^2}\vphi_i\frac{\pa V_\si}{\pa r}
     + (\La-(2^*-1))V_\si^{2^*-2}\frac{\pa V_\si}{\pa r}\vphi_i
     + \frac{2\mu V_\si}{r^3}\vphi_i\\
&& +\int_{\pa B_R(0)}\left(\frac{\pa V_\si}{\pa r}\cdot\frac{\pa \vphi_i}{\pa r}
     -\vphi_i\frac{\pa^2 V_\si}{\pa r^2}\right).
\end{eqnarray*}
Now let $R$ be the first zero of $\vphi_i$; $R:=+\infty$ if $\vphi_i$ is never zero. Without loss of generality we assume $\vphi_i(r)>0$ for $r\in(0,R)$. Then
$\frac{\pa\vphi_i}{\pa r}(R)\le0$, and we finish the proof.
\end{proof}

Let us define the projection $P:H^1(\R^N)\to H_0^1(\Om)$, that is, $\De Pu=\De u$ in $\Om$, $Pu=0$ on $\pa\Om$.

\begin{proposition}\label{proposition-projection estimate}
Let $0\in\Om$ be a smooth bounded domain. Denote $\vphi_\si:=V_\si-PV_\si$. Then
\begin{equation}\label{equ-add1-projection estimate}
0\le\vphi_\si\le V_\si,\quad\text{where }
\vphi_\si(x)
 = C_\mu\ov{d}^{\sqrt{\ov{\mu}}-\sqrt{\ov{\mu}-\mu}}(x)H(0,x)\si^{\frac{N-2}{2}}+\hbar_\si;
\end{equation}
here $d_{\inf}\le \ov{d}\le d_{\sup}$, $d_{\inf}=\dist(0,\pa\Om)=\inf\{|x|:x\in\pa\Om\}$, $d_{\sup}=\sup\{|x|:x\in\pa\Om\}$, and $\hbar_\si$ satisfies the uniform estimates
\begin{equation}\label{equ-2-projection estimate}
\hbar_\si=O(\si^{\frac{N+2}{2}}),\quad\frac{\pa \hbar_\si}{\pa \si}=O(\si^{\frac{N}{2}}).
\end{equation}
\end{proposition}

\begin{proof}
It is easy to see that $\vphi_\si$ satisfies
\[
\begin{cases}
\De \vphi_\si(x)=0 \quad& \text{in} ~\Om\backslash\{0\},\\
\vphi_\si(x)=V_\si(x)=C_\mu(\frac{\si}{\si^2
|x|^{\be_1}
+|x|^{\be_2}})^{\frac{N-2}{2}}\quad&  \text{on}~\pa \Om.
\end{cases}
\]
Then the first part of (\ref{equ-add1-projection estimate}) holds by the maximum principle.

Consider the function $H$ satisfying
\[
\begin{cases}
\De H(0,x)=0 \quad& \text{in} ~\Om\backslash\{0\},\\
H(0,x)=\frac{1}{|x|^{N-2}} \quad& \text{on }\pa\Om.
\end{cases}
\]
Notice that on $\pa \Om$,
\[
\vphi_\si-C_\mu d_{\inf}^{\sqrt{\ov{\mu}}-\sqrt{\ov{\mu}-\mu}}H(0,x)\si^{\frac{N-2}{2}}
 = C_\mu \si^{\frac{N-2}{2}}[\frac{1}{(\si^2|x|^{\be_1}+|x|^{\be_2})^{\frac{N-2}{2}}}
   -\frac{d_{\inf}^{\sqrt{\ov{\mu}}-\sqrt{\ov{\mu}-\mu}}}{|x|^{N-2}}]
 \ge O(\si^{\frac{N+2}{2}}),
\]
and
\[
\vphi_\si-C_\mu d_{\sup}^{\sqrt{\ov{\mu}}-\sqrt{\ov{\mu}-\mu}}H(0,x)\si^{\frac{N-2}{2}}
 = C_\mu \si^{\frac{N-2}{2}}[\frac{1}{(\si^2|x|^{\be_1}+|x|^{\be_2})^{\frac{N-2}{2}}}
   -\frac{d_{\sup}^{\sqrt{\ov{\mu}}-\sqrt{\ov{\mu}-\mu}}}{|x|^{N-2}}]
 \le O(\si^{\frac{N+2}{2}}).
\]
Then the maximum principle and direct computations yield the second part of \eqref{equ-add1-projection estimate} and \eqref{equ-2-projection estimate}.
\end{proof}

\begin{remark}
\begin{itemize}
\item[a)] If $\mu\to0^+,$ then
\begin{equation}\label{equ-1-projection estimate}
\vphi_\si(x) = C_0H(0,x)\si^{\frac{N-2}{2}}+O(\mu\si^{\frac{N-2}{2}})+\hbar_\si.
\end{equation}
\item[b)] Let us recall the similar results for $U_{\de,\xi}$ obtained in \cite{Rey90JFA}, that is
\begin{equation}\label{equ-3-projection estimate}
0 \le \vphi_{\de,\xi}:=U_{\de,\xi}-PU_{\de,\xi}\le U_{\de,\xi},~\vphi_{\de,\xi}
  = C_0 H(\xi,\cdot)\de^{\frac{N-2}{2}}+\hbar_{\de,\xi},
\end{equation}
where $\hbar_{\de,\xi}=O(\de^{\frac{N+2}{2}})$.
\end{itemize}
\end{remark}

%%---------------------------------------------------------------------------------
\section{\textbf{Solutions with multiple bubbles concentrating at different points \label{Section 3}}}

%%---------------------------------------------------------------------------------
\subsection{\textbf{The finite dimensional reduction}\label{Subsection 3.1}}

We introduce some notation. Fix an integer $k\ge0$. For $\la=(\la_1,\la_2,\dots,\la_{k},\ov{\la})\in \R_+^{k+1}$ and $\xi=(\xi_1,\xi_2,\dots,\xi_{k})\in\Om^{k}$ we define
\[%begin{equation}%\label{subspace-W-1}
W_{\eps,\la,\xi}
 :=\sum_{i=1}^{k}\Ker\left(-\De-(2^*-1)U_{\de_i,\xi_i}^{2^*-2}\right)
     + \Ker\left(-\De-\frac{\mu}{|x|^2}-(2^*-1)V_\si^{2^*-2}\right),
\]%end{equation}
where $\de_i=\la_i\eps^{\frac{1}{N-2}}$, $\si=\ov{\la}\eps^{\frac{1}{N-2}}$. From \cite{BianEg91JFA}, the kernel of the operator $-\De-(2^*-1)U_{\de_i,\xi_i}^{2^*-2}$ on $L^2(\R^N)$ has dimension $N+1$ and is spanned by $\frac{\pa U_{\de_i,\xi_i}}{\pa \de_i}$, $\frac{\pa U_{\de_i,\xi_i}}{\pa (\xi_i)^j}$, $j=1,2,\dots,N$, where $(\xi_i)^j$ is the $j-$th component of $\xi_i$. Combining this with Proposition~\ref{proposition-eigenvalue}, we have
\[
W_{\eps,\la,\xi}
 = \span\left\{\Psi_i^j,\ \Psi_i^0,\ \ov{\Psi},\ i=1,2,\dots,k,\ j=1,2,\dots,N\right\},
\]
where for $i=1,2,\dots,k$ and $j=1,2,\dots,N$:
$$
\Psi_i^j:=\frac{\pa U_{\de_i,\xi_i}}{\pa (\xi_i)^j},\quad
\Psi_i^0:=\frac{\pa U_{\de_i,\xi_i}}{\pa \de_i},\quad
\ov{\Psi}:=\frac{\pa V_{\si}}{\pa \si}.
$$
For $\eta\in(0,1)$ we define
\[%begin{equation}%\label{subspace-O-1}
\begin{aligned}
\cO_\eta
 &:=\big\{(\la,\xi)\in\R_+^{k+1}\times \Om^{k}: \la_i\in(\eta,\eta^{-1}),\ov{\la}\in(\eta,\eta^{-1}),\ \dist(\xi_i,\pa\Om)>\eta,\\
 &\hspace{2cm}
 |\xi_i|>\eta,\ |\xi_{i_1}-\xi_{i_2}|>\eta,\ i,i_1,i_2=1,2,\dots,k,\ i_1\neq i_2\big\}.
\end{aligned}
\]%end{equation}

Let us introduce the spaces
\[
K_{\eps,\la,\xi}:=P W_{\eps,\la,\xi},
\]
and
\[
K^\bot_{\eps,\la,\xi}
 :=\{\phi\in H_\mu(\Om):(\phi,P\Psi)=0,\text{ for all }\Psi\in W_{\eps,\la,\xi}\},
\]
and the $(\cdot,\cdot)$-orthogonal projections
\[
\Pi_{\eps,\la,\xi} := H_\mu(\Om)\to K_{\eps,\la,\xi},
\]
and
\[
\Pi^\bot_{\eps,\la,\xi} := Id-\Pi_{\eps,\la,\xi}:H_\mu(\Om)\to K^\bot_{\eps,\la,\xi}.
\]

Solving problem \eqref{pro} is equivalent to finding $\eta>0$, $\eps>0$, $(\la,\xi)\in\cO_\eta$ and $\phi_{\eps,\la,\xi}\in K^\bot_{\eps,\la,\xi}$ such that:
\begin{equation}\label{main-equality Com-1}
\Pi^\bot_{\eps,\la,\xi}\left(V_{\eps,\la,\xi} + \phi_{\eps,\la,\xi}
 - \iota^*(f_\eps(V_{\eps,\la,\xi} + \phi_{\eps,\la,\xi}))\right)
= 0,
\end{equation}
and
\begin{equation} \label{main-equality-1}
\Pi_{\eps,\la,\xi}\left(V_{\eps,\la,\xi} + \phi_{\eps,\la,\xi} - \iota^*(f_\eps(V_{\eps,\la,\xi}
 + \phi_{\eps,\la,\xi}))\right)
= 0,
\end{equation}
where
\begin{equation}\label{shape of V-1}
V_{\eps,\la,\xi}=-\sum\limits_{i=1}^{k}PU_{\de_i,\xi_i}+PV_\si
\end{equation}
or
\begin{equation}\label{shape of V-add1}
V_{\eps,\la,\xi}=\sum\limits_{i=1}^{k}(-1)^iPU_{\de_i,\xi_i}+PV_\si.
\end{equation}
In the rest of this section, we only consider $V_{\eps,\la,\xi}$ as in \eqref{shape of V-1} because the argument for the solutions of the form \eqref{shape of V-add1} is similar.

We prove \eqref{main-equality Com-1} first. Let us introduce the operator $L_{\eps,\la,\xi}:K^\bot_{\eps,\la,\xi}\to K^\bot_{\eps,\la,\xi}$ defined by
\[%begin{equation}%\label{definition-L-1}
L_{\eps,\la,\xi}(\phi)=\phi-\Pi^\bot_{\eps,\la,\xi}\iota^*(f'_0(V_{\eps,\la,\xi})\phi).
\]%end{equation}

\begin{proposition}\label{proposition-operator L-1}
For any $\eta>0$, there exist $\eps_0>0$ and $c>0$ such that for every $(\la,\xi)\in\cO_\eta$ and for every $\eps\in(0,\eps_0)$:
\[%begin{equation}\label{inequality operator L-1}
\|L_{\eps,\la,\xi}(\phi)\|_{\mu}\ge c\|\phi\|_\mu, \quad\text{for all }\phi\in K^\bot_{\eps,\la,\xi}.
\]%end{equation}
In particular, $L_{\eps,\la,\xi}$ is invertible with continuous inverse.
\end{proposition}

\begin{proof}
We argue by contradiction, following the same line as in \cite{MussoPis02IUMJ}. Suppose there exist $\eta>0$, sequences $\eps^n>0$, $(\la^n,\xi^n)\in\cO_\eta$, $\phi^n\in H_\mu(\Om)$ such that $\eps^n\to0$,
$\la^n=(\la_1^n,\dots,\la_k^n,\ov{\la}^n) \to (\la_1,\dots\la_k,\ov{\la})$, $\xi^n=(\xi_1^n,\dots,\xi_k^n) \to (\xi_1,\dots\xi_k)$, as $n\to\infty$, and
\begin{equation}\label{equ-1-proposition-operator L-1}
\phi^n\in K^\bot_{\eps^n,\la^n,\xi^n},\|\phi^n\|_\mu=1,
\end{equation}
\begin{equation}\label{equ-2-proposition-operator L-1}
L_{\eps^n,\la^n,\xi^n}(\phi^n)=h^n, \text{with}~\|h^n\|_\mu\to 0.
\end{equation}
Thus we have
\begin{equation}\label{equ-3-proposition-operator L-1}
\phi^n-\iota^*(f'_0(V_{\eps^n,\la^n,\xi^n})\phi^n)
 = h^n-\Pi_{\eps^n,\la^n,\xi^n}(\iota^*(f'_0(V_{\eps^n,\la^n,\xi^n})\phi^n)).
\end{equation}
Setting
\[
\de_i^n = \la_i^n\eps^{\frac{1}{N-2}},\quad \si^n = \ov{\la}^n\eps^{\frac{1}{N-2}}
\]
and
\[
(\Psi_i^j)_n := \frac{\pa U_{\de_i^n,\xi_i^n}}{\pa (\xi_i^n)^j}\text{ for } j=1,2,\dots,N, \quad
(\Psi_i^0)_n:=\frac{\pa U_{\de_i^n,\xi_i^n}}{\pa \de_i^n},\quad
(\ov{\Psi})_n:=\frac{\pa V_{\si^n}}{\pa \si^n},
\]
where $(\xi_i^n)^j$ is the $j-$th component of $\xi_i^n$, we obtain
\[
w^n := -\Pi_{\eps^n,\la^n,\xi^n}(\iota^*(f'_0(V_{\eps^n,\la^n,\xi^n})\phi^n))
 = \sum\limits_{i=1}^{k}\sum_{j=0}^N c_{i,j}^nP(\Psi_i^j)_n+c_0^n P(\ov{\Psi})_n
\]
for some coefficients $c_{i,j}^n, c_0^n$. Now we argue in three steps.

\emph{Step 1.} We prove
\begin{equation}\label{equ-wn-proposition-operator L-1}
 \lim_{n\to\infty}\|w^n\|_\mu=0.
\end{equation}

Multiplying (\ref{equ-3-proposition-operator L-1}) by $\De P(\Psi_l^h)_n+\mu\frac{P(\Psi_l^h)_n}{|x|^2}$, we get
\[%\label{equ-4-proposition-operator L-1}
\begin{aligned}
&\int_\Om \phi^n\left(\De P(\Psi_l^h)_n+\mu\frac{P(\Psi_l^h)_n}{|x|^2}\right)
  - \int_\Om\iota^*(f'_0(V_{\eps^n,\la^n,\xi^n})\phi^n)
      \left(\De P(\Psi_l^h)_n + \mu\frac{P(\Psi_l^h)_n}{|x|^2}\right)\\
&\hspace{1cm}
= -\int_\Om h^n\left(-\De P(\Psi_l^h)_n-\mu\frac{P(\Psi_l^h)_n}{|x|^2}\right)
      + \int_\Om w^n\left(\De P(\Psi_l^h)_n+\mu\frac{P(\Psi_l^h)_n}{|x|^2}\right)
\end{aligned}
\]
and then
\[
\begin{aligned}
&\sum_{i=1}^{k}\sum_{j=0}^N c_{i,j}^n\big(P(\Psi_i^j)_n,P(\Psi_l^h)_n\big)
       + c_{0}^n\big(P(\ov{\Psi})_n,P(\Psi_l^h)_n\big)\\
&\hspace{1cm}
=\big(\phi^n,P(\Psi_l^h)_n\big)
    - \big(\iota^*(f'_0(V_{\eps^n,\la^n,\xi^n})\phi^n),P(\Psi_l^h)_n\big)
    - \big(h^n,P(\Psi_l^h)_n\big).
\end{aligned}
\]
From Lemma \ref{e56-e62} we deduce:
\begin{equation}\label{equ-4-proposition-operator L-1a}
c_{l,h}^n\wt{c}_{l,h}^n\frac{1}{(\de_l^n)^2}+o\left(\frac{1}{(\de_l^n)^2}\right)
 = -\big(\iota^*(f'_0(V_{\eps^n,\la^n,\xi^n})\phi^n),P(\Psi_l^h)_n\big),
\end{equation}
where $\wt{c}_{l,h}^n>0$ is a constant.

Proposition \ref{proposition-projection estimate} implies
\begin{eqnarray*}
0&=&(\phi^n,P(\Psi_l^h)_n)=\int_\Om\nabla\phi^n\nabla P(\Psi_l^h)_n-\mu\frac{\phi^nP(\Psi_l^h)_n}{|x|^2}\\
&=&\int_\Om\nabla\phi^n\nabla (\Psi_l^h)_n-\mu\frac{\phi^n(\Psi_l^h)_n}{|x|^2}+o(1)\\
&=&\int_\Om f'_0(U_{\de_l^n,\xi_l^n})(\Psi_l^h)_n\phi^n+o(1),
\end{eqnarray*}
and then
\[
\begin{aligned}
&-(\iota^*(f'_0(V_{\eps^n,\la^n,\xi^n})\phi^n),P(\Psi_l^h)_n)
  = -\int_\Om f'_0(V_{\eps^n,\la^n,\xi^n})\phi^nP(\Psi_l^h)_n\\
&\hspace{1cm}
 \le \left|\int_\Om \big(f'_0(V_{\eps^n,\la^n,\xi^n})
         - f'_0(U_{\de_l^n,\xi_l^n})\big)\phi^n(\Psi_l^h)_n\right|
 + \left|\int_\Om f'_0(V_{\eps^n,\la^n,\xi^n})\phi^n\big(P(\Psi_l^h)_n
         - (\Psi_l^h)_n\big)\right|
 + o(1)\\
&\hspace{1cm}
 = o(1)
\end{aligned}
\]
by Lemma \ref{e50-e51} and Lemma \ref{e83-85-Lemma A}.

Combining the above inequality with \eqref{equ-4-proposition-operator L-1a} yields $c_{l,h}^n\to0$ as $n\to\infty$. Similar arguments show that $c_0^n\to0$ as $n\to\infty$, and  $\lim\limits_{n\to\infty}\|w^n\|_\mu=0$ follows.

\emph{Step 2.} Let $\chi:\R^N\to [0,1]$ be a smooth cut-off function, such that $\chi(x)=1$ if $|x|\le\eta/4$, $\chi(x)=0$ if $|x|\ge\eta/2$, and $|\nabla\chi(x)| \le \frac{C}{\eta}$. We set
\[
\phi_i^n(x)
 := \big((\eps^n)^{\al_1}\big)^{\frac{N-2}{2}}\phi^n\big((\eps^n)^{\al_1}x+\xi_i^n\big)
            \chi\big((\eps^n)^{\al_1}x\big),
     \quad x\in\Om_i^n:=\frac{\Om-\xi_i^n}{(\eps^n)^{\al_1}},\ i=1,\dots,k,
\]
and
\[
\phi_0^n(x)
 := \big((\eps^n)^{\al_2}\big)^{\frac{N-2}{2}}\phi^n\big((\eps^n)^{\al_2}x\big)
           \chi\big((\eps^n)^{\al_2}x\big),
     \quad x\in\Om_0^n:=\frac{\Om}{(\eps^n)^{\al_2}},
\]
where $\al_1,\al_2$ are positive constants which will be determined later. Since $\phi_i^n$ is bounded in $D^{1,2}(\R^N)$, we may assume, up to a subsequence,
\[%begin{equation}\label{equ-5-proposition-operator L-1}
\phi_i^n\weakto \phi_i^\infty \quad \text{weakly in } D^{1,2}(\R^N),\ i=0,1,2,\dots,k.
\]%end{equation}
Now we claim that
\begin{equation}\label{equ-6-proposition-operator L-1}
\phi_i^\infty(x)=0,\quad i=0,1,\dots,k.
\end{equation}

Firstly we prove \eqref{equ-6-proposition-operator L-1} for $i=1,\dots,k$. Notice that $\big|\nabla\chi\big((\eps^n)^{\al_1}x\big)\big|
 = \big|(\eps^n)^{\al_1}\nabla\chi(\cdot)\big| \le \frac{C(\eps^n)^{\al_1}}{\eta}=o(1)$.
Thus we have for any $\psi\in C_0^\infty(\R^N)$:
\begin{equation}\label{equ-7-proposition-operator L-1}
\big((\eps^n)^{\al_1}\big)^{\frac{N-2}{2}}
 \int_{\Om_i^n}\nabla\chi\big((\eps^n)^{\al_1}x\big)
  \big(\phi^n((\eps^n)^{\al_1}x+\xi_i^n)\nabla\psi
        - \psi\nabla\phi^n((\eps^n)^{\al_1}x+\xi_i^n)\big)
 = o(1).
\end{equation}
On the other hand, taking $\al_1=\frac{1}{N-2}$ and noticing $N\ge7$, we get:
\begin{equation}\label{equ-8-proposition-operator L-1}
\big((\eps^n)^{\al_1}\big)^{\frac{2-N}{2}}\mu\int_{\Om}
 \frac{\iota^*\big(f'_0(V_{\eps^n,\la^n,\xi^n}(y))\phi^n(y)\big)
        \chi(y-\xi_i^n)\psi\left(\frac{y-\xi_i^n}{(\eps^n)^{\al_1}}\right)}{|y|^2}
 = o(1).
\end{equation}

By \eqref{equ-7-proposition-operator L-1}, \eqref{equ-3-proposition-operator L-1}, \eqref{equ-2-proposition-operator L-1}, \eqref{equ-wn-proposition-operator L-1}, \eqref{equ-8-proposition-operator L-1} and \eqref{equ-3-projection estimate}, we have for any $\psi\in C_0^\infty(\R^N)$:
\begin{equation}\label{equ-9-proposition-operator L-1}
\begin{aligned}
&\int_{\Om_i^n}\nabla\phi_i^n\nabla\psi\\
&\hspace{1cm}
 = \big((\eps^n)^{\al_1}\big)^{\frac{N-2}{2}}
      \int_{\Om_i^n}\Big(\nabla\phi^n\big((\eps^n)^{\al_1}x+\xi_i^n\big)
                    \nabla(\chi((\eps^n)^{\al_1}x)\psi)\\
&\hspace{2cm}
         +\nabla\chi\big((\eps^n)^{\al_1}x\big)
         \big(\phi^n\big((\eps^n)^{\al_1}x+\xi_i^n\big)\nabla\psi
              -\psi\nabla\phi^n((\eps^n)^{\al_1}x+\xi_i^n)\big)\Big)\\
&\hspace{1cm}
 = \big((\eps^n)^{\al_1}\big)^{\frac{N-2}{2}}
      \int_{\Om_i^n}\nabla\phi^n\big((\eps^n)^{\al_1}x+\xi_i^n\big)
            \nabla\big(\chi((\eps^n)^{\al_1}x)\psi\big)+o(1)\\
&\hspace{1cm}
 = \big((\eps^n)^{\al_1}\big)^{\frac{N-2}{2}}
     \int_{\Om_i^n}\nabla\iota^*
         \big(f'_0\big(V_{\eps^n,\la^n,\xi^n}((\eps^n)^{\al_1}x+\xi_i^n)\big)
          \phi^n\big((\eps^n)^{\al_1}x+\xi_i^n\big)\big)
          \nabla\big(\chi((\eps^n)^{\al_1}x)\psi\big)\\
&\hspace{2cm}
      + \big((\eps^n)^{\al_1}\big)^{\frac{N-2}{2}}
         \int_{\Om_i^n}\nabla h_n((\eps^n)^{\al_1}x+\xi_i^n)
               \nabla\big(\chi((\eps^n)^{\al_1}x)\psi\big)\\
&\hspace{2cm}
      + \big((\eps^n)^{\al_1}\big)^{\frac{N-2}{2}}
      \int_{\Om_i^n}\nabla w_n\big((\eps^n)^{\al_1}x+\xi_i^n\big)
              \nabla(\chi((\eps^n)^{\al_1}x)\psi)
      + o(1)\\
&\hspace{1cm}
 = \big((\eps^n)^{\al_1}\big)^{\frac{N-2}{2}}
    \int_{\Om_i^n}\nabla\iota^*
      \big(f'_0\big(V_{\eps^n,\la^n,\xi^n}((\eps^n)^{\al_1}x+\xi_i^n)\big)
          \phi^n\big((\eps^n)^{\al_1}x+\xi_i^n\big)\big)
          \nabla\big(\chi((\eps^n)^{\al_1}x)\psi\big)\\
&\hspace{2cm}
      +o(1)\\
&\hspace{1cm}
 = \big((\eps^n)^{\al_1}\big)^{\frac{2-N}{2}}
    \int_{\Om}\nabla\iota^*\big(f'_0(V_{\eps^n,\la^n,\xi^n}(y))\phi^n(y)\big)
       \nabla\big(\chi(y-\xi_i^n)\psi(\frac{y-\xi_i^n}{(\eps^n)^{\al_1}})\big)
     + o(1)\\
&\hspace{1cm}
 = \big((\eps^n)^{\al_1}\big)^{\frac{2-N}{2}}
    \int_{\Om}f'_0\big(V_{\eps^n,\la^n,\xi^n}(y)\big)\phi^n(y)\chi(y-\xi_i^n)
         \psi\left(\frac{y-\xi_i^n}{(\eps^n)^{\al_1}}\right) + o(1)\\
&\hspace{1cm}
 = \big((\eps^n)^{\al_1}\big)^{\frac{2-N}{2}}
    \int_{|y-\xi_i^n|\le\eta/2}
       f'_0\left(-\sum_{j=1}^{k}PU_{\de_j^n,\xi_j^n}(y)+P V_{\si^n}(y)\right)
           \phi^n(y)\chi(y-\xi_i^n)\psi\left(\frac{y-\xi_i^n}{(\eps^n)^{\al_1}}\right)\\
&\hspace{2cm}
      +o(1)\\
&\hspace{1cm}
 = \big((\eps^n)^{\al_1}\big)^{\frac{2-N}{2}}
     \int_{|y-\xi_i^n|\le\eta/2}
        f'_0\big(U_{\de_i^n,\xi_i^n}(y)\big)\phi^n(y)\chi(y-\xi_i^n)
             \psi(\frac{y-\xi_i^n}{(\eps^n)^{\al_1}})
        + o(1)\\
&\hspace{1cm}
 = \int_{|(\eps^n)^{\al_1}x|\le\eta/2}
     f'_0\big(U_{\la_i^n,0}(x)\big)\big((\eps^n)^{\al_1}\big)^{\frac{N-2}{2}}
      \phi^n\big((\eps^n)^{\al_1}x+\xi_i^n\big)\chi\big((\eps^n)^{\al_1}x\big)\psi(x)
   + o(1)\\
&\hspace{1cm}
 = \int_{\R^N}f'_0(U_{\la_i,0}(x))\phi_i^\infty(x)\psi(x) + o(1).
\end{aligned}
\end{equation}
Therefore $\phi_i^\infty$ is a weak solution of
\begin{equation}\label{equ-10-proposition-operator L-1}
-\De \phi_i^\infty = f'_0(U_{\la_i,0})\phi_i^\infty\quad \text{in } D^{1,2}(\R^N).
\end{equation}

In order to continue we denote $\Psi_{\la_i,0}^j:=\frac{\pa U_{\la_i,0}}{\pa x^j}$ for $j=1,\dots,N$, and $\Psi_{\la_i,0}^0:=\frac{\pa U_{\la_i,0}}{\pa \la_i}$. Now we claim that
\begin{equation}\label{equ-11-proposition-operator L-1}
\int_{\R^N}\nabla\phi_i^\infty(x)\nabla\Psi_{\la_i,0}^j(x)=0,\ j=0,1,\dots, N.
\end{equation}
In fact,
\begin{eqnarray}\label{equ-12-proposition-operator L-1}
%\begin{aligned}
&&\left|\int_{\Om_i^n}f'_0\big(U_{\la_i^n,0}(x)\big)\phi_i^n(x)\Psi_{\la_i^n,0}^j(x)\right|\\
&&\hspace{.3cm}
 = \left|\int_{\Om_i^n}f'_0\big(U_{\la_i^n,0}(x)\big)
    \big((\eps^n)^{\al_1}\big)^{\frac{N-2}{2}}
    \phi^n\big((\eps^n)^{\al_1}x+\xi_i^n\big)\chi((\eps^n)^{\al_1}x)
    \Psi_{\la_i^n,0}^j(x)\right|\nonumber\\
&&\hspace{.3cm}
 = \left|\int_{(\eps^n)^{-\al_1}\Om}
     f'_0\left(U_{\la_i^n,\frac{\xi_i^n}{(\eps^n)^{\al_1}}}(x)\right)
     \big((\eps^n)^{\al_1}\big)^{\frac{N-2}{2}}
     \phi^n\big((\eps^n)^{\al_1}x\big)\chi\big((\eps^n)^{\al_1}x-\xi_i^n\big)
     \Psi_{\la_i^n,0}^j\big(x-\frac{\xi_i^n}{(\eps^n)^{\al_1}}\big)\right|.\nonumber
%\end{aligned}
\end{eqnarray}
Noticing that
\[%begin{eqnarray}\label{equ-13-proposition-operator L-1}
\begin{aligned}
&\int_{(\eps^n)^{-\al_1}\Om}f'_0\left(U_{\la_i^n,\frac{\xi_i^n}{(\eps^n)^{\al_1}}}(x)\right)
        \big((\eps^n)^{\al_1}\big)^{\frac{N-2}{2}}
        \phi^n\big((\eps^n)^{\al_1}x\big)
        \Psi_{\la_i^n,0}^j\big(x-\frac{\xi_i^n}{(\eps^n)^{\al_1}}\big)\\
&\hspace{1cm}
 = -(\eps^n)^{\al_1}\int_{\Om}f'_0\big(U_{\de_i^n,\xi_i^n}(y)\big)\phi^n(y)
     \big(\Psi_i^j\big)_n(y)
 = o(1),
\end{aligned}
\]%end{eqnarray}
then
\[
\begin{aligned}
&\eqref{equ-12-proposition-operator L-1}\\
&\hspace{.3cm}
 = \left|\int_{(\eps^n)^{-\al_1}\Om}
     f'_0\left(U_{\la_i^n,\frac{\xi_i^n}{(\eps^n)^{\al_1}}}(x)\right)
     \big((\eps^n)^{\al_1}\big)^{\frac{N-2}{2}}
     \phi^n\big((\eps^n)^{\al_1}x\big)\big(\chi((\eps^n)^{\al_1}x-\xi_i^n)-1\big)
     \Psi_{\la_i^n,0}^j\left(x-\frac{\xi_i^n}{(\eps^n)^{\al_1}}\right)\right|\\
&\hspace{1.3cm}
    +o(1)\\
&\hspace{.3cm}
 \le \left|\int_{\left|x-\frac{\xi_i^n}{(\eps^n)^{\al_1}}\right|\ge\frac{\eta/4}{(\eps^n)^{\al_1}}}
       f'_0\left(U_{\la_i^n,\frac{\xi_i^n}{(\eps^n)^{\al_1}}}(x)\right)
       \big((\eps^n)^{\al_1}\big)^{\frac{N-2}{2}}
       \phi^n\big((\eps^n)^{\al_1}x\big)
       \Psi_{\la_i^n,0}^j\left(x-\frac{\xi_i^n}{(\eps^n)^{\al_1}}\right)\right|\\
&\hspace{1.3cm}
    +o(1)\\
&\hspace{.3cm}
 \le C\|\phi^n\|_{\frac{2N}{N-2}}
    \left(
    \int_{\left|x-\frac{\xi_i^n}{(\eps^n)^{\al_1}}\right|\ge\frac{\eta/4}{(\eps^n)^{\al_1}}}
     \left(U_{\la_i^n,\frac{\xi_i^n}{(\eps^n)^{\al_1}}}(x)\right)^{\frac{2N}{N-2}}\right)^{\frac2N}\\
&\hspace{1.3cm}
    \times\left(
     \int_{\left|x-\frac{\xi_i^n}{(\eps^n)^{\al_1}}\right|\ge\frac{\eta/4}{(\eps^n)^{\al_1}}}
     \left(\Psi_{\la_i^n,0}^j\left(x-\frac{\xi_i^n}{(\eps^n)^{\al_1}}\right)
      \right)^{\frac{2N}{N-2}}\right)^{\frac{N-2}{2N}}\\
&\hspace{.3cm}
 = o(1).
\end{aligned}
\]
Therefore \eqref{equ-11-proposition-operator L-1} holds. Using this and
\eqref{equ-10-proposition-operator L-1} we conclude that \eqref{equ-6-proposition-operator L-1} holds for $i=1,\dots,k$.

Now we turn to the proof of $\phi_0^\infty=0$. Setting $\al_2=\frac{1}{N-2}$ we obtain as in \eqref{equ-9-proposition-operator L-1}:
\[%begin{eqnarray}\label{equ-14-proposition-operator L-1}
\begin{aligned}
&\int_{\Om_0^n}\nabla\phi_0^n\nabla\psi
 = \big((\eps^n)^{\al_2}\big)^{\frac{2-N}{2}}\int_{\Om}
    f'_0\big(V_{\eps^n,\la^n,\xi^n}(y)\big)\phi^n(y)\chi(y)\psi\big((\eps^n)^{-\al_2}y\big)
     + o(1)\\
&\hspace{1cm}
 = \big((\eps^n)^{\al_2}\big)^{\frac{2-N}{2}}\int_{|y|\le\eta/2}
    f'_0\left(-\sum_{j=1}^{k}PU_{\de_j^n,\xi_j^n}(y)+PV_{\si^n}(y)\right)
    \phi^n(y)\chi(y)\psi((\eps^n)^{-\al_2}y) + o(1)\\
&\hspace{1cm}
 = \big((\eps^n)^{\al_2}\big)^{\frac{2-N}{2}}\int_{|y|\le\eta/2}
    f'_0\big(V_{\si^n}(y)\big)\phi^n(y)\chi(y)\psi((\eps^n)^{-\al_2}y) + o(1)\\
&\hspace{1cm}
 = \int_{|(\eps^n)^{\al_2}x|\le\eta/2}
    f'_0\big(U_{\ov{\la}^n,0}(x)\big)\big((\eps^n)^{\al_2}\big)^{\frac{N-2}{2}}
    \phi^n\big((\eps^n)^{\al_2}x\big)\chi\big((\eps^n)^{\al_2}x\big)\psi(x) + o(1)\\
&\hspace{1cm}
 = \int_{\R^N}f'_0\big(U_{\ov{\la},0}\big)\phi_0^\infty\psi(x) + o(1).
\end{aligned}
\]%end{eqnarray}
Therefore $\phi_0^\infty$ is a weak solution of
\begin{equation}\label{equ-15-proposition-operator L-1}
-\De \phi_0^\infty=f'_0(U_{\ov{\la},0})\phi_0^\infty,\quad \text{in } D^{1,2}(\R^N).
\end{equation}
Similarly to \eqref{equ-11-proposition-operator L-1} there holds
\begin{equation}\label{equ-16-proposition-operator L-1}
\int_{\R^N}\nabla\phi_0^\infty(x)\nabla\Psi_{\ov{\la},0}^j(x)=0,
\quad\text{for } j=0,1,\dots, N,
\end{equation}
where $\Psi_{\ov{\la},0}^j:=\frac{\pa U_{\ov{\la},0}}{\pa x^j}$, for $j=1,\dots,N$, and $\Psi_{\ov{\la},0}^0:=\frac{\pa U_{\ov{\la},0}}{\pa \ov{\la}}$. This shows that $\phi_0^\infty=0$ as claimed.

\emph{Step 3.} We obtain a contradiction.

First we claim that
\begin{equation}\label{equ-17-proposition-operator L-1}
\lim_{n\to\infty}\int_{\Om}f'_0\big(V_{\eps^n,\la^n,\xi^n}(y)\big)(\phi^n(y))^2 = 0.
\end{equation}
In fact, \eqref{equ-1-projection estimate} and \eqref{equ-3-projection estimate} imply:
\[
\int_{\Om}
f'_0(V_{\eps^n,\la^n,\xi^n}(y))(\phi^n(y))^2
 =\int_{B(0,\frac{\eta}{4})\cup\bigcup\limits_{i=1}^{k} B(\xi_i,\frac{\eta}{4})}
f'_0\left(-\sum_{j=1}^{k}U_{\de_j^n,\xi_j^n}(y)+V_{\si^n}(y)\right)(\phi^n(y))^2 + o(1).
\]
Notice that $f'_0(U_{\la_i^n,0})\in L^{\frac{N}{2}}({\R^N})$ and \eqref{equ-6-proposition-operator L-1} imply
\begin{equation}\label{equ-18-proposition-operator L-1}
\begin{aligned}
\int_{B(\xi_i,\frac{\eta}{4})}
    f'_0\left(-\sum_{j=1}^{k}U_{\de_j^n,\xi_j^n}(y)+V_{\si^n}(y)\right)\big(\phi^n(y)\big)^2
 &= \int_{B(\xi_i,\frac{\eta}{4})}f'_0\big(U_{\de_i^n,\xi_i^n}(y)\big)(\phi^n(y))^2+o(1)\\
 &= \int_{|(\eps^n)^{\al_1}x|\le\frac{\eta}{4}}f'_0\big(U_{\la_i^n,0}(x)\big)(\phi_i^n(x))^2
           + o(1)\\
 &= o(1).
\end{aligned}
\end{equation}
Similarly we obtain:
\begin{equation}\label{equ-19-proposition-operator L-1}
\int_{B(0,\frac{\eta}{4})}
     f'_0\left(-\sum_{j=1}^{k}U_{\de_j^n,\xi_j^n}(y)+V_{\si^n}(y)\right)(\phi^n(y))^2
 = o(1).
\end{equation}
Now we obtain \eqref{equ-17-proposition-operator L-1} from
\eqref{equ-18-proposition-operator L-1} and \eqref{equ-19-proposition-operator L-1}.

On the other hand, \eqref{equ-3-proposition-operator L-1},
\eqref{equ-2-proposition-operator L-1}, and \eqref{equ-wn-proposition-operator L-1} imply:
\[
\begin{aligned}
\int_\Om |\nabla\phi^n|^2
 &= \int_\Om \nabla \iota^*(f'_0(V_{\eps^n,\la^n,\xi^n})\phi^n)\nabla\phi^n
        +\int_\Om \nabla h^n\nabla\phi^n+\int_\Om \nabla w^n\nabla\phi^n\\
 &= \int_\Om \nabla \iota^*\big(f'_0(V_{\eps^n,\la^n,\xi^n})\phi^n\big)\nabla\phi^n
    -\mu\int_\Om\frac{\iota^*(f'_0(V_{\eps^n,\la^n,\xi^n})\phi^n)\phi^n}{|x|^2}
    + o(1)\\
 &= \int_{\Om}f'_0\big(V_{\eps^n,\la^n,\xi^n}(y)\big)(\phi^n(y))^2 + o(1),
\end{aligned}
\]
which contradicts \eqref{equ-17-proposition-operator L-1} using \eqref{equ-1-proposition-operator L-1}.
\end{proof}

\begin{proposition}\label{proposition-estimate of error-1}
For every $\eta>0$ there exist $\eps_0>0$ and $c_0>0$ with the following property: for every $(\la,\xi)\in\cO_\eta$ and for every $\eps\in(0,\eps_0)$ there exists a unique solution $\phi_{\eps,\la,\xi}\in K^\bot_{\eps,\la,\xi}$ of equation \eqref{main-equality Com-1} satisfying
\begin{equation}\label{inequality-estimate of error-1}
\|\phi_{\eps,\la,\xi}\|_{\mu}
 \le c_0\left(\eps^{\frac{N+2}{2(N-2)}}+\eps^{\frac{1+2\al}{4}}\right),
\end{equation}
and $\Phi_\eps:\cO_\eta\to K^\bot_{\eps,\la,\xi}$ defined by $\Phi_\eps(\la,\xi):=\phi_{\eps,\la,\xi}$ is $C^1$.
\end{proposition}

\begin{proof} As in \cite{BarMiPis06CVPDE} solving (\ref{main-equality Com-1}) is equivalent to finding a fixed point of the operator
$T_{\eps,\la,\xi}:K^\bot_{\eps,\la,\xi}\to K^\bot_{\eps,\la,\xi}$ defined by
\[%begin{equation}%\label{equ-1-proposition-estimate of error-1}
T_{\eps,\la,\xi}(\phi)
 = L^{-1}_{\eps,\la,\xi}\Pi^\bot_{\eps,\la,\xi}
    \big(\iota^*\big(f_\eps(V_{\eps,\la,\xi}+\phi) - f'_0(V_{\eps,\la,\xi})\phi\big)
    - V_{\eps,\la,\xi}\big).
\]%end{equation}
We claim that $T_{\eps,\la,\xi}$ is a contraction mapping.

First of all, Proposition \ref{proposition-operator L-1}, Lemma \ref{(M17)-Lemma A} and \eqref{ineq-adjoint operator} imply
\begin{equation}\label{equ-2-proposition-estimate of error-1}
\begin{aligned}
&\|T_{\eps,\la,\xi}(\phi)\|_\mu
 \le C \|\iota^*(f_\eps(V_{\eps,\la,\xi}+\phi)-f'_0(V_{\eps,\la,\xi})\phi)
       - V_{\eps,\la,\xi}\|_\mu\\
&\hspace{1cm}
 \le C\left(\left\|\iota^*\left(f_\eps(V_{\eps,\la,\xi}+\phi)-f'_0(V_{\eps,\la,\xi})\phi
     -\left(-\sum_{i=1}^kf_0(U_{\de_i,\xi_i})+f_0(V_\si)\right)\right)\right\|_\mu\right.\\
&\hspace{2cm}
  + \left.\left\|\iota^*\left(-\sum_{i=1}^kf_0(U_{\de_i,\xi_i})+f_0(V_\si)\right)
     -V_{\eps,\la,\xi}\right\|_\mu\right)\\
&\hspace{1cm}
 \le C\left(\left\|f_\eps(V_{\eps,\la,\xi}+\phi)-f'_0(V_{\eps,\la,\xi})\phi
      -\left(-\sum_{i=1}^kf_0(U_{\de_i,\xi_i})+f_0(V_\si)\right)\right\|_{2N/(N+2)}\right.\\
&\hspace{2cm}
  + \left.
     \sum_{i=1}^kO(\mu\de_i)+O\left((\mu\si^{\frac{N-2}{2}})^{\frac{1}{2}}\right)\right)\\
&\hspace{1cm}
 \le C\big\|f_\eps(V_{\eps,\la,\xi}+\phi) - f_\eps(V_{\eps,\la,\xi})
       - f'_\eps(V_{\eps,\la,\xi})\phi\big\|_{2N/(N+2)}
     + C\big\|(f'_\eps(V_{\eps,\la,\xi})-f'_0(V_{\eps,\la,\xi}))\phi\big\|_{2N/(N+2)}\\
&\hspace{2cm}
     + C\big\|f_\eps(V_{\eps,\la,\xi})-f_0(V_{\eps,\la,\xi})\big\|_{2N/(N+2)}\\
&\hspace{2cm}
     + C\left\|f_0(V_{\eps,\la,\xi})
     - \left(-\sum_{i=1}^{k}f_0(U_{\de_i,\xi_i})+f_0(V_\si)\right)\right\|_{2N/(N+2)}\\
&\hspace{2cm}
     + \sum_{i=1}^{k}O(\mu\de_i) + O\left((\mu\si^{\frac{N-2}{2}})^{\frac{1}{2}}\right).
\end{aligned}
\end{equation}
By using Lemma \eqref{e53-e55} and noticing that
\begin{eqnarray*}
\|f_\eps(V_{\eps,\la,\xi}+\phi)-f_\eps(V_{\eps,\la,\xi})-f'_\eps(V_{\eps,\la,\xi})\phi\|_{2N/(N+2)}\le C\|\phi\|_\mu^{2^*-1},
\end{eqnarray*}
we deduce
\[
\begin{aligned}
\|T_{\eps,\la,\xi}(\phi)\|_\mu
 &\le C\|\phi\|_\mu^{2^*-1} + C\eps\|\phi\|_\mu + C\eps + O(\si^{\frac{N+2}{2}})
       + \sum_{i=1}^{k}O\big(\de_i^{\frac{N+2}{2}}\big)
       +\sum_{i=1}^{k}O(\mu\de_i) + O\big((\mu\si^{\frac{N-2}{2}})^{\frac{1}{2}}\big)\\
 &= C\|\phi\|_\mu^{2^*-1} + C\eps\|\phi\|_\mu + O\big(\eps^{\frac{N+2}{2(N-2)}}\big)
     + O\big(\eps^{\frac{1+2\al}{4}}\big).
\end{aligned}
\]
The remaining argument can be obtained by standard arguments, see e.g.~\cite{BarMiPis06CVPDE}.
\end{proof}

Now we consider the reduced functional
\[%begin{equation}%\label{def-I-1}
I_\eps(\la,\xi)=J_\eps(V_{\eps,\la,\xi}+\phi_{\eps,\la,\xi}).
\]%end{equation}

\begin{proposition}\label{proposition-reducement-1}
Let $(\la^0,\xi^0)$ with $\la^0=(\la_1^0,\dots,\la_{k}^0,\ov{\la}^0)$ and $\xi^0=(\xi_1^0,\xi_2^0,\dots,\xi_{k}^0)$ be a critical point of $I_\eps(\la,\xi)$. Then there exists a family of solutions to problem \eqref{pro} of the form
\[%begin{equation}%\label{equ-proposition-reducement-shape of solutions-1}
u_\eps=V_{\eps,\la,\xi}+\phi_{\eps,\la,\xi}.
\]%end{equation}
\end{proposition}

\begin{proof}
It is enough to prove that \eqref{main-equality-1} holds. Let $\pa_s$ denote one of $\pa_{\la_i}$, $\pa_{\ov{\la}}$, $\pa_{(\xi_i)^j}$, $i=1,\dots,k$, $j=1,\dots,N$. As in \cite{MussoPis10JMPA}, equation \eqref{main-equality Com-1} implies:
\[
\begin{aligned}%\label{equ1-proposition-reducement-1}
\pa_s I_\eps(\la,\xi)
 &= J'_\eps(V_{\eps,\la,\xi}+\phi_{\eps,\la,\xi})
     (\pa_s V_{\eps,\la,\xi} + \pa_s\phi_{\eps,\la,\xi})\\
 &= \sum_{i=1}^{k}\sum_{j=0}^N c_{i,j}
     \big(P\Psi_i^j,\pa_s V_{\eps,\la,\xi}+\pa_s \phi_{\eps,\la,\xi}\big)
    +c_0\big(P\ov{\Psi},\pa_s V_{\eps,\la,\xi}+\pa_s \phi_{\eps,\la,\xi}\big).
\end{aligned}
\]
Now it remains to prove that $c_{i,j}=0$ for $i=1,\dots,k$ and $j=0,\dots,N$, and $c_0=0$, provided $\eps>0$ is small enough.

If $(\la,\xi)$ is a critical point of $I_\eps(\la,\xi)$, then
\begin{equation}\label{equ2-proposition-reducement-1}
\sum_{i=1}^{k}\sum_{j=0}^N c_{i,j}(P\Psi_i^j,\pa_s V_{\eps,\la,\xi}+\pa_s \phi_{\eps,\la,\xi})
  + c_0(P\ov{\Psi},\pa_s V_{\eps,\la,\xi}+\pa_s \phi_{\eps,\la,\xi})
= 0.
\end{equation}
Observe that
\begin{equation}\label{equ3-proposition-reducement-1}
\pa_{\la_i} V_{\eps,\la,\xi} = -\eps^{\frac{1}{N-2}}P\Psi_i^0,\quad
\pa_{\ov{\la}}V_{\eps,\la,\xi}=\eps^{\frac{1}{N-2}}P\ov{\Psi},\quad
\pa_{(\xi_i)^j} V_{\eps,\la,\xi}=-P\Psi_i^j,\quad j=1,\dots,N.
\end{equation}
On the other hand, $(P\Psi_i^j,\phi_{\eps,\la,\xi})=0$ for $j=0,1,\dots,N$, Proposition \ref{proposition-estimate of error-1} and Lemma \ref{e63-68} imply
\[%begin{eqnarray}%\label{equ4-proposition-reducement-1}
(P\Psi_i^j,\pa_s\phi_{\eps,\la,\xi})=-(\pa_s P\Psi_i^j,\phi_{\eps,\la,\xi})=O(\|\pa_s P\Psi_i^j\|_\mu \|\phi_{\eps,\la,\xi}\|_\mu)=o(\|\pa_s P\Psi_i^j\|_\mu)=o(\de_i^{-2}).
\]%end{eqnarray}
Similarly,
\[%begin{eqnarray}\label{equ5-proposition-reducement-1}
(P\ov{\Psi},\pa_s\phi_{\eps,\la,\xi})=o(\|\pa_s P\ov{\Psi}\|_\mu)=o(\eps^{\frac{1}{N-2}}\si^{-2}).
\]%end{eqnarray}
Now Lemma \ref{e56-e62}, \eqref{equ2-proposition-reducement-1} and \eqref{equ3-proposition-reducement-1} yield
\[
\begin{aligned}%\label{equ6-proposition-reducement-1}
0 &= \sum_{i=1}^k\sum_{j=0}^N c_{i,j}(P\Psi_i^j,\pa_{\ov{\la}} V_{\eps,\la,\xi})
      + c_0(P\ov{\Psi},\pa_{\ov{\la}}V_{\eps,\la,\xi})
      + o\big(\eps^{\frac{1}{N-2}}\si^{-2}\big)\\
 &= \eps^{\frac{1}{N-2}}\left(\sum_{i=1}^k\sum_{j=0}^N
        c_{i,j}(P\Psi_i^j,P\ov{\Psi})+c_0(P\ov{\Psi},P\ov{\Psi})\right)
      +o\big(\eps^{\frac{1}{N-2}}\si^{-2}\big)\\
 &= c_0\wt{c}_0\eps^{\frac{1}{N-2}}\si^{-2}\big(1+o(1)\big),
\end{aligned}
\]
which implies $c_0=0$. Similar arguments show that $c_{i,j}=0$ for $i=1,\dots,k$, $j=0,1,\dots,N$.
\end{proof}

%%---------------------------------------------------------------------------------
\subsection{\textbf{Proofs of Theorems \ref{theorem-2 bubbles} and \ref{theorem-ball-k bubbles}\label{Subsection 3.2}}}

In this part, we consider $V_{\eps,\la,\xi}=-\sum\limits_{i=1}^{k}PU_{\de_i,\xi_i}+PV_\si$. The reduced energy is expanded as follows.

\begin{lemma}\label{lemma-expansion of J-1}
For $\eps\to0^+$ there holds
\begin{equation}\label{equ-expansion of J-1}
I_\eps(\la,\xi) = a_1+a_2\eps-a_3\eps^\al-a_4\eps\ln \eps+\psi(\la,\xi)\eps+o(\eps)
\end{equation}
$C^1$-uniformly with respect to $(\la,\xi)$ in compact sets of
$\cO_\eta$. The constants are given by $a_1=\frac{1}{N}(k+1)S_0^{\frac{N}{2}}$,
$a_2 = \frac{(k+1)}{2^*}\int_{\R^N} U_{1,0}^{2^*}\ln U_{1,0} -
 \frac{k+1}{(2^*)^2}S_0^{\frac{N}2}$,
$a_3=\frac{1}{2} S_0^{\frac{N-2}{2}}\ov{S}\mu_0$, and
$a_4=\frac{k+1}{2 \cdot2^*}\int_{\R^N} U_{1,0}^{2^*}$.
The function $\psi$ is given by
\[
\begin{aligned}%\label{equ-proposition-reducement-1}
\psi(\la,\xi)
 &= b_1(H(0,0)\ov{\la}^{N-2} + \sum_{i=1}^kH(\xi_i,\xi_i)\la_i^{N-2}
    + 2\sum_{i=1}^kG(\xi_i,0)\la_i^{\frac{N-2}{2}}\ov{\la}^{\frac{N-2}{2}}\\
 &\hspace{2cm}
   - 2\sum_{i,j=1,i<j}^kG(\xi_i,\xi_j)\la_i^{\frac{N-2}{2}}\la_j^{\frac{N-2}{2}})
   - b_2\ln(\la_1\la_2\dots\la_{k}\ov{\la})^{\frac{N-2}{2}},
\end{aligned}
\]
with $b_1=\frac{1}{2}C_0\int_{\R^N}U_{1,0}^{2^*-1}$ and
$b_2=\frac{1}{2^*}\int_{\R^N} U_{1,0}^{2^*}$.
\end{lemma}

\begin{proof} Observe that
\begin{eqnarray}\label{equ1-1-lemma-expansion of J-1}
J_\eps(V_{\eps,\la,\xi})
&=&\frac{1}{2}\int_{\Om}(|\nabla
V_{\eps,\la,\xi}|^2-\mu\frac{|V_{\eps,\la,\xi}|^2}{|x|^2})\\
\label{equ1-2-lemma-expansion of J-1}
&&-\frac{1}{2^*}\int_{\Om}|V_{\eps,\la,\xi}|^{2^*}\\
\label{equ1-3-lemma-expansion of J-1}
&&+(\frac{1}{2^*}\int_{\Om}|V_{\eps,\la,\xi}|^{2^*}
-\frac{1}{2^*-\eps}\int_{\Om}|V_{\eps,\la,\xi}|^{2^*-\eps}).
\end{eqnarray}
By Lemma \ref{e12,e24,e30,e34}, Lemma \ref{e44-e48}, and noticing $\mu=\mu_0\eps^\al$, $\eps\to0^+$, we obtain
\begin{eqnarray*}
\eqref{equ1-1-lemma-expansion of J-1}
&=& \frac12\int_{\Om}\left(|\nabla P V_\si|^2-\mu\frac{|P V_\si|^2}{|x|^2}\right)
     + \sum_{i=1}^k\left(|\nabla PU_{\de_i,\xi_i}|^2
          - \mu\frac{|PU_{\de_i,\xi_i}|^2}{|x|^2}\right)\\
&& -\sum_{i=1}^k\int_{\Om}\left(\nabla P V_\si\nabla PU_{\de_i,\xi_i}
         -\mu\frac{P V_\si PU_{\de_i,\xi_i}}{|x|^2}\right)\\
&& +\sum_{i,j=1,i<j}^k\int_\Om\left(\nabla PU_{\de_i,\xi_i}\nabla PU_{\de_j,\xi_j}
         -\mu\frac{PU_{\de_i,\xi_i}PU_{\de_j,\xi_j}}{|x|^2}\right)\\
&=& \frac12(k+1)S_0^{\frac{N}{2}}-\frac{N}{4} S_0^{\frac{N-2}{2}}\ov{S}\mu_0\eps^\al
     +\frac12C_0 \int_{\R^N}U_{1,0}^{2^*-1}
         \left(-H(0,0)\si^{N-2}-\sum_{i=1}^kH(\xi_i,\xi_i)\de_i^{N-2}\right.\\
&& \left. -2\sum_{i=1}^k\si^{\frac{N-2}{2}}\de_i^{\frac{N-2}{2}}G(\xi_i,0)
          +2\sum_{i,j=1,i<j}^kG(\xi_i,\xi_j)\de_i^{\frac{N-2}{2}}\de_j^{\frac{N-2}{2}}\right)
          +o(\eps).
\end{eqnarray*}

By Lemma \ref{e4,e41} and Lemma \ref{e44-e48}, and again observing $\mu=\mu_0\eps^\al$, $\eps\to0^+$ we obtain:
\begin{eqnarray*}
(\ref{equ1-2-lemma-expansion of J-1})
&=& -\frac1{2^*}(k+1)S_0^{\frac{N}{2}}+\frac{N-2}{4} S_0^{\frac{N-2}{2}}\ov{S}\mu_0\eps^\al
     +C_0 \int_{\R^N}U_{1,0}^{2^*-1}\left(H(0,0)\si^{N-2}
     +\sum_{i=1}^kH(\xi_i,\xi_i)\de_i^{N-2}\right.\\
&&\left. +2\sum_{i=1}^k\si^{\frac{N-2}{2}}\de_i^{\frac{N-2}{2}}G(\xi_i,0)
         -2\sum_{i,j=1,i<j}^kG(\xi_i,\xi_j)\de_i^{\frac{N-2}{2}}\de_j^{\frac{N-2}{2}}\right)
         +o(\eps).
\end{eqnarray*}
Next Lemma \ref{e4,e41}, Lemma \ref{e42} and Lemma \ref{e44-e48} yield:
\begin{eqnarray*}
\eqref{equ1-3-lemma-expansion of J-1}
&=& -\frac{\eps}{(2^*)^2}\int_{\Om}|V_{\eps,\la,\xi}|^{2^*}
    +\frac{\eps}{2^*}\int_{\Om}|V_{\eps,\la,\xi}|^{2^*}\ln|V_{\eps,\la,\xi}| + o(\eps)\\
&=& -\frac{\eps}{(2^*)^2}(k+1)S_0^{\frac{N}{2}} +
     \frac{\eps}{2^*}(-\frac{N-2}{2}\ln\si\cdot\int_{\R^N} V_1^{2^*}
     -\frac{N-2}{2}\ln(\de_1 \dots\de_{k})\cdot\int_{\R^N} U_{1,0}^{2^*}\\
&&  +\int_{\R^N} V_1^{2^*}\ln V_1 + k\int_{\R^N} U_{1,0}^{2^*}\ln U_{1,0}) + o(\eps)\\
&=& -\frac{\eps}{(2^*)^2}(k+1)S_0^{\frac{N}{2}}
    -\frac{(N-2)\eps}{2\cdot2^*}\int_{\R^N} U_{1,0}^{2^*}\cdot\ln(\de_1\dots\de_{k}\si)\\
&& +\frac{(k+1)\eps}{2^*}\int_{\R^N} U_{1,0}^{2^*}\ln U_{1,0} + o(\eps).
\end{eqnarray*}

Arguing similarly to Lemma 6.1 in \cite{MussoPis10JMPA}, we deduce from Proposition
\ref{proposition-estimate of error-1}, \eqref{equ-1-projection estimate}, \eqref{equ-3-projection estimate}, and Lemma \ref{e53-e55}, that
\begin{equation}\label{equ1-4-lemma-expansion of J-1}
\begin{aligned}
J_\eps(V_{\eps,\la,\xi}+\phi_{\eps,\la,\xi})-J_\eps(V_{\eps,\la,\xi})
 &=\frac12\|\phi_{\eps,\la,\xi}\|_\mu^2
    +\int_\Om(\nabla V_{\eps,\la,\xi}\nabla\phi_{\eps,\la,\xi}
    -\mu\frac{V_{\eps,\la,\xi}\phi_{\eps,\la,\xi}}{|x|^2})\\
 &\hspace{1cm}
  -\frac{1}{2^*-\eps}(\int_{\Om}|V_{\eps,\la,\xi}+\phi_{\eps,\la,\xi}|^{2^*-\eps}
    - |V_{\eps,\la,\xi}|^{2^*-\eps})\\
 &= o(\eps).
\end{aligned}
\end{equation}

Now \eqref{equ1-1-lemma-expansion of J-1}, \eqref{equ1-2-lemma-expansion of J-1}, \eqref{equ1-3-lemma-expansion of J-1} and \eqref{equ1-4-lemma-expansion of J-1} imply \eqref{equ-expansion of J-1}. Actually,  \eqref{equ-expansion of J-1}) also holds $C^1$-uniformly with respect to $(\la,\xi)$ in compact sets of $\cO_\eta$; see for
example \cite[Lemma~7.1]{MussoPis10JMPA}. We omit the details here.
\end{proof}

\begin{altproof}{\ref{theorem-2 bubbles}}
%\noindent\textbf{Proof of Theorem \ref{theorem-2 bubbles}.}
The reduced function $\psi(\la,\xi)$ from Lemma \ref{lemma-expansion of J-1} becomes (here $k=1$):
\[
\psi(\la,\xi)
 = b_1\left(H(0,0)\ov{\la}^{N-2} + H(\xi_1,\xi_1)\la_1^{N-2}
    + 2G(\xi_1,0)\la_1^{\frac{N-2}{2}}\ov{\la}^{\frac{N-2}{2}}\right)
    - b_2\ln(\la_1\ov{\la})^{\frac{N-2}{2}}.
\]
Now the first part is almost the same as the one in \cite[Theorem~1]{BarMiPis06CVPDE}. We therefore omit it here.

Now we prove the second part. The symmetry assumption $(S_1)$ and the principle of symmetric criticality (see e.g.~\cite[Lemma 2.4]{BarDPis-prepr1}) allow us to consider the constrained function of $\psi(\la,\xi)$ as follows:
\[
\ov{\psi}(\la,t)
 = b_1\left(h(0,0)\ov{\la}^{N-2} + h(t,t)\la_1^{N-2} +
            2g(t,0)\la_1^{\frac{N-2}{2}}\ov{\la}^{\frac{N-2}{2}}\right)
   - b_2\ln(\la_1\ov{\la})^{\frac{N-2}{2}}.
\]
For $t\in(a,b)\setminus\{0\}$, let
\begin{equation}\label{equ-1-proof-theorem-2 bubbles}
\frac{\pa\ov{\psi}(\la,t)}{\pa\la_1}
 = (N-2)b_1\left(h(t,t)\la_1^{N-3}
        + g(t,0)\la_1^{\frac{N-4}{2}}\ov{\la}^{\frac{N-2}{2}}\right)
    -\frac{(N-2)b_2}{2\la_1}=0,
\end{equation}
and
\begin{equation}\label{equ-2-proof-theorem-2 bubbles}
\frac{\pa\ov{\psi}(\la,t)}{\pa\ov{\la}}
 =(N-2)b_1\left(h(0,0)\ov{\la}^{N-3}
        + g(t,0)\la_1^{\frac{N-2}{2}}\ov{\la}^{\frac{N-4}{2}}\right)
   - \frac{(N-2)b_2}{2\ov{\la}}
 = 0.
\end{equation}
Then it is easy to obtain a unique $\la_1(t)>0$ and a unique $\ov{\la}(t)>0$ with
\begin{equation}\label{equ-3-proof-theorem-2 bubbles}
(\la_1(t))^{\frac{N-2}{2}}
 = \sqrt{\frac{1}{h(t,t)+g(t,0)(\frac{h(t,t)}{h(0,0)})^{\frac{1}{2}}}\cdot\frac{b_2}{2b_1}}
\end{equation}
and
\begin{equation}\label{equ-4-proof-theorem-2 bubbles}
(\ov{\la}(t))^{\frac{N-2}{2}}
 = \sqrt{\frac{1}{h(0,0)+g(t,0)(\frac{h(0,0)}{h(t,t)})^{\frac{1}{2}}}\cdot\frac{b_2}{2b_1}}.
\end{equation}
Now an easy computation using \eqref{equ-1-proof-theorem-2 bubbles} and \eqref{equ-2-proof-theorem-2 bubbles} shows that:
\begin{eqnarray*}
\frac{\pa^2\ov{\psi}(\la,t)}{\pa\la_1^2}\Big|_{\la_1=\la_1(t),\ov{\la}=\ov{\la}(t)}
&=& (N-2)b_1\left((N-3)h(t,t)\la_1^{N-4}
            + \frac{N-4}{2}g(t,0)\la_1^{\frac{N-6}{2}}\ov{\la}^{\frac{N-2}{2}}\right)
      + \frac{(N-2)b_2}{2\la_1^2}\\
&=& (N-2)b_1\left((N-2)h(t,t)\la_1^{N-4}
      +\frac{N-2}{2}g(t,0)\la_1^{\frac{N-6}{2}}\ov{\la}^{\frac{N-2}{2}}\right),
\end{eqnarray*}
and
\begin{eqnarray*}
\frac{\pa^2\ov{\psi}(\la,t)}{\pa\ov{\la}^2}\Big|_{\la_1=\la_1(t),\ov{\la}=\ov{\la}(t)}
&=& (N-2)b_1\left((N-3)h(0,0)\ov{\la}^{N-4}
           +\frac{N-4}{2}g(t,0)\la_1^{\frac{N-2}{2}}\ov{\la}^{\frac{N-6}{2}}\right)
      +\frac{(N-2)b_2}{2\ov{\la}^2}\\
&=& (N-2)b_1\left((N-2)h(0,0)\ov{\la}^{N-4}
           +\frac{N-2}{2}g(t,0)\la_1^{\frac{N-2}{2}}\ov{\la}^{\frac{N-6}{2}}\right),\\
\frac{\pa^2\ov{\psi}(\la,t)}{\pa\ov{\la}\pa\la_1}\Big|_{\la_1=\la_1(t),\ov{\la}=\ov{\la}(t)}
&=& \frac{(N-2)^2}{2}b_1g(t,0)\la_1^{\frac{N-4}{2}}\ov{\la}^{\frac{N-4}{2}}.
\end{eqnarray*}
It follows that the Hessian matrix $D^2_{\la_1,\ov{\la}}\ov{\psi}(\la,t)|_{\la_1=\la_1(t),\ov{\la}=\ov{\la}(t)}$
is positively definite and therefore nondegenerate.

Now we consider the reduced function
\[
\nu(t)
 :=\ov{\psi}(\la_1(t),\ov{\la}(t),t)
 = b_2-b_2\ln\big(\la_1(t)\ov{\la}(t)\big)^{\frac{N-2}{2}}
 = b_2-b_2\ln\frac{b_2}{2b_1} + b_2\ln\vphi(t),
\]
where we used \eqref{equ-1-proof-theorem-2 bubbles}, \eqref{equ-2-proof-theorem-2 bubbles}, \eqref{equ-3-proof-theorem-2 bubbles}, and \eqref{equ-4-proof-theorem-2 bubbles}. Now observe that
\[
\nu'(t) = b_2\frac{\vphi'(t)}{\vphi(t)},\quad
\nu''(t) = b_2\frac{\vphi''(t)\vphi(t)-(\vphi'(t))^2}{\vphi^2(t)},
\]
so, using that $t^*$ is a (local) minimum of $\vphi(t)$ in $(a,b)\backslash\{0\}$, then it is also a (local) minimum, and therefore a nondegenerate critical point of $\nu(t)$.
\end{altproof}

When $\Om=B(0,1)$, we consider the existence of solutions with $k+1$ bubbles, one positive and $k$ negative for $k=2,3$. We also show that the idea for $k=2,3$ is not applicable for $k=4$. Notice first that the principle of symmetric criticality (see e.g.~\cite[Lemma 2.4]{BarDPis-prepr1}) allows us to constrain the problem on
$B^{(2)}=\{x=(x_1,x_2,0,\dots,0): x\in B(0,1)\}$
and then to place the bubbles at
\begin{eqnarray}\label{equ-1-lemma blowing up speed}
\xi_i, \xi_{i+1}\in B^{(2)}, \quad
\xi_{i+1}=e^{2\pi \sqrt{-1}/k} \xi_i \text{ for } i=1,\dots,k-1.
\end{eqnarray}
Here we used complex coordinates in $B^{(2)} \subset \R^2\times\{0\}$.

For $x,y,z\in\Om,x\neq y\neq z$, let
\begin{eqnarray*}
\al_1(x,y)&=&\frac{-G(x,0)+\sqrt{G^2(x,0)+4H(0,0)(H(x,x)-G(x,y))}}{2H(0,0)},\\
\be_1(x,y)&=&H(x,x)-G(x,y)+G(x,0)\al_1(x,y),\\
\al_2(x,y)&=&\frac{-2G(x,0)+\sqrt{4G^2(x,0)+4H(0,0)(H(x,x)-2G(x,y))}}{2H(0,0)},\\
\be_2(x,y)&=&H(x,x)-2G(x,y)+G(x,0)\al_2(x,y),\\
\al_3(x,y,z)&=&\frac{-3G(x,0)+\sqrt{9G^2(x,0)+4H(0,0)(H(x,x)-2G(x,y)-G(x,z))}}{2H(0,0)},\\
\be_3(x,y,z)&=&H(x,x)-2G(x,y)-G(x,z)+G(x,0)\al_3(x,y,z).
\end{eqnarray*}
Then we have the following lemma.
\begin{lemma}\label{lemma-blowing up speed}
Let $\Om=B(0,1)$ and $k=2,3,4$. If $(\la,\xi)=(\la_1,\dots,\la_{k},\ov{\la},\xi_1,\dots,\xi_{k})$ is a critical point of $\psi(\la,\xi)$ such that \eqref{equ-1-lemma blowing up speed} holds, then
\begin{equation}\label{equ-2-add-lemma blowing up speed}
\la_1=\ldots=\la_{k}.
\end{equation}
Moreover we have for $k=2$:
\begin{equation}\label{equ-2-add1-lemma blowing up speed}
\ov{\la}^{\frac{N-2}{2}}=\al_1(\xi_1,\xi_2)\la_1^{\frac{N-2}{2}},\quad
\la_1^{\frac{N-2}{2}}=\sqrt{\frac{1}{\be_1(\xi_1,\xi_2)}\cdot\frac{b_2}{2b_1}},\quad
H(\xi_1,\xi_1)-G(\xi_1,\xi_2)>0;
\end{equation}
for $k=3$:
\begin{equation}\label{equ-2-add2-lemma blowing up speed}
\ov{\la}^{\frac{N-2}{2}}=\al_2(\xi_1,\xi_2)\la_1^{\frac{N-2}{2}},\quad
\la_1^{\frac{N-2}{2}}=\sqrt{\frac{1}{\be_2(\xi_1,\xi_2)}\cdot\frac{b_2}{2b_1}},\quad H(\xi_1,\xi_1)-2G(\xi_1,\xi_2)>0;
\end{equation}
and for $k=4$:
\begin{equation}\label{equ-2-add3-lemma blowing up speed}
\ov{\la}^{\frac{N-2}{2}}=\al_3(\xi_1,\xi_2,\xi_3)\la_1^{\frac{N-2}{2}},\
\la_1^{\frac{N-2}{2}}=\sqrt{\frac{1}{\be_3(\xi_1,\xi_2,\xi_3)}\cdot\frac{b_2}{2b_1}},\
H(\xi_1,\xi_1)-2G(\xi_1,\xi_2)-G(\xi_1,\xi_3)>0.
\end{equation}
\end{lemma}

\begin{proof}
If $k=2$ and $(\la,\xi)$ is a critical point of $\psi(\la,\xi)$, then the equations $\frac{\pa\psi(\la,\xi)}{\pa \ov{\la}}
 = \frac{\pa\psi(\la,\xi)}{\pa \la_1}
 = \frac{\pa\psi(\la,\xi)}{\pa \la_2} = 0$,
imply
\begin{eqnarray}\label{equ-2-lemma blowing up speed}
H(0,0)\ov{\la}^{N-2} + G(\xi_1,0)\la_1^{\frac{N-2}{2}}\ov{\la}^{\frac{N-2}{2}}
   +G(\xi_2,0)\la_2^{\frac{N-2}{2}}\ov{\la}^{\frac{N-2}{2}}
 = \frac{b_2}{2b_1},\\
\label{equ-3-lemma blowing up speed}
H(\xi_1,\xi_1)\la_1^{N-2} + G(\xi_1,0)\la_1^{\frac{N-2}{2}}\ov{\la}^{\frac{N-2}{2}}
   -G(\xi_1,\xi_2)\la_1^{\frac{N-2}{2}}\la_2^{\frac{N-2}{2}}
 = \frac{b_2}{2b_1},\\
\label{equ-4-lemma blowing up speed}
H(\xi_2,\xi_2)\la_2^{N-2}+G(\xi_2,0)\la_2^{\frac{N-2}{2}}\ov{\la}^{\frac{N-2}{2}}
-G(\xi_1,\xi_2)\la_1^{\frac{N-2}{2}}\la_2^{\frac{N-2}{2}}
 = \frac{b_2}{2b_1}.
\end{eqnarray}
Notice that \eqref{equ-1-lemma blowing up speed} yields
\[
H(\xi_1,\xi_1)=H(\xi_2,\xi_2), G(\xi_1,0)=G(\xi_2,0),
\]
so \eqref{equ-3-lemma blowing up speed} and \eqref{equ-4-lemma blowing up speed} imply
\[
\left(\la_2^{\frac{N-2}{2}}-\la_1^{\frac{N-2}{2}}\right)
   \left(H(\xi_1,\xi_1)\left(\la_2^{\frac{N-2}{2}}+\la_1^{\frac{N-2}{2}}\right)
   + G(\xi_1,0)\ov{\la}^{\frac{N-2}{2}}\right)
 = 0.
\]
Since $H(\xi_1,\xi_1)>0$, $G(\xi_1,0)>0$, $\la_1>0$, $\la_2>0$, and $\ov{\la}>0$, we obtain $\la_1=\la_2$, and \eqref{equ-2-add1-lemma blowing up speed} follows by an easy computation.

If $k=3$ and $(\la,\xi)$ is a critical point of $\psi(\la,\xi)$, then
\begin{eqnarray}\label{equ-5-lemma blowing up speed}
H(0,0)\ov{\la}^{N-2} + G(\xi_1,0)\la_1^{\frac{N-2}{2}}\ov{\la}^{\frac{N-2}{2}}
    + G(\xi_2,0)\la_2^{\frac{N-2}{2}}\ov{\la}^{\frac{N-2}{2}}
    + G(\xi_3,0)\la_3^{\frac{N-2}{2}}\ov{\la}^{\frac{N-2}{2}}
 = \frac{b_2}{2b_1},\\
\label{equ-6-lemma blowing up speed}
H(\xi_1,\xi_1)\la_1^{N-2} + G(\xi_1,0)\la_1^{\frac{N-2}{2}}\ov{\la}^{\frac{N-2}{2}}
    - G(\xi_1,\xi_2)\la_1^{\frac{N-2}{2}}\la_2^{\frac{N-2}{2}}
    - G(\xi_1,\xi_3)\la_1^{\frac{N-2}{2}}\la_3^{\frac{N-2}{2}}
 = \frac{b_2}{2b_1},\\
\label{equ-7-lemma blowing up speed}
H(\xi_2,\xi_2)\la_2^{N-2} + G(\xi_2,0)\la_2^{\frac{N-2}{2}}\ov{\la}^{\frac{N-2}{2}}
    - G(\xi_1,\xi_2)\la_1^{\frac{N-2}{2}}\la_2^{\frac{N-2}{2}}
    - G(\xi_2,\xi_3)\la_2^{\frac{N-2}{2}}\la_3^{\frac{N-2}{2}}
 = \frac{b_2}{2b_1},\\
\label{equ-8-lemma blowing up speed}
H(\xi_3,\xi_3)\la_3^{N-2} + G(\xi_3,0)\la_3^{\frac{N-2}{2}}\ov{\la}^{\frac{N-2}{2}}
    - G(\xi_1,\xi_3)\la_1^{\frac{N-2}{2}}\la_3^{\frac{N-2}{2}}
    - G(\xi_2,\xi_3)\la_2^{\frac{N-2}{2}}\la_3^{\frac{N-2}{2}}
 = \frac{b_2}{2b_1}.
\end{eqnarray}
The ansatz \eqref{equ-1-lemma blowing up speed} gives
\[
H(\xi_1,\xi_1) = H(\xi_2,\xi_2)=H(\xi_3,\xi_3),\
G(\xi_1,0) = G(\xi_2,0)=G(\xi_3,0),\
G(\xi_1,\xi_2) = G(\xi_2,\xi_3)=G(\xi_1,\xi_3).
\]
We can see from (\ref{equ-6-lemma blowing up speed}) and
(\ref{equ-7-lemma blowing up speed}) that
\[
\left(\la_2^{\frac{N-2}{2}}-\la_1^{\frac{N-2}{2}}\right)
   (H(\xi_1,\xi_1)(\la_2^{\frac{N-2}{2}}
   + \la_1^{\frac{N-2}{2}}) + G(\xi_1,0)\ov{\la}^{\frac{N-2}{2}}
   - G(\xi_1,\xi_2)\la_3^{\frac{N-2}{2}})
 = 0.
\]
Assume that
\begin{equation}\label{equ-9-lemma blowing up speed}
H(\xi_1,\xi_1)\left(\la_2^{\frac{N-2}{2}} + \la_1^{\frac{N-2}{2}}\right)
   + G(\xi_1,0)\ov{\la}^{\frac{N-2}{2}} - G(\xi_1,\xi_2)\la_3^{\frac{N-2}{2}}
 = 0.
\end{equation}
Multiplying this by $\la_1^{\frac{N-2}{2}}$ and combing it with
\eqref{equ-6-lemma blowing up speed}, we obtain
\[
-H(\xi_1,\xi_1)\la_1^{\frac{N-2}{2}}\la_2^{\frac{N-2}{2}}
   - G(\xi_1,\xi_2)\la_1^{\frac{N-2}{2}}\la_2^{\frac{N-2}{2}} - \frac{b_2}{2b_1}
 = 0
\]
which is obviously impossible. Therefore $\la_2=\la_1$, and similarly $\la_3=\la_1$. Finally, (\ref{equ-2-add2-lemma blowing up speed}) can be computed directly.

If $k=4$ and $(\la,\xi)$ is a critical point of $\psi(\la,\xi)$, then
\begin{eqnarray}\label{equ-10-lemma blowing up speed}
H(0,0)\ov{\la}^{N-2}+\sum_{i=1}^4G(\xi_i,0)\la_1^{\frac{N-2}{2}}\ov{\la}^{\frac{N-2}{2}}
 = \frac{b_2}{2b_1},\\
\label{equ-11-lemma blowing up speed}
H(\xi_1,\xi_1)\la_1^{N-2}+G(\xi_1,0)\la_1^{\frac{N-2}{2}}\ov{\la}^{\frac{N-2}{2}}
-\sum_{i\neq1}G(\xi_1,\xi_i)\la_1^{\frac{N-2}{2}}\la_i^{\frac{N-2}{2}}
 = \frac{b_2}{2b_1},\\
\label{equ-12-lemma blowing up speed}
H(\xi_2,\xi_2)\la_2^{N-2}+G(\xi_2,0)\la_2^{\frac{N-2}{2}}\ov{\la}^{\frac{N-2}{2}}
-\sum_{i\neq2}G(\xi_2,\xi_i)\la_2^{\frac{N-2}{2}}\la_i^{\frac{N-2}{2}}
 = \frac{b_2}{2b_1},\\
\label{equ-13-lemma blowing up speed}
H(\xi_3,\xi_3)\la_3^{N-2}+G(\xi_3,0)\la_3^{\frac{N-2}{2}}\ov{\la}^{\frac{N-2}{2}}
-\sum_{i\neq3}G(\xi_3,\xi_i)\la_3^{\frac{N-2}{2}}\la_i^{\frac{N-2}{2}}
 = \frac{b_2}{2b_1},\\
\label{equ-14-lemma blowing up speed}
H(\xi_4,\xi_4)\la_4^{N-2}+G(\xi_4,0)\la_4^{\frac{N-2}{2}}\ov{\la}^{\frac{N-2}{2}}
-\sum\limits_{i\neq4}G(\xi_4,\xi_i)\la_4^{\frac{N-2}{2}}\la_i^{\frac{N-2}{2}}
 = \frac{b_2}{2b_1}.
\end{eqnarray}
The ansatz \eqref{equ-1-lemma blowing up speed} gives
\[
H(\xi_1,\xi_1) = H(\xi_2,\xi_2) = H(\xi_3,\xi_3)=H(\xi_4,\xi_4),\quad
G(\xi_1,0) = G(\xi_2,0) = G(\xi_3,0) = G(\xi_4,0),
\]
and
\[
G(\xi_1,\xi_2) = G(\xi_2,\xi_3) = G(\xi_3,\xi_4) = G(\xi_4,\xi_1),\quad
G(\xi_1,\xi_3) = G(\xi_2,\xi_4).
\]
Then using the same argument as the one in case $k=3$, \eqref{equ-11-lemma blowing up speed} and \eqref{equ-13-lemma blowing up speed} imply $\la_1=\la_3$. Similarly we obtain
$\la_2=\la_4$. Substituting this into \eqref{equ-11-lemma blowing up speed} and \eqref{equ-12-lemma blowing up speed}, we then get $\la_1=\la_2$. At last, \eqref{equ-2-add3-lemma blowing up speed} follows by direct computation.
\end{proof}

\begin{altproof}{\ref{theorem-ball-k bubbles}}
%\noindent\textbf{Proof of Theorem \ref{theorem-ball-k bubbles}.}
For $k=2$, by \eqref{equ-1-lemma blowing up speed} we can assume $\xi_1=-\xi_2=(t,0,\dots,0)$, $0<t<1$. Then by \eqref{equ-2-add-lemma blowing up speed} in Lemma \ref{lemma-blowing up speed}, the reduced function $\psi(\la,t)$ becomes
\[
f_1(\la_1,\ov{\la},t)
 = b_1\left(h(0,0)\ov{\la}^{N-2} + 2h(t,t)\la_1^{N-2}
        + 4g(t,0)\la_1^{\frac{N-2}{2}}\ov{\la}^{\frac{N-2}{2}} - 2g(t,-t)\la_1^{N-2}\right)
   - b_2\ln\left(\la_1^2\ov{\la}\right)^{\frac{N-2}{2}}.
\]
The remaining argument is the same as the one in \cite{BarDPis-prepr2}.

Now we consider the case $k=3$. As a consequence of \eqref{equ-1-lemma blowing up speed} we may assume
$\xi_1=(t,0,\ldots,0)$, $\xi_2=\left(-\frac{t}{2},\frac{\sqrt{3}t}{2},0,\ldots,0\right)$, $\xi_3=\left(-\frac{t}{2},-\frac{\sqrt{3}t}{2},0,\ldots,0\right),$ $0<t<1$. Then Lemma~\ref{lemma-blowing up speed} allows us to consider the function
\[
\begin{aligned}
f_2(\la_1,\ov{\la},t)
 &= b_1\left(H(0,0)\ov{\la}^{N-2} + 3H(\xi_1,\xi_1)\la_1^{N-2}
        + 6G(\xi_1,0)\la_1^{\frac{N-2}{2}}\ov{\la}^{\frac{N-2}{2}}
        - 6G(\xi_1,\xi_2)\la_1^{N-2}\right)\\
 &\hspace{1cm}
   - b_2\ln\left(\la_1^3\ov{\la}\right)^{\frac{N-2}{2}}.
\end{aligned}
\]
Setting
\[
\ga_1(t)
 := H(\xi_1,\xi_1)-2G(\xi_1,\xi_2)
   = \frac{1}{(1-t^2)^{N-2}}-\frac{2}{(\sqrt{3}t)^{N-2}}+\frac{2}{(t^4+t^2+1)^{\frac{N-2}{2}}}
\]
and
\[
\tau_1(t) := G(\xi_1,0) = \frac{1}{t^{N-2}}-1,
\]
a direct computation shows that $\ga'_1(t)>0$, $\ga_1(t)\to -\infty$ as $t\to 0^+$, and $\ga_1(\frac12)>0$. Thus there exists $t^*\in(0,\frac{1}{2})$ such that
\begin{equation}\label{equ-1-proof-theorem-ball-k bubbles}
\ga_1(t^*)=0,\quad\ga_1(t)>0 \text{ for all } t\in(t^*,1).
\end{equation}
Then for $t\in(t^*,1)$ there exist unique $\la_1(t)$, $\ov{\la}(t)$ such that
\begin{equation}\label{equ-2-proof-theorem-ball-k bubbles}
\frac{\pa f_2(\la_1,\ov{\la},t)}{\pa\la_1} = 0,\quad
\frac{\pa f_2(\la_1,\ov{\la},t)}{\pa\ov{\la}} = 0,
\end{equation}
where $\la_1(t), \ov{\la}(t)$ are from \eqref{equ-2-add2-lemma blowing up speed}. Moreover, a direct computation shows that
\begin{eqnarray*}
\frac{\pa^2 f_2(\la_1(t),\ov{\la}(t),t)}{\pa\la_1^2}
&=& 3(N-2)b_1\left((N-3)\ga_1(t)\la_1^{N-4}
       +\frac{N-4}{2}\tau_1(t)\la_1^{\frac{N-6}{2}}\ov{\la}^{\frac{N-2}{2}}\right)
     +\frac{3(N-2)b_2}{2\la_1^2}\\
&=& 3(N-2)b_1\left((N-2)\ga_1(t)\la_1^{N-4}
         + \frac{N-2}{2}\tau_1(t)\la_1^{\frac{N-6}{2}}\ov{\la}^{\frac{N-2}{2}}\right),\\
\frac{\pa^2 f_2(\la_1(t),\ov{\la}(t),t)}{\pa\ov{\la}^2}
&=& (N-2)b_1\left((N-3)H(0,0)\ov{\la}^{N-4}
        + \frac{3(N-4)}{2}\tau_1(t)\la_1^{\frac{N-2}{2}}\ov{\la}^{\frac{N-6}{2}}\right)
     + \frac{(N-2)b_2}{2\ov{\la}^2}\\
&=& (N-2)b_1\left((N-2)H(0,0)\ov{\la}^{N-4}
        +\frac{3(N-2)}{2}\tau_1(t)\la_1^{\frac{N-2}{2}}\ov{\la}^{\frac{N-6}{2}}\right),\\
\frac{\pa^2 f_2(\la_1(t),\ov{\la}(t),t)}{\pa\ov{\la}\pa\la_1}
&=& \frac{3(N-2)^2}{2}b_1\tau_1(t)\la_1^{\frac{N-4}{2}}\ov{\la}^{\frac{N-4}{2}}.
\end{eqnarray*}
It follows that the Hessian matrix
$D^2_{\la_1,\ov{\la}}f_2(\la_1(t),\ov{\la}(t),t)$ is positively definite and therefore nondegenerate. Then it is enough to consider the function
\[
\nu_1(t)
 := f_2\left(\la_1(t),\ov{\la}(t),t\right)
  = 2b_2-b_2\ln\left(\la_1^3(t)\ov{\la}(t)\right)^{\frac{N-2}{2}}.
\]
As in \cite[(3.4)]{BarDPis-prepr2} there holds
\begin{equation}%\label{equ-3-proof-theorem-ball-k bubbles}
 \lim_{t\to (t^*)^+}\nu_1(t)=-\infty\quad\text{and}\quad
\lim_{t\to 1^-}\nu_1(t)=+\infty.
\end{equation}
Now we prove $\nu_1'(\frac{1}{2})<0$ for $N$ large. Setting
\[
\al_2(t):=\al_2(\xi_1,\xi_2)=-\tau_1(t)+\sqrt{\tau_1^2(t)+\ga_1(t)},
\]
where we used $H(0,0)=1$, we obtain
\[
\nu_1'(t)
 = \frac{\pa f_2(\la_1(t),\ov{\la}(t),t)}{\pa t}
 = 3b_1\left(\ga'_1(t)+2\al_2(t)\tau'_1(t)\right)\la_1^{N-2}.
\]
Then by letting $\iota_1(t):=\ga'_1(t)+2\al_2(t)\tau'_1(t)$, we need to show $\iota_1\left(\frac12\right)<0$ for $N$ large enough. Since $\frac{\ga_1(\frac12)}{\tau_1^2(\frac12)}<1$ for $N$ large we see as in \cite[(3.9)]{BarDPis-prepr2} that
\[
\iota_1({\textstyle\frac12})
 \le \ga'_1({\textstyle\frac12})
      + \frac{4\ga_1(\frac12)}{5\tau_1(\frac{1}{2})}\tau'_1({\textstyle\frac12}).
\]
A direct computation gives for $N$ large:
\begin{eqnarray*}
\ga'_1({\textstyle\frac12})
 &=& (N-2)\left(({\textstyle\frac43})^{N-1} + 4({\textstyle\frac{2}{\sqrt{3}}})^{N-2}
      - \frac{\frac{3}{2}}{(\frac{1}{16}+\frac{1}{4}+1)^{\frac{N}{2}}}\right)
  < \frac{11(N-2)}{10}({\textstyle\frac43})^{N-1},\\
\tau'_1({\textstyle\frac12})
 &=& -(N-2)2^{N-1},\\
\frac{\ga_1(\frac12)}{\tau_1(\frac12)}
 &=& \frac{(\frac{4}{3})^{N-2}-2(\frac{2}{\sqrt{3}})^{N-2}
      + \frac{2}{(\frac{1}{16}+\frac{1}{4}+1)^{\frac{N-2}{2}}}}{2^{N-2}-1}
  > \frac{11}{12}\cdot\frac{(\frac{4}{3})^{N-2}}{2^{N-2}},\\
\end{eqnarray*}
which yield $\iota_1(\frac{1}{2})<0$ for $N$ large enough. Then we have $t_1,t_2\in(t^*,1)$, $t_1\neq t_2$ such that $\nu_1'(t_1)=0$, $\nu_1'(t_2)=0$, and we conclude as in \cite{BarDPis-prepr2}.
\end{altproof}

\begin{remark}\label{rem:nonexist}
a) For $k=3$, $N=7$, numerical computations show that one cannot find $t_0\in(t^*,1)$ such that $\nu_1'(t_0)=0$. Therefore we can only consider $N$ large enough in this case.

b) For $k=4$, the idea above also cannot give the existence of solutions with $k+1$ bubbles, one positive at the origin and $k$ negative. In fact, following the above idea, with the ansatz \eqref{equ-1-lemma blowing up speed} we may assume $\xi_1=(t,0,\dots,0)$, $\xi_2=(0,t,0,\dots,0)$, $\xi_3=(-t,0,\dots,0)$, and $\xi_4=(0,-t,\dots,0)$,
$0<t<1$. As a consequence of Lemma~\ref{lemma-blowing up speed} we need to consider the function
\[
\begin{aligned}
f_3(\la_1,\ov{\la},t)
 &= b_1\left(H(0,0)\ov{\la}^{N-2}+4H(\xi_1,\xi_1)\la_1^{N-2}+8G(\xi_1,0)\la_1^{\frac{N-2}{2}}\ov{\la}^{\frac{N-2}{2}}\right.
\\
&\hspace{1cm}
  \left.  -8G(\xi_1,\xi_2)\la_1^{N-2}-4G(\xi_1,\xi_3)\la_1^{N-2} \right)-b_2\ln(\la_1^4\ov{\la})^{\frac{N-2}{2}}.
  \end{aligned}
\]

Now let $\tau_1(t)$ be the same as above and define
\begin{eqnarray*}
\ga_2(t)
 &:=& H(\xi_1,\xi_1)-2G(\xi_1,\xi_2)-G(\xi_1,\xi_3)\\
 &=& \frac{1}{(1-t^2)^{N-2}} - \frac{2}{(\sqrt{2}t)^{N-2}} + \frac{2}{(t^4+1)^{\frac{N-2}{2}}}
      - \frac{1}{(2t)^{N-2}}+\frac{1}{(t^2+1)^{N-2}}.
\end{eqnarray*}
A direct computation shows that
\[
\ga'_2(t)
 = (N-2)\left(\frac{2t}{(1-t^2)^{N-1}} + \frac{2}{(\sqrt{2})^{N-2}t^{N-1}}
    - \frac{4t^3}{(t^4+1)^{\frac{N}{2}}} + \frac{1}{2^{N-2}t^{N-1}}
    - \frac{2t}{(t^2+1)^{N-1}}\right)
 > 0.
\]
Clearly $\ga_2(t)\to -\infty$ as $t\to 0^+$, and $\ga_2\left(\frac{1}{\sqrt{2}}\right)>0$. Then there exists $t^*\in(0,\frac{1}{\sqrt{2}})$ such that
\begin{equation}\label{equ-1-proof-theorem-ball-k bubbles-2}
\ga_2(t^*)=0,\ga_2(t)>0,~~ \forall~ t\in(t^*,1).
\end{equation}
Set $\iota_2(t):=\ga'_2(t)+2\al_3(t)\tau'_1(t)$, where
$\al_3(t):=\al_3(\xi_1,\xi_2,\xi_3)=\frac{-3\tau_1(t)+\sqrt{9\tau_1^2(t)+4\ga_2(t)}}{2}$. If we can prove that $\iota_2(t_0)=0$ for some $t_0\in(t^*,1)$, then problem \eqref{pro} admits a solution with $5$ bubbles, one positive at the origin and $4$ negative. But the following proposition shows that $t_0$ does not exist.
\end{remark}

\begin{proposition}\label{k=5-impossible-proposition}
For any $t\in(t^*,1), N\ge7$, it always holds that $\iota_2(t)>0$.
\end{proposition}
\begin{proof}
We first show that  $t^*>\frac{\sqrt{6}-\sqrt{2}}{2}$, where $t^*$ is from  \eqref{equ-1-proof-theorem-ball-k bubbles-2}. In order to see that we prove
$\ga_2\left(\frac{\sqrt{6}-\sqrt{2}}{2}\right) < 0$. Since $2^{2/5}\cdot2(\frac{\sqrt{6}-\sqrt{2}}{2})^2<1<(\frac{\sqrt{6}-\sqrt{2}}{2})^4+1$, we have
\[
\frac{1}{(\sqrt{2}\cdot\frac{\sqrt{6}-\sqrt{2}}{2})^{N-2}}
 > \frac{2}{((\frac{\sqrt{6}-\sqrt{2}}{2})^4+1)^{\frac{N-2}{2}}},
\quad\text{for all } N\ge7.
\]
On the other hand, it is easy to see that
\[
\frac{1}{(1-(\frac{\sqrt{6}-\sqrt{2}}{2})^2)^{N-2}}
 = \frac{1}{(\sqrt{2}(\frac{\sqrt{6}-\sqrt{2}}{2}))^{N-2}}, \quad
\frac{1}{(2(\frac{\sqrt{6}-\sqrt{2}}{2}))^{N-2}}
 > \frac{1}{((\frac{\sqrt{6}-\sqrt{2}}{2})^2+1)^{N-2}},
\]
and we conclude $\ga_2(\frac{\sqrt{6}-\sqrt{2}}{2}) < 0$.

Now we prove $\iota_2(t)>0$ for $t\in(t^*,1)\subset(\frac{\sqrt{6}-\sqrt{2}}{2},1)$. It
is easy to see that for $t\in\left(\frac{\sqrt{6}-\sqrt{2}}{2},1\right)$ there holds
\[%begin{equation}\label{equ-1-k=5-impossible-proposition}
\ga'_2(t) \ge (N-2)\cdot\frac{2t}{(1-t^2)^{N-1}},\quad
\ga_2(t) \le \frac{1}{(1-t^2)^{N-2}}.
\]%end{equation}
Then for all $t\in (t^*,1)$ and $N\ge7$:
\begin{eqnarray*}
\frac{\iota_2(t)}{N-2}
 &\ge& \frac{2t}{(1-t^2)^{N-1}}
       - 3\left(\frac{1}{t^{N-2}}-1\right)\cdot\frac{1}{t^{N-1}}
       \left(\sqrt{1+\frac{4\cdot\frac{1}{(1-t^2)^{N-2}}}{9(\frac{1}{t^{N-2}}-1)^2}}-1\right)\\
&\ge& \frac{2t}{(1-t^2)^{N-1}} - \frac{3}{t^{2N-3}}\cdot
      \left(\sqrt{1+\frac{4}{9}(\frac{t^2}{1-t^2})^{N-2}\cdot\frac{1}{(1-t^{N-2})^2}}-1\right).
\end{eqnarray*}
Setting $T:=\frac{t^2}{1-t^2}$ it is enough to prove that
\[%begin{equation}%\label{equ-2-k=5-impossible-proposition}
\frac23\cdot T^{N-1}+1
 > \sqrt{1+\frac{4}{9}\cdot T^{N-2}\cdot\frac{1}{(1-t^{N-2})^2}},
\]%end{equation}
which is equivalent to
\begin{equation}\label{equ-3-k=5-impossible-proposition}
(T^N+3T)\cdot(1-t^{N-2})^2>1.
\end{equation}
It is easy to see that
\begin{equation}\label{equ-4-k=5-impossible-proposition}
3T\cdot(1-t^{N-2})^2 \ge 3(1-t^5)^2>1
\quad\text{if } t\in{\textstyle[\frac{1}{\sqrt{2}},\frac45)}
\end{equation}
and
\begin{equation}\label{equ-5-k=5-impossible-proposition}
T^N\cdot(1-t^{N-2})^2
 \ge T^N\cdot(1-t)^2=\frac{t^4}{(1+t)^2}\cdot T^{N-2}
 > \frac{(\frac{4}{5})^4}{4}(\frac{(\frac{4}{5})^2}{1-(\frac{4}{5})^2})^5
 > 1
\text{ if } t\in[\frac{4}{5},1).
\end{equation}

Now it is left to prove (\ref{equ-3-k=5-impossible-proposition}) for $t\in\left(t^*,\frac{1}{\sqrt{2}}\right)$. First of all, if
$t\in\left(\frac{\sqrt{6}-\sqrt{2}}{2},\frac{1}{\sqrt{2}}\right)$, then $T\in\left(\frac{\sqrt{3}-1}{2},1\right)$. Setting
\[
f(T):=3T(1-t^{N-2})^2=3T(1-(\frac{T}{1+T})^{\frac{N-2}{2}})^2,
\]
a direct computation shows that
\begin{eqnarray*}
f'(T)
 &=& \left(1-\left(\frac{T}{1+T}\right)^{\frac{N-2}{2}}\right)
          \left(3-3\left(\frac{T}{1+T}\right)^{\frac{N-2}{2}}
      - 3(N-2)\left(\frac{T}{1+T}\right)^{\frac{N-2}{2}}\frac{1}{1+T}\right)\\
&\ge& \left(1-\left(\frac12\right)^{\frac{N-2}{2}}\right)
       \left(3-3\left(\frac12\right)^{\frac{N-2}{2}}
       -3(N-2)\left(\frac12\right)^{\frac{N-2}{2}}\frac{1}{1+\frac{\sqrt{3}-1}{2}}\right)\\
&\ge& \left(1-\left(\frac12\right)^{\frac{N-2}{2}}\right)
      \left(3-3\left(\frac12\right)^{\frac52}
       -15\left(\frac12\right)^{\frac52}\frac{1}{1+\frac{\sqrt{3}-1}{2}}\right)
 > 0,
\end{eqnarray*}
where in the second inequality we use the fact that $3-3(\frac{1}{2})^{\frac{N-2}{2}}-3(N-2)(\frac{1}{2})^{\frac{N-2}{2}}\frac{1}{1+\frac{\sqrt{3}-1}{2}}$ is increasing in $N$. Now we conclude that
\begin{eqnarray}\label{equ-6-k=5-impossible-proposition}
f(T)
 > 3\cdot\frac{\sqrt{3}-1}{2}
   \left(1-\left(\frac{\frac{\sqrt{3}-1}{2}}{1+\frac{\sqrt{3}-1}{2}}\right)^{\frac52}\right)^2
 > 1\quad\text{for all } N\ge7.
\end{eqnarray}
Combining  \eqref{equ-3-k=5-impossible-proposition}, \eqref{equ-4-k=5-impossible-proposition}, \eqref{equ-5-k=5-impossible-proposition}, and \eqref{equ-6-k=5-impossible-proposition}, the proof is finished.
\end{proof}

\begin{remark} Let $\mu=\mu_0\eps^\al$, $\mu_0>0$, $\al>\frac{N-4}{N-2}$, $N$ large enough. For $k=2,3$, we can also consider smooth bounded domains with the following symmetry condition.
\begin{itemize}
\item[$(S_2)$] If $(\wt{x},x')\in\Om\subset\R^2\times\R^{N-2}$ then
$(e^{2\pi\sqrt{-1}/k}\wt{x},x')\in\Om$, and $(\wt{x},-x')\in\Om$.
\end{itemize}

We assume further that

\begin{itemize}
\item[$(S_3)$] The map $\vphi_k:\Om^{(2)}:=\{\wt{x}\in\R^2:(\wt{x},0)\in\Om\}\to\R$, defined by
  \begin{eqnarray*}
  \vphi_k(\wt{x}):=\frac{\al_{k-1}\left((\wt{x},0),(e^{2\pi\sqrt{-1}/k}\wt{x},0)\right)}
           {\be^{\frac{k+1}{2}}_{k-1}\left((\wt{x},0),(e^{2\pi\sqrt{-1}/k}\wt{x},0)\right)},
\end{eqnarray*}
admits a nondegenerate critical point at $\wt{\xi}^*\in\Om^{(2)}$ such that
    \[
    H\left((\wt{\xi}^*,0),(\wt{\xi}^*,0)\right)
     - (k-1)G\left((\wt{\xi}^*,0),(-\wt{\xi}^*,0)\right)
     > 0.
    \]
\end{itemize}
Then for $k=2,3$, and using the same methods as in Theorem \ref{theorem-ball-k bubbles}, one can show that there exists $\eps_0>0$ such that for any $\eps\in(0,\eps_0)$, there exists a pair of solutions $\pm u_\eps$ to problem \eqref{pro} satisfying
\[%begin{equation}\label{shape of solution-k bubbles}
u_\eps(x)
 = C_\mu\left(\frac{\si^\eps}{(\si^\eps)^2|x|^{\be_1}+|x|^{\be_2}}\right)^{\frac{N-2}{2}}
 -C_0\sum_i^k(\frac{\de^\eps}{(\de^\eps)^2 + |x-\xi^\eps_i|^2})^{\frac{N-2}{2}}
   + o(1),
\]%end{equation}
where $\de^\eps=\la^\eps\eps^{\frac{1}{N-2}}$,
$\xi^\eps_i=(e^{2\pi i\sqrt{-1}/k}\wt{\xi}^\eps,0)$, $\wt{\xi}^\eps\in\Om^{(2)}$,
$\si^\eps=\ov{\la}^\eps\eps^{\frac{1}{N-2}}$,
and for some $\eta>0$ small enough, $\eta<|\wt{\xi}^\eps|<1-\eta$,
$\la^\eps,\ov{\la}^\eps\in(\eta,\frac{1}{\eta})$. Moreover
$\wt{\xi}^\eps\to\wt{\xi}^*$ as $\eps\to0$.

The question is what kind of domain, besides $B(0,1)$, satisfies the assumption $(S_3)$?
\end{remark}

%%---------------------------------------------------------------------------------
\subsection{\textbf{Proof of Theorem \ref{theorem-ball-k bubbles-add}\label{Subsection 3.3}}}

In this part, we turn to solutions of the form $V_{\eps,\la,\xi}=\sum\limits_{i=1}^{k}(-1)^iPU_{\de_i,\xi_i}+PV_\si$. Then the reduced function in Lemma \ref{lemma-expansion
of J-1} becomes
\[%begin{equation}%\label{equ-proposition-reducement-add1}
\begin{aligned}
\wt{\psi}(\la,\xi)
 &= b_1\left(H(0,0)\ov{\la}^{N-2}+\sum_{i=1}^kH(\xi_i,\xi_i)\la_i^{N-2}
    + 2\sum_{i=1}^k(-1)^{i-1}G(\xi_i,0)\la_i^{\frac{N-2}{2}}\ov{\la}^{\frac{N-2}{2}}\right.
\\
&\hspace{1cm}
 \left.
  +2\sum_{i,j=1,i<j}^k(-1)^{i+j-1}G(\xi_i,\xi_j)\la_i^{\frac{N-2}2}\la_j^{\frac{N-2}2}\right)
          -b_2\ln(\la_1\la_2\dots\la_{k}\ov{\la})^{\frac{N-2}{2}},
\end{aligned}
\]%end{equation}
where $b_1, b_2$ are as in Lemma~\ref{lemma-expansion of J-1}.

\begin{altproof}{\ref{theorem-ball-k bubbles-add}}
%\noindent\textbf{Proof of Theorem \ref{theorem-ball-k bubbles-add}.}
Here we have $k=4$. Due to the ansatz \eqref{equ-1-lemma blowing up speed} we may assume
$\xi_1=(t,0,\dots,0)$, $\xi_2=(0,t,0,\dots,0)$, $\xi_3=(-t,0,0,\dots,0)$, $\xi_4=(0,-t,0,\dots,0)$, $0<t<1$. It is obvious that
\[
H(\xi_1,\xi_1)=H(\xi_2,\xi_2)=H(\xi_3,\xi_3)=H(\xi_4,\xi_4),\quad
G(\xi_1,0)=G(\xi_2,0)=G(\xi_3,0)=G(\xi_4,0),
\]
and
\[
G(\xi_1,\xi_2)=G(\xi_2,\xi_3)=G(\xi_3,\xi_4)=G(\xi_4,\xi_1),\quad
G(\xi_1,\xi_3)=G(\xi_2,\xi_4).
\]

As in the proof of Lemma \ref{lemma-blowing up speed} we have
\[%begin{equation}\label{equ-1-ball-k=5-theorem}
\la_1=\la_3, \la_2=\la_4,
\]%end{equation}
which allows us to consider the function
\begin{eqnarray*}
f_4(\la_1,\la_2,\ov{\la},t)
&=& b_1\left(H(0,0)\ov{\la}^{N-2} + 2H(\xi_1,\xi_1)(\la_1^{N-2}+\la_2^{N-2})
    +4G(\xi_1,0)\left(\la_1^{\frac{N-2}{2}}-\la_2^{\frac{N-2}{2}}\right)
       \ov{\la}^{\frac{N-2}{2}}\right.\\
&&\left.
    + 8G(\xi_1,\xi_2)\la_1^{\frac{N-2}{2}}\la_2^{\frac{N-2}{2}}
    - 2G(\xi_1,\xi_3)\la_1^{N-2}-2G(\xi_1,\xi_3)\la_2^{N-2}\right)
    - b_2\ln\left(\la_1^2\la_2^2\ov{\la}\right)^{\frac{N-2}{2}}.
\end{eqnarray*}

Now suppose $\nabla_{\la_1,\la_2,\ov{\la}} f_4(\la_1,\la_2,\ov{\la},t)=0$. Then we have
\begin{eqnarray}\label{equ-2-ball-k=5-theorem}
H(0,0)\ov{\la}^{N-2} + 2G(\xi_1,0)(\la_1^{\frac{N-2}{2}}
    -\la_2^{\frac{N-2}{2}})\ov{\la}^{\frac{N-2}{2}}
&=&\frac{b_2}{2b_1},\\
\label{equ-3-ball-k=5-theorem}
(H(\xi_1,\xi_1) - G(\xi_1,\xi_3))\la_1^{N-2}
    + G(\xi_1,0)\la_1^{\frac{N-2}{2}}\ov{\la}^{\frac{N-2}{2}}
    + 2G(\xi_1,\xi_2)\la_1^{\frac{N-2}{2}}\la_2^{\frac{N-2}{2}}
&=& \frac{b_2}{2b_1},\\
\label{equ-4-ball-k=5-theorem}
(H(\xi_1,\xi_1) - G(\xi_1,\xi_3))\la_2^{N-2}
    - G(\xi_1,0)\la_2^{\frac{N-2}{2}}\ov{\la}^{\frac{N-2}{2}}
    + 2G(\xi_1,\xi_2)\la_1^{\frac{N-2}{2}}\la_2^{\frac{N-2}{2}}
&=& \frac{b_2}{2b_1}.
\end{eqnarray}
From \eqref{equ-3-ball-k=5-theorem} and  \eqref{equ-4-ball-k=5-theorem} we deduce
\begin{equation}\label{equ-5-ball-k=5-theorem}
\la_2^{\frac{N-2}{2}}-\la_1^{\frac{N-2}{2}}
 = \frac{G(\xi_1,0)}{H(\xi_1,\xi_1)-G(\xi_1,\xi_3)}\ov{\la}^{\frac{N-2}{2}},
\end{equation}
which combined with \eqref{equ-2-ball-k=5-theorem} implies:
\begin{equation}\label{equ-6-ball-k=5-theorem}
\ov{\la}^{N-2}
 = \frac{H(\xi_1,\xi_1)-G(\xi_1,\xi_3)}
     {H(\xi_1,\xi_1)-G(\xi_1,\xi_3)-2G^2(\xi_1,0)}\cdot\frac{b_2}{2b_1}.
\end{equation}
As a consequence of \eqref{equ-5-ball-k=5-theorem} we get
\[
\la_1^{N-2}+\la_2^{N-2}-2\la_1^{\frac{N-2}{2}}\la_2^{\frac{N-2}{2}}
 = \left(\frac{G(\xi_1,0)}{H(\xi_1,\xi_1)-G(\xi_1,\xi_3)}\right)^2\ov{\la}^{N-2},
\]
and then \eqref{equ-3-ball-k=5-theorem} and \eqref{equ-4-ball-k=5-theorem} yield:
\begin{equation}\label{equ-7-ball-k=5-theorem}
\la_1^{\frac{N-2}{2}}\la_2^{\frac{N-2}{2}}
 = \frac{1}{H(\xi_1,\xi_1)-G(\xi_1,\xi_3)+2G(\xi_1,\xi_2)}\cdot\frac{b_2}{2b_1}
\end{equation}
and
\begin{equation}\label{equ-8-ball-k=5-theorem}
\la_1^{N-2}+\la_2^{N-2}
 = \frac{1}{H(\xi_1,\xi_1) - G(\xi_1,\xi_3)+2G(\xi_1,\xi_2)}\cdot\frac{b_2}{b_1}
    +\left(\frac{G(\xi_1,0)}{H(\xi_1,\xi_1)-G(\xi_1,\xi_3)}\right)^2\ov{\la}^{N-2}.
\end{equation}

Let $\tau_1(t)$ be as above, and set
\[
\ga_3(t)
 := H(\xi_1,\xi_1)-G(\xi_1,\xi_3)
  = \frac{1}{(1-t^2)^{N-2}}-\frac{1}{(2t)^{N-2}}+\frac{1}{(t^2+1)^{N-2}}\\
\]
and
\[
\ga_4(t) := G(\xi_1,\xi_2) = \frac{1}{(\sqrt{2}t)^{N-2}}-\frac{1}{(t^4+1)^{\frac{N-2}{2}}}.
\]
A direct computation shows that $\ga'_3(t)>0$, $\ga_3(t)\to -\infty$ as $t\to 0^+$, $\ga_3(t)\to +\infty$ as $t\to 1^-$, and $\ga_3(\frac{1}{2})>0$. Thus there exists $t_1^*\in(0,\frac{1}{2})$ such that
\[
\ga_3(t_1^*)=0,\quad\ga_3(t)<0 \text{ for all } t\in(0,t_1^*).
\]
On the other hand, $(\ga_3(t)-2\tau_1^2(t))'>0$, $\ga_3(t)-2\tau_1^2(t)\to -\infty$ as $t\to0^+$, $\ga_3(t)-2\tau_1^2(t)\to +\infty$ as $t\to 1^-$, and $\ga_3(\frac{1}{2})-2\tau_1^2(\frac{1}{2})<0$. Thus there exists $t_2^*\in(\frac{1}{2},1)$ such that
\[
\ga_3(t_2^*)-2\tau_1^2(t_2^*)=0,\quad\ga_3(t)-2\tau_1^2(t)>0 \text{ for all } t\in(t_2^*,1).
\]
It follows that for every $t\in(0,t_1^*)\cup(t_2^*,1)$ there exist unique $\la_1(t)$, $\la_2(t)$, $\ov{\la}(t)$ such that
\[
\nabla_{\la_1,\la_2,\ov{\la}} f_4(\la_1(t), \la_2(t),\ov{\la}(t),t)=0,
\]
where $\la_1(t), \la_2(t),\ov{\la}(t)$ satisfy \eqref{equ-5-ball-k=5-theorem}, \eqref{equ-6-ball-k=5-theorem}, \eqref{equ-7-ball-k=5-theorem} and  \eqref{equ-8-ball-k=5-theorem}.
Moreover, a direct computation using  \eqref{equ-2-ball-k=5-theorem}, \eqref{equ-3-ball-k=5-theorem}, and \eqref{equ-4-ball-k=5-theorem} shows that
\[
\begin{aligned}
\frac{\pa^2 f_4(\la_1,\la_2,\ov{\la},t)}{\pa\la_1^2}
 &= (N-2)b_1\left(2(N-3)\ga_3(t)\la_1^{N-4}
     +(N-4)\tau_1(t)\la_1^{\frac{N-6}{2}}\ov{\la}^{\frac{N-2}{2}}\right.\\
 &\hspace{1cm}
     \left. + 2(N-4)\ga_4(t)\la_1^{\frac{N-6}{2}}\la_2^{\frac{N-2}{2}}\right)
       +\frac{(N-2)b_2}{\la_1^2}\\
 &= (N-2)^2b_1\left(2\ga_3(t)\la_1^{N-4}
     + \tau_1(t)\la_1^{\frac{N-6}{2}}\ov{\la}^{\frac{N-2}{2}}
     + 2\ga_4(t)\la_1^{\frac{N-6}{2}}\la_2^{\frac{N-2}{2}}\right),
\end{aligned}
\]
\[
\begin{aligned}
\frac{\pa^2 f_4(\la_1,\la_2,\ov{\la},t)}{\pa\la_2^2}
 &= (N-2)b_1\left(2(N-3)\ga_3(t)\la_2^{N-4}
     - (N-4)\tau_1(t)\la_2^{\frac{N-6}{2}}\ov{\la}^{\frac{N-2}{2}}\right.\\
 &\hspace{1cm}
     \left. + 2(N-4)\ga_4(t)\la_1^{\frac{N-6}{2}}\la_2^{\frac{N-2}{2}}\right)
       + \frac{(N-2)b_2}{\la_2^2}\\
 &= (N-2)^2b_1\left(2\ga_3(t)\la_2^{N-4}
     - \tau_1(t)\la_2^{\frac{N-6}{2}}\ov{\la}^{\frac{N-2}{2}}
     + 2\ga_4(t)\la_1^{\frac{N-2}{2}}\la_2^{\frac{N-6}{2}}\right),
\end{aligned}
\]
\[
\begin{aligned}
\frac{\pa^2 f_4(\la_1,\la_2,\ov{\la},t)}{\pa\ov{\la}^2}
 &= (N-2)b_1\left((N-3)H(0,0)\ov{\la}^{N-4} + (N-4)\tau_1(t)(\la_1^{\frac{N-2}{2}}
      - \la_2^{\frac{N-2}{2}})\ov{\la}^{\frac{N-6}{2}}\right)\\
 &\hspace{1cm}
      + \frac{(N-2)b_2}{2\ov{\la}^2}\\
 &= (N-2)^2b_1\left(H(0,0)\ov{\la}^{N-4} + \tau_1(t)(\la_1^{\frac{N-2}{2}}
      - \la_2^{\frac{N-2}{2}})\ov{\la}^{\frac{N-6}{2}}\right),
\end{aligned}
\]
\[
\frac{\pa^2 f_4(\la_1,\la_2,\ov{\la},t)}{\pa\ov{\la}\pa\la_1}
 = (N-2)^2b_1\tau_1(t)\la_1^{\frac{N-4}{2}}\ov{\la}^{\frac{N-4}{2}},
\]
\[
\frac{\pa^2
 f_4(\la_1,\la_2,\ov{\la},t)}{\pa\ov{\la}\pa
\la_2}=-(N-2)^2b_1\tau_1(t)\la_2^{\frac{N-4}{2}}\ov{\la}^{\frac{N-4}{2}},
\]
\[
\frac{\pa^2 f_4(\la_1,\la_2,\ov{\la},t)}{\pa\la_1\pa\la_2}
 = 2(N-2)^2b_1\ga_4(t)\la_1^{\frac{N-4}{2}}\la_2^{\frac{N-4}{2}}.
\]
For simplicity, we introduce the notation
\[
X:=\ov{\la}^{\frac{N-2}{2}},\quad Y:=\la_1^{\frac{N-2}{2}},\quad Z:=\la_2^{\frac{N-2}{2}}.
\]
In order to show that the Hessian matrix
$D^2_{\la_1,\la_2,\ov{\la}}f_4(\la_1,\la_2,\ov{\la},t)$ is nondegenerate for any $t\in(0,t_1^*)\cup(t_2^*,1)$, it suffices to show that the matrix
\[
\left(
\begin{array}{ccc}
X+\tau_1(t)(Y-Z) & \tau_1(t)Y^\frac{2}{N-2}X^{\frac{N-4}{N-2}}
 & -\tau_1(t)Z^\frac{2}{N-2}X^{\frac{N-4}{N-2}}\\
\tau_1(t)Y^{\frac{N-4}{N-2}}X^\frac{2}{N-2}  & 2\ga_3(t)Y+\tau_1(t)X+2\ga_4(t)Z
 & 2\ga_4(t)Y^{\frac{N-4}{N-2}}Z^\frac{2}{N-2}\\
-\tau_1(t)Z^{\frac{N-4}{N-2}}X^\frac{2}{N-2} & 2\ga_4(t)Y^\frac{2}{N-2}Z^{\frac{N-4}{N-2}}
 & 2\ga_3(t)Z-\tau_1(t)X+2\ga_4(t)Y \\
\end{array}
\right)
\]
is nondegenerate. Using \eqref{equ-2-ball-k=5-theorem}, \eqref{equ-3-ball-k=5-theorem} and \eqref{equ-4-ball-k=5-theorem} this is equivalent to show that the matrix
\[
\left(
\begin{array}{ccc}
\frac{X}{2}+\frac{b_2}{4b_1}\cdot\frac{1}{X}& \tau_1(t)Y^\frac{2}{N-2}X^{\frac{N-4}{N-2}} & -\tau_1(t)Z^\frac{2}{N-2}X^{\frac{N-4}{N-2}}\\
\tau_1(t)Y^{\frac{N-4}{N-2}}X^\frac{2}{N-2}  & \ga_3(t)Y+\frac{b_2}{2b_1}\cdot\frac{1}{Y} & 2\ga_4(t)Y^{\frac{N-4}{N-2}}Z^\frac{2}{N-2}\\
-\tau_1(t)Z^{\frac{N-4}{N-2}}X^\frac{2}{N-2} & 2\ga_4(t)Y^\frac{2}{N-2}Z^{\frac{N-4}{N-2}} & \ga_3(t)Z+\frac{b_2}{2b_1}\cdot\frac{1}{Z} \\
\end{array}
\right)
\]
is nondegenerate. A direct computation, using \eqref{equ-7-ball-k=5-theorem}, shows that the determinant of the above matrix has the same sign as $\ga_3(t)$, and hence is nondegenerate.

Now in order to finish the proof, we look for $t_0\in(0,t_1^*)$ such that $\nu_2'(t_0)=0$, where
\[
\nu_2(t):=f_4\big(\la_1(t),\la_2(t),\ov{\la}(t),t\big).
\]
Observe that
\[
\begin{aligned}
\nu_2'(t) &= \frac{\pa f_4(\la_1(t),\la_2(t),\ov{\la}(t),t)}{\pa t}\\
 &= 2b_1\left(\ga'_3(t)\left(\la_1^{N-2} + \la_2^{N-2}\right)
    + 2\tau'_1(t)\left(\la_1^{\frac{N-2}2}-\la_2^{\frac{N-2}2}\right)\ov{\la}^{\frac{N-2}2}
    + 4\ga'_4(t)\la_1^{\frac{N-2}{2}}\la_2^{\frac{N-2}{2}}\right)
\end{aligned}
\]
where $\la_1,\la_2,\ov{\la}$ satisfy \eqref{equ-5-ball-k=5-theorem}, \eqref{equ-6-ball-k=5-theorem}, \eqref{equ-7-ball-k=5-theorem} and \eqref{equ-8-ball-k=5-theorem}. Therefore, $\nu_2'(t)=0$ for $t\in(0,t_1^*)$ is equivalent to
\[
\begin{aligned}
\iota_3(t)
 &:= \ga'_3(t)\big(2\ga_3(t)(\ga_3(t) - 2\tau_1^2(t)) + \tau_1^2(t)(\ga_3(t)+2\ga_4(t))\big)
       - 2\tau'_1(t)\tau_1(t)\ga_3(t)\big(\ga_3(t)+2\ga_4(t)\big)\\
 &\hspace{1cm}
       + 4\ga'_4(t)\ga_3(t)\big(\ga_3(t)-2\tau_1^2(t)\big)\\
 &= 0.
\end{aligned}
\]
It is easy to check that $\iota_3(t)\to-\infty$ as $t\to 0^+$ and $\iota_3(t_1^*)>0$ since $\ga'_3(t_1^*)>0$, $\ga_4(t_1^*)>0$ and $\ga_3(t_1^*)=0$. Hence there exists $t_0\in(0,t_1^*)$ such that $\iota_3(t_0)=0$, which finishes the proof.
\end{altproof}

\begin{remark}
It seems that there also should exist $t_0\in (t_2^*,1)$ such that $\iota_3(t_0)=0$. This is not considered here because the computations get enormous.
\end{remark}

%%---------------------------------------------------------------------------------
\section{\textbf{Solutions with tower of bubbles concentrating at the origin}\label{Section 4}}

In this section we prove Theorem~\ref{theorem-tower of bubble} where $\al=1$. We use the same notations in similar settings as Section \ref{Section 3}.

%%---------------------------------------------------------------------------------
\subsection{\textbf{The finite dimensional reduction}\label{Subsection 4.1}}
We fix an integer $k\ge0$. For $\la = (\la_1,\ldots,\la_k,\ov{\la})\in \R^{k+1}$ we set $\de_i = \la_i\eps^{\frac{2i-1}{N-2}}$, for $i=1,\ldots,k,$ and consider
$\ze = (\ze_1,\ldots,\ze_k)\in(\R^N)^k$ such that
\[
\xi = (\xi_1,\ldots,\xi_{k})=(\de_1\ze_1,\ldots,\de_k\ze_{k}) \in \Om^k.
\]
We also set $\si=\ov{\la}\eps^{\frac{2(k+1)-1}{N-2}}$. Now we define for $\eta\in(0,1)$:
\[%begin{eqnarray}%\label{subspace-O-2}
\cO_\eta
 = \left\{(\la,\ze)\in\R_+^{k+1}\times (\R^N)^{k}:
         \la_i\in(\eta,\eta^{-1}),\ \ov{\la}\in(\eta,\eta^{-1}),\ |\ze_i|\le\frac1\eta,\
         i=1,\dots,k\right\}.
\]%end{eqnarray}
We also recall the sets $W_{\eps,\la,\xi}$, $K_{\eps,\la,\xi}$, $K^\bot_{\eps,\la,\xi}$, and the projections $\Pi_{\eps,\la,\xi}$, $\Pi^\bot_{\eps,\la,\xi}$ as in Section~\ref{Section 3}. Now we want to find $\eta>0$ ,$\eps>0$, $(\la,\ze)\in\cO_\eta$ and
$\phi_{\eps,\la,\xi}\in K^\bot_{\eps,\la,\xi}$ such that:
\begin{eqnarray}
\label{main-equality Com-2}
\Pi^\bot_{\eps,\la,\xi}\left(V_{\eps,\la,\xi} + \phi_{\eps,\la,\xi}
      - \iota^*(f_\eps(V_{\eps,\la,\xi} + \phi_{\eps,\la,\xi}))\right)
 =0,\\
\label{main-equality-2}
\Pi_{\eps,\la,\xi}\left(V_{\eps,\la,\xi} + \phi_{\eps,\la,\xi}
      - \iota^*(f_\eps(V_{\eps,\la,\xi} + \phi_{\eps,\la,\xi}))\right)
 = 0,
\end{eqnarray}
where
\begin{equation}\label{shape of V-2}
V_{\eps,\la,\xi}=\sum_{i=1}^{k}(-1)^{i-1}PU_{\de_i,\xi_i}+(-1)^k PV_\si.
\end{equation}
Next we take $\rho>0$ small enough and let
\begin{equation}\label{shape of domain Ai-2}
A_{k+1}:=B(0,\sqrt{\de_{k+1}\de_k}),\quad A_i:=B(0,\sqrt{\de_i\de_{i-1}})\setminus B(0,\sqrt{\de_i\de_{i+1}})\ \text{ for } i=1,\ldots,k;
\end{equation}
here $\de_0=\frac{\rho^2}{\de_1}$, $\de_{k+1}=\si$; cf.~ \cite{MussoPis10JMPA}.
Finally recall the operator $L_{\eps,\la,\xi}$ from Section~\ref{Section 3}.

Now we first solve \eqref{main-equality Com-2}.

\begin{proposition}\label{proposition-operator L-2}
For any $\eta>0$, there exist $\eps_0>0$ and $c>0$ such that for every $(\la,\ze)\in\cO_\eta$, and for every $\eps\in(0,\eps_0)$:
\begin{equation}\label{inequality operator L-2}
\|L_{\eps,\la,\xi}(\phi)\|_{\mu}\ge c\|\phi\|_\mu
\quad \text{for all } \phi\in K^\bot_{\eps,\la,\xi}.
\end{equation}
Consequently, $L_{\eps,\la,\xi}$ is invertible with continuous inverse.
\end{proposition}

\begin{proof}
Arguing by contradiction, we assume that there exist $\eta>0$, sequences
$\eps^n>0$, $(\la^n,\ze^n)\in\cO_\eta$, $\phi^n\in H_\mu(\Om)$ with
$\eps^n\to0$, $\la_i^n\to\la_i$, $\ov{\la}^n\to\ov{\la}$, $\ze_i^n\to\ze_i$, as $n\to\infty$
and such that
\begin{equation}\label{equ-1-proposition-operator L-2}
\phi^n\in K^\bot_{\eps^n,\la^n,\xi^n},\quad \|\phi^n\|_\mu=1,
\end{equation}
and
\begin{equation}\label{equ-2-proposition-operator L-2}
L_{\eps^n,\la^n,\xi^n}(\phi^n)=h^n\ \text{ with } \|h^n\|_\mu\to 0;
\end{equation}
here $\la^n=(\la_1^{n},\ldots,\la_{k}^{n},\ov{\la}^n)$, $\ze^n=(\ze_1^{n},\ldots,\ze_{k}^{n})$, $\xi^n = (\xi_1^n,\ldots,\xi_{k}^n) = (\de_1^n\ze_1^n,\de_2^n\ze_2^n,\dots,\de_k^n\ze_{k}^n)
 \in\Om^{k}$, $\de_i^n=\la_i^n\eps^{\frac{2i-1}{N-2}}$ for $i=1,2,\dots,k$,
$\si^n=\ov{\la}^n\eps^{\frac{2(k+1)-1}{N-2}}$. We need the sets
\[
\textstyle
A_{k+1}^n:=B\left(0,\sqrt{\de_{k+1}^n\de_k^n}\right),\ \
A_i^n := B\left(0,\sqrt{\de_i^n\de_{i-1}^n}\right)\setminus
            B\left(0,\sqrt{\de_i^n\de_{i+1}^n}\right),\ i=1,2,\dots,k,
\]
where $\de_0^n:=\frac{\rho^2}{\de_1^n}$, $\de_{k+1}^n:=\si^n$.

Thus we have:
\begin{equation}\label{equ-3-proposition-operator L-2}
\phi^n-\iota^*\left(f'_0(V_{\eps^n,\la^n,\xi^n})\phi^n\right)
 = h^n-\Pi_{\eps^n,\la^n,\xi^n}\left(\iota^*(f'_0(V_{\eps^n,\la^n,\xi^n})\phi^n)\right).
\end{equation}
Then we obtain as in Proposition \ref{proposition-operator L-1}
\[
w^n
 := -\Pi_{\eps^n,\la^n,\xi^n}(\iota^*(f'_0(V_{\eps^n,\la^n,\xi^n})\phi^n))
  = \sum_{i=1}^{k}\sum_{j=0}^N c_{i,j}^nP(\Psi_i^j)_n+c_0^n P(\ov{\Psi})_n
\]
for some coefficients $c_{i,j}^n$, $c_0^n$, where $(\Psi_i^j)_n$, $j=1,\ldots,N,(\Psi_i^0)_n$,
and $(\ov{\Psi})_n$ are defined as in the proof of Proposition~\ref{proposition-operator L-1}.

\emph{Step 1.} We claim that
\begin{equation}\label{equ-wn-proposition-operator L-2}
 \lim_{n\to\infty}\|w^n\|_\mu=0.
\end{equation}
Multiplying \eqref{equ-3-proposition-operator L-2} by
$\De P(\Psi_l^h)_n + \mu\frac{P(\Psi_l^h)_n}{|x|^2}$, using Lemma~\ref{e56-e62-Lemma B},
Lemma~\ref{e50-e51}, Lemma~\ref{e83-85-Lemma B}, and arguing as in the proof of Proposition~\ref{proposition-operator L-1}, we deduce $c_{l,h}^n\to0$, for $l=1,\dots,k$, $h=0,1,\ldots,N$, and  $c_0^n\to0$, as $n\to\infty$.
The claim $\lim\limits_{n\to\infty}\|w^n\|_\mu=0$ follows.

\emph{Step 2.} As in \cite{MussoPis10JMPA}, we use cut-off functions $\chi_i^n$, $i=1,\dots,k+1$, with the properties
\[
\left\{
\begin{aligned}
&\chi_i^n(x)=1
 \quad \text{if }\sqrt{\de_i^n\de_{i+1}^n} \le |x|\le\sqrt{\de_i^n\de_{i-1}^n}; \\
&\chi_i^n(x)=0
 \quad \text{if }|x| \le \frac{\sqrt{\de_i^n\de_{i+1}^n}}{2}
        \text{ or } |x| \ge 2\sqrt{\de_i^n\de_{i-1}^n};\\
&|\nabla\chi_i^n(x)| \le \frac{1}{\sqrt{\de_i^n\de_{i-1}^n}}
\quad\text{and} |\nabla^2\chi_i^n(x)| \le \frac{4}{\de_i^n\de_{i-1}^n},
\end{aligned}
\right.
\]
for $i=1,\dots,k$, and
\[
\left\{
\begin{aligned}
&\chi_{k+1}^n(x)=1, \quad \text{if } |x|\le\sqrt{\de_{k+1}^n\de_{k}^n}; \\
&\chi_{k+1}^n(x)=0, \quad \text{if } |x|\ge2\sqrt{\de_{k+1}^n\de_{k}^n};\\
&|\nabla\chi_{k+1}^n(x)|\le\frac{1}{\sqrt{\de_{k+1}^n\de_{k}^n}},
\quad\text{and}\quad
|\nabla^2\chi_{k+1}^n(x)|\le \frac{4}{\de_{k+1}^n\de_{k}^n}.
\end{aligned}
\right.
\]
The function $\phi_i^n$ defined by
\[
\phi_i^n(y):=(\de_i^n)^{\frac{N-2}{2}}\phi^n(\de_i^n y)\chi_i^n(\de_i^n y),
\quad\text{for } y\in \Om_i^n:=\frac{\Om}{\de_i^n},\ i=1,\dots, k+1.
\]
is bounded in $D^{1,2}(\R^N)$. Therefore we may assume, up to a subsequence,
\[%begin{equation}\label{equ-5-proposition-operator L-2}
\phi_i^n \weakto \phi_i^\infty\
\text{ weakly in } D^{1,2}(\R^N),\ i=1,2,\dots,k+1.
\]%end{equation}
Now we prove
\begin{equation}\label{equ-6-proposition-operator L-2}
\phi_i^\infty=0\quad\text{for } i=1,\dots,k+1.
\end{equation}
As in Proposition \ref{proposition-operator L-1}, using
\eqref{equ-3-proposition-operator L-2}, \eqref{equ-2-proposition-operator L-2},
\eqref{equ-wn-proposition-operator L-2}, we have for any $\psi\in C_0^\infty(\R^N)$, and for $i=1,\dots,k$:
\[%begin{eqnarray}\label{equ-7-proposition-operator L-2}
\begin{aligned}
\int_{\Om_i^n}\nabla\phi_i^n(y)\nabla\psi(y)
 &= (\de_i^n)^{\frac{2-N}{2}}\int_{\Om}
       \nabla\iota^*\left(f'_0(V_{\eps^n,\la^n,\xi^n}(x))\phi^n(x)\right)
       \nabla\left(\chi_i^n(x)\psi\left(\frac{x}{\de_i^n}\right)\right)
     + o(1)\\
 &= (\de_i^n)^{\frac{2-N}{2}}\int_{\Om}
      f'_0\left(V_{\eps^n,\la^n,\xi^n}(x)\right)
      \phi^n(x)\chi_i^n(x)\psi\left(\frac{x}{\de_i^n}\right)
      + o(1)\\
 &= (\de_i^n)^{2}\int_{\Om_i^n}
      f'_0\left(V_{\eps^n,\la^n,\xi^n}(\de_i^n y)\right) \phi_i^n(y)\psi(y) + o(1)\\
 &= \int_{\R^N}f'_0\left(U_{1,\ze_i}(y)\right) \phi_i^\infty(y)\psi(y) + o(1).
\end{aligned}
\]%end{eqnarray}
Hence $\phi_i^\infty$ is a weak solution of
\[%begin{equation}\label{equ-8-proposition-operator L-2}
-\De\phi_i^\infty=f'_0(U_{1,\ze_i})\phi_i^\infty,~
\text{in}~D^{1,2}(\R^N).
\]%end{equation}
Setting $\Psi_{1,\ze_i}^j:=\frac{\pa U_{1,\ze_i}}{\pa (\ze_i)^j}$, for $j=1,\dots,N$, and $\Psi_{1,\ze_i}^0:=\frac{\pa U_{\de,\ze_i}}{\pa \de}|_{\de=1}$ we obtain as in
\cite[Lemma 3.1]{MussoPis10JMPA}:
\[%begin{equation}\label{equ-9-proposition-operator L-2}
\int_{\R^N}\nabla\phi_i^\infty(x)\nabla\Psi_{1,\ze_i}^j(x)=0,~
j=0,1,2,\dots, N,~ i=1,2,\dots,k.
\]%end{equation}
Then (\ref{equ-6-proposition-operator L-2}) holds for $i=1,\ldots,k$. The proof of $\phi_{k+1}^\infty=0$ is similar.

\emph{Step 3.} A contradiction arises as in Proposition \ref{proposition-operator L-1} and
\cite{MussoPis02IUMJ}.
\end{proof}

\begin{proposition}\label{proposition-estimate of error-2}
For any $\eta>0$, there exist $\eps_0>0$, $c_0>0$ such that for every $(\la,\ze)\in\cO_\eta$ and every $\eps\in(0,\eps_0),$ there exists a unique solution
$\phi_{\eps,\la,\xi}\in K^\bot_{\eps,\la,\xi}$ of equation \eqref{main-equality Com-2}. Moreover, we have
\begin{equation}\label{inequality-estimate of error-2}
\|\phi_{\eps,\la,\xi}\|_{\mu}\le
c_0(\eps^{\frac{N+2}{2(N-2)}}+\eps^{\frac{2k+3}{4}}),
\end{equation}
and the map $\Phi_\eps:\cO_\eta\to K^\bot_{\eps,\la,\xi}$ defined by $\Phi_\eps(\la,\xi):=\phi_{\eps,\la,\xi}$ is of class $C^1$.
\end{proposition}

\begin{proof}  As in \cite{BarMiPis06CVPDE}, we define the operator $T_{\eps,\la,\xi}:K^\bot_{\eps,\la,\xi}\to K^\bot_{\eps,\la,\xi}$ by
\[%begin{equation}\label{equ-1-proposition-estimate of error-2}
T_{\eps,\la,\xi}(\phi)=L^{-1}_{\eps,\la,\xi}\Pi^\bot_{\eps,\la,\xi}
(\iota^*(f_\eps(V_{\eps,\la,\xi}+\phi)-f'_0(V_{\eps,\la,\xi})\phi)-V_{\eps,\la,\xi}).
\]%end{equation}
Now we prove that $T_{\eps,\la,\xi}$ is a contraction mapping. Proposition~\ref{proposition-operator L-2}, \eqref{ineq-adjoint operator} and Lemma~\ref{(M17)-Lemma B} imply as in \eqref{equ-2-proposition-estimate of error-1}:
\[%\label{equ-2-proposition-estimate of error-2}
\begin{aligned}
\|T_{\eps,\la,\xi}(\phi)\|_\mu
 &\le C\|f_\eps(V_{\eps,\la,\xi}+\phi)-f_\eps(V_{\eps,\la,\xi})
          -f'_\eps(V_{\eps,\la,\xi})\phi\|_{2N/(N+2)}\\\nonumber
&\hspace{1cm}
   + C\|(f'_\eps(V_{\eps,\la,\xi})-f'_0(V_{\eps,\la,\xi}))\phi\|_{2N/(N+2)}\\\nonumber
&\hspace{1cm}
   + C\|f_\eps(V_{\eps,\la,\xi})-f_0(V_{\eps,\la,\xi})\|_{2N/(N+2)}\\\nonumber
&\hspace{1cm}
   + C\left\|f_0(V_{\eps,\la,\xi})
        -\left(\sum_{i=1}^{k}(-1)^{i-1}f_0(U_{\de_i,\xi_i})
           +(-1)^kf_0(V_\si)\right)\right\|_{2N/(N+2)}\\\nonumber
&\hspace{1cm}
   + \sum_{i=1}^{k}O(\mu\de_i)+O\left(\left(\mu\si^{\frac{N-2}{2}}\right)^{\frac12}\right).
\end{aligned}
\]
Using Lemma (\ref{e53-e55-B}) and observing that
\[
\|f_\eps(V_{\eps,\la,\xi}+\phi) - f_\eps(V_{\eps,\la,\xi}) -
   f'_\eps(V_{\eps,\la,\xi})\phi\|_{2N/(N+2)}
\le C\|\phi\|_\mu^{2^*-1},
\]
we have
\[
\begin{aligned}
\|T_{\eps,\la,\xi}(\phi)\|_\mu
 &\le C\|\phi\|_\mu^{2^*-1}+C\eps\|\phi\|_\mu
       + C\eps+O\left(\eps^{\frac{N+2}{2(N-2)}}\right)+\sum_{i=1}^k O(\mu\de_i)
       + O\left(\left(\mu\si^{\frac{N-2}{2}}\right)^{\frac12}\right)\\
 &= C\|\phi\|_\mu^{2^*-1} + C\eps\|\phi\|_\mu + O\left(\eps^{\frac{N+2}{2(N-2)}}\right)
     + O\left(\eps^{\frac{2k+3}{4}}\right).
\end{aligned}
\]
The remaining part of the argument is standard.
\end{proof}

For $\la=(\la_1,\ldots,\la_{k},\ov{\la})$ and $\ze=(\ze_1,\ldots,\ze_{k})$ we now consider the reduced functional
\[
I_\eps(\la,\ze)=J_\eps(V_{\eps,\la,\xi}+\phi_{\eps,\la,\xi}).
\]

\begin{proposition}\label{proposition-reducement-2}
If $(\la^0,\ze^0)$ is a critical point of $I_\eps$ then there exists a family of solutions $u_\eps$ to problem \eqref{pro} having the shape
\begin{equation}%\label{equ-proposition-reducement-shape of solutions-2}
u_\eps(x)=V_{\eps,\la,\xi}+\phi_{\eps,\la,\xi},
\end{equation}
where $V_{\eps,\la,\xi}$ is the one stated in (\ref{shape of V-2}).
\end{proposition}

\begin{proof}
We omit the proof because it is similarly to the one of Proposition~\ref{proposition-reducement-1}.
\end{proof}

%%---------------------------------------------------------------------------------
\subsection{\textbf{Proof of Theorem \ref{theorem-tower of bubble}}\label{Subsection 4.2}}
For convenience, we use the notation $\la_{k+1}:=\ov{\la}$ in this subsection.

\begin{lemma}\label{lemma-expansion of J-2}
For $\eps\to0^+$, there holds
\begin{eqnarray}\label{equ-expansion of J-2}
I_\eps(\la,\ze)&=&a_1+a_2\eps-a_3\eps\ln
\eps+\psi(\la,\ze)\eps+o(\eps)
\end{eqnarray}
$C^1$-uniformly with respect to $(\la,\ze)$ in compact sets of $\cO_\eta$. The constants are given by $a_1=\frac{k+1}{N}S_0^{\frac{N}{2}}$,
$a_2 = \frac{(k+1)}{2^*}\int_{\R^N} U_{1,0}^{2^*}\ln U_{1,0}
 -\frac{k+1}{(2^*)^2}S_0^{\frac{N}{2}}-\frac{1}{2}
S_0^{\frac{N-2}{2}}\ov{S}\mu_0$, and
$a_3=\frac{(k+1)^2}{2\cdot2^*}\int_{\R^N} U_{1,0}^{2^*}$.
The function $\psi$ is given by
\[%begin{eqnarray}\label{equ-proposition-reducement-2}
\psi(\la,\ze)
 = b_1\la_{1}^{N-2} +\sum_{i=1}^{k}b_2(\frac{\la_{i+1}}{\la_i})^{\frac{N-2}{2}}h_1(\ze_i)
    -\sum_{i=1}^{k}b_3h_2(\ze_i)-b_4\ln(\la_1\dots\la_{k+1})^{\frac{N-2}{2}},\\
\]
with $b_1=\frac{1}{2}C_0\int_{\R^N}U_{1,0}^{2^*-1}$, $b_2=C_0^{2^*}$,
$b_3=\frac{1}{2}C_0^2\mu_0,~ b_4=\frac{1}{2^*}\int_{\R^N}U_{1,0}^{2^*}$, and
\[
h_1(\ze_i) = \int_{\R^N}\frac{1}{|y+\ze_i|^{N-2}(1+|y|^2)^{\frac{N+2}{2}}},\quad
h_2(\ze_i) = \int_{\R^N}\frac{1}{|y+\ze_i|^2(1+|y|^2)^{N-2}}.
\]%end{eqnarray}
\end{lemma}

\begin{proof} Observe that
\begin{eqnarray}\label{equ1-1-lemma-expansion of J-2}
J_\eps(V_{\eps,\la,\xi})
&=&\frac{1}{2}\int_{\Om}\left(|\nabla V_{\eps,\la,\xi}|^2
    -\mu\frac{|V_{\eps,\la,\xi}|^2}{|x|^2}\right)\\
\label{equ1-2-lemma-expansion of J-2}
&&-\frac{1}{2^*}\int_{\Om}|V_{\eps,\la,\xi}|^{2^*}\\
\label{equ1-3-lemma-expansion of J-2}
&&+\left(\frac{1}{2^*}\int_{\Om}|V_{\eps,\la,\xi}|^{2^*}
-\frac{1}{2^*-\eps}\int_{\Om}|V_{\eps,\la,\xi}|^{2^*-\eps}\right).
\end{eqnarray}
For $k\ge1$ Lemma~\ref{e72-e78} and Lemma~\ref{e44-e48} yield
\[
\begin{aligned}
\eqref{equ1-1-lemma-expansion of J-2}
&=\frac{1}{2}(k+1) S_0^{\frac{N}{2}}-\frac{N}{4} S_0^{\frac{N-2}{2}}\ov{S}\mu_0\eps
    -\frac12 C_0^{2^*}H(0,0)\la_1^{N-2}\int_{\R^N}\frac{1}{(1+|z|^2)^{\frac{N+2}{2}}}
       \cdot\eps\\
&\hspace{1cm}
 -\frac12\sum_{i=1}^{k}\mu_0 C_0^2\int_{\R^N}\frac{1}{|y|^2(1+|y-\ze_i|^2)^{N-2}}\cdot\eps\\
&\hspace{1cm}
 -C_0^{2^*}\left(\frac{\la_{k+1}}{\la_k}\right)^{\frac{N-2}{2}}
   \int_{\R^N}\frac{1}{(1+|y|^2)^{\frac{N+2}{2}}}\cdot\frac{1}{(1+|\ze_k|^2)^{\frac{N-2}{2}}}
     \cdot\eps\\
&\hspace{1cm}
 -\sum_{i=1}^{k-1}C_0^{2^*}\left(\frac{\la_{i+1}}{\la_i}\right)^{\frac{N-2}{2}}
    \int_{\R^N}\frac{1}{(1+|y|^2)^{\frac{N+2}{2}}}\cdot\frac{1}{(1+|\ze_i|^2)^{\frac{N-2}{2}}}
        \cdot\eps + o(\eps).
\end{aligned}
\]
From Lemma \ref{e79} and Lemma \ref{e44-e48} we deduce:
\[
\begin{aligned}
\eqref{equ1-2-lemma-expansion of J-2}
&=-\frac{1}{2^*}(k+1) S_0^{\frac{N}{2}} + \frac{N-2}{4} S_0^{\frac{N-2}{2}}\ov{S}\mu_0\eps
   + C_0^{2^*}H(0,0)\la_1^{N-2}\int_{\R^N} \frac{1}{(1+|z|^2)^{\frac{N+2}{2}}}\cdot\eps\\
&\hspace{1cm}
 +C_0^{2^*}\sum_{i=1}^{k}\left(\frac{\la_{i+1}}{\la_i}\right)^{\frac{N-2}{2}}
    \int_{\R^N}\frac{1}{|y|^{N-2}(1+|y-\ze_i|^2)^{\frac{N+2}{2}}}\cdot\eps\\
&\hspace{1cm}
 +C_0^{2^*}\sum_{i=1}^{k}(\frac{\la_{i+1}}{\la_i})^{\frac{N-2}{2}}
   \int_{\R^N}\frac{1}{(1+|y|^2)^{\frac{N+2}{2}}}\frac{1}{(1+|\ze_i|^2)^{\frac{N-2}{2}}}
     \cdot\eps + o(\eps).
\end{aligned}
\]
By Lemma \ref{e82} and Lemma \ref{e44-e48},
\[
\begin{aligned}
\eqref{equ1-3-lemma-expansion of J-2}
&= -\frac{\eps}{(2^*)^2}(k+1)S_0^{\frac{N}{2}}
   - \frac{(N-2)\eps}{2\cdot2^*}\int_{\R^N} U_{1,0}^{2^*}\cdot\ln(\de_1\dots\de_k\si)\\
&\hspace{1cm}
   + \frac{(k+1)\eps}{2^*}\int_{\R^N} U_{1,0}^{2^*}\ln U_{1,0}+o(\eps)\\
&= - \frac{\eps}{(2^*)^2}(k+1)S_0^{\frac{N}{2}}
   - \frac{(N-2)\eps}{2\cdot2^*}
     \int_{\R^N} U_{1,0}^{2^*}\cdot\ln(\la_1\dots\la_k\ov{\la})\\
&\hspace{1cm}
   - \frac{(k+1)^2}{2\cdot2^*}\int_{\R^N} U_{1,0}^{2^*}\cdot\eps\ln\eps
   + \frac{(k+1)}{2^*}\int_{\R^N} U_{1,0}^{2^*}\ln U_{1,0}\cdot\eps + o(\eps).
\end{aligned}
\]
Using Proposition \ref{proposition-estimate of error-2}, \eqref{equ-1-projection estimate}, \eqref{equ-3-projection estimate}), Lemma~\ref{e53-e55-B}, we get:
\begin{eqnarray}\label{equ1-4-lemma-expansion of J-2}
J_\eps(V_{\eps,\la,\xi}+\phi_{\eps,\la,\xi})-J_\eps(V_{\eps,\la,\xi})=o(\eps).
\end{eqnarray}
Now we conclude the proof for $k\ge1$ by \eqref{equ1-1-lemma-expansion of J-2}, \eqref{equ1-2-lemma-expansion of J-2},\eqref{equ1-3-lemma-expansion of J-2}, and \eqref{equ1-4-lemma-expansion of J-2}.

The case $k=0$ can be easily dealt with using \eqref{equ1-Lemma A-e12,e24,e30,e34}, \eqref{equ1-Lemma A-e4,e41} and \eqref{equ1-4-lemma-expansion of J-2}. Observe here that \eqref{equ-expansion of J-2} holds $C^1$-uniformly with respect to $(\la,\ze)$ in compact sets of $\cO_\eta$; see \cite[Lemma~7.1]{MussoPis10JMPA}.
\end{proof}

%Define
%\begin{eqnarray}%\label{space-finite dimensional-2}
%\mathcal{M}_4=\{(\la,\ze)\in\R_+^{k+1}\times (\R^N)^{k}: \la_i>0,\ov{\la}>0,i=1,2,\dots,k\}.
%\end{eqnarray}

\begin{altproof}{\ref{theorem-tower of bubble}}
By the change of variables
\[%begin{eqnarray}\label{equ-1-proof-theorem-tower of bubble}
\la_1^{\frac{N-2}{2}}
 = s_1,\quad \left(\frac{\la_2}{\la_1}\right)^{\frac{N-2}{2}}=s_2,\ \ldots,\
    \left(\frac{\la_{k+1}}{\la_k}\right)^{\frac{N-2}{2}}=s_{k+1},
\]%end{eqnarray}
$\psi(\la,\ze)$ can be rewritten as
\[%begin{eqnarray}\label{equ-2-proof-theorem-tower of bubble}
\wh{\psi}(s,\ze)
 = b_1 s_{1}^2+\sum_{i=1}^{k}b_2 s_{i+1} h_1(\ze_i)-\sum_{i=1}^{k}b_3 h_2(\ze_i)
    - b_4\ln(s_1^{k+1}s_2^{k}\dots s_{k+1}),
\]%end{eqnarray}
where $s=(s_1,s_2,\ldots,s_{k+1})$.

Suppose $\frac{\wh{\psi}(s,\ze)}{\pa{s_i}}=0$ for $i=1,\ldots,k+1,$ and let $\wh{s}(\ze)=(\wh{s}_1(\ze),\ldots,\wh{s}_{k+1}(\ze))$ be the corresponding critical point. Then
\[%begin{eqnarray}\label{equ-3-proof-theorem-tower of bubble}
\wh{s}_1 = \sqrt{\frac{(k+1)b_4}{2b_1}},\quad
\wh{s}_2=\frac{k b_4}{b_2h_1(\ze_1)},\ \dots,\ \wh{s}_{k+1}=\frac{b_4}{b_2h_1(\ze_{k})},
\]%end{eqnarray}
and it is easy to show that $\wh{s}(\ze)$ is non-degenerate. Plugging these into $\wh{\psi}(s,\ze)$ gives
\begin{equation}%\label{equ-4-proof-theorem-tower of bubble}
\begin{aligned}
\wh{\psi}(\wh{s}(\ze),\ze)
 &= \frac{(k+1)^2b_4}{2}-\sum_{i=1}^{k}b_3 h_2(\ze_i)
    -b_4(\frac{k+1}{2}\ln\frac{(k+1)b_4}{2b_1}\\
 &\hspace{1cm}
    + \sum_{i=1}^k i\ln\frac{i b_4}{b_2})+\sum_{i=1}^k b_4(k+1-i)\ln h_1(\ze_{i})\\
 &= C_1+\sum_{i=1}^k g_i(\ze_i),
\end{aligned}
\end{equation}
where
\[
C_1
 = \frac{(k+1)^2b_4}{2}
   - b_4\left(\frac{k+1}{2}\ln\frac{(k+1)b_4}{2b_1}+\sum_{i=1}^k i\ln\frac{i b_4}{b_2}\right),
\]
and
\[
g_i(\ze_i)
 = b_4(k+1-i)\ln \int_{\R^N}\frac{1}{|y+\ze_i|^{N-2}(1+|y|^2)^{\frac{N+2}{2}}}
   - b_3 \int_{\R^N}\frac{1}{|y+\ze_i|^2(1+|y|^2)^{N-2}}.
\]
A direct computation shows that $\ze_i=0$ is a critical point of $g_i(\ze_i)$ such that
\[%begin{eqnarray}\label{equ-5-proof-theorem-tower of bubble}
\frac{\pa^2 g_i(\ze_i)}{\pa \ze_i^j\pa\ze_i^l}|_{\ze_i=0} = 0 \text{ if } j\neq l;
\]
and
\[
%\label{equ-6-proof-theorem-tower of bubble}
\frac{\pa^2 g_i(\ze_i)}{\pa (\ze_i^j)^2}|_{\ze_i=0}
 = \frac{2N-8}{N}\int_{\R^N}\frac{b_3}{|y|^4(1+|y|^2)^{N-2}} > 0.
\]%end{eqnarray}
Consequently $\ze_i=0$ is a nondegenerate local minimum of $g_i$, hence $\ze=0$ is a $C^1$-stable critical point of $\wh{\psi}(\wh{s}(\ze),\ze)$. In particular, small $C^1$-perturbations of  $\wh{\psi}(\wh{s}(\ze),\ze)$ still have a critical point, close to $0$. Thus we conclude the proof.
\end{altproof}

\noindent\textbf{Acknowledgements}. The authors would like to thank Prof. Daomin Cao for many helpful discussions during the preparation of this paper. This work was carried out while Qianqiao Guo was visiting Justus-Liebig-Universit\"{a}t Gie{\ss}en, to which he would like to
express his gratitude for their warm hospitality.

%%appendix------------------------------------------------------------------------
\appendix
\section{\textbf{Appendix}\label{Appendix A}}

In this part, we give the lemmas used in Section \ref{Section 3}.

 Take $0<\eta<\min\{|\xi_i|, \text{dist}(\xi_i,\pa \Om), |\xi_{i_1}-\xi_{i_2}|, i, i_1,i_2=1,2,\dots,k\}$. Similarly to Lemma A.5 in \cite{MussoPis10JMPA}, we obtain the following.
\begin{lemma}\label{e56-e62}
For $i,l=1,2,\dots,k,$ and $j,h=0,1,\dots,N$, it holds
\begin{eqnarray}\label{equ1-Lemma A-e56-e62}
(P\ov{\Psi},P\ov{\Psi})&=&\wt{c}_0\frac{1}{\si^2}+o(\frac{1}{\si^2});\\\label{equ2-Lemma
A-e56-e62}
(P\ov{\Psi},P\Psi_i^j)&=&o(\frac{1}{\si^2})(\text{and}~o(\frac{1}{\de_i^2}));\\\label{equ3-Lemma
A-e56-e62}
(P\Psi_i^j,P\Psi_i^j)&=&\wt{c}_{i,j}\frac{1}{\de_i^2}+o(\frac{1}{\de_i^2});\\\label{equ4-Lemma
A-e56-e62}
(P\Psi_i^j,P\Psi_l^h)&=&o(\frac{1}{\de_i^2})(\text{and}~o(\frac{1}{\de_l^2}))~~~
\text{if}~i\neq l~ \text{or}~ j\neq h,
\end{eqnarray}
where $\wt{c}_0>0, \wt{c}_{i,j}>0$ are constants.
\end{lemma}
\begin{proof} We only prove (\ref{equ1-Lemma A-e56-e62}) and (\ref{equ2-Lemma A-e56-e62}) as $j=0$. (\ref{equ2-Lemma A-e56-e62}) as $j\neq0$ is similar. (\ref{equ3-Lemma A-e56-e62}) and (\ref{equ4-Lemma A-e56-e62}) are from Lemma A.5 in \cite{MussoPis10JMPA}.

To prove (\ref{equ1-Lemma A-e56-e62}), noticing that $\ov{\Psi}$ is an eigenfunction to (\ref{pro-eigenvalue}) with $\La=2^*-1$, by Proposition (\ref{proposition-projection estimate}), we have
\begin{eqnarray*}
&&(P\ov{\Psi},P\ov{\Psi})=\int_\Om|\nabla P\ov{\Psi}|^2-\mu\frac{|P\ov{\Psi}|^2}{|x|^2}\\
&=&\int_\Om\nabla \ov{\Psi}\nabla P\ov{\Psi}-\mu\frac{\ov{\Psi}P\ov{\Psi}}{|x|^2}-\mu\frac{(P\ov{\Psi}-\ov{\Psi})P\ov{\Psi}}{|x|^2}\\
&=&(2^*-1)\int_\Om V_\si^{2^*-2} \ov{\Psi}^2-(2^*-1)\int_\Om V_\si^{2^*-2} \ov{\Psi}(\ov{\Psi}-P\ov{\Psi})-\mu\frac{(P\ov{\Psi}-\ov{\Psi})P\ov{\Psi}}{|x|^2}\\
&=&(2^*-1)\int_\Om V_\si^{2^*-2} \ov{\Psi}^2+O(\si^{\frac{N-2}{2}})\\
&=&\frac{(N^2-4) C_\mu^{2^*}}{4}\int_\Om\frac{\si^2}{({\si^2|x|^{\be_1}}
+|x|^{\be_2})^2}\cdot\si^{N-4}\frac{(|x|^{\be_2}-\si^2|x|^{\be_1})^2}{(\si^2
|x|^{\be_1}+|x|^{\be_2})^N}+O(\si^{\frac{N-2}{2}})\\
&=&\frac{(N^2-4) C_\mu^{2^*}}{4}\int_{\frac{\Om}{\si^{\frac{\sqrt{\ov{\mu}}}{\sqrt{\ov{\mu}-\mu}}}}}\frac{1}{\si^2}
\frac{(|y|^{\be_2}-|y|^{\be_1})^2}{(|y|^{\be_1}+|y|^{\be_2})^{2+N}}+O(\si^{\frac{N-2}{2}})~~~(x=\si^{\frac{\sqrt{\ov{\mu}}}{\sqrt{\ov{\mu}-\mu}}}y)\\
&=&\frac{(N^2-4) C_\mu^{2^*}}{4}\int_{\R^N}\frac{1}{\si^2}\frac{(|y|^{\be_2}-|y|^{\be_1})^2}{(|y|^{\be_1}+|y|^{\be_2})^{2+N}}+o(\frac{1}{\si^2})\\
&=&\wt{c}_0\frac{1}{\si^2}+o(\frac{1}{\si^2}),
\end{eqnarray*}
for a positive constant $\wt{c}_0.$
Similarly,
\begin{eqnarray*}
&&(P\ov{\Psi},P\Psi_i^0)=\int_\Om\nabla P\ov{\Psi}\cdot \nabla P\Psi_i^0-\mu\frac{P\ov{\Psi}\cdot P\Psi_i^0}{|x|^2}\\
&=&\int_\Om\nabla \ov{\Psi}\nabla P\Psi_i^0-\mu\frac{\ov{\Psi}P\Psi_i^0}{|x|^2}-\mu\frac{(P\ov{\Psi}-\ov{\Psi})P\Psi_i^0}{|x|^2}\\
&=&(2^*-1)\int_\Om V_\si^{2^*-2} \ov{\Psi}\Psi_i^0+O(\si^{\frac{N-2}{2}})\\
&=&C_0 C_\mu^{2^*-1}(2^*-1)\int_\Om\frac{\si^2}{({\si^2|x|^{\be_1}}
+|x|^{\be_2})^2}\cdot\frac{N-2}{2}\si^{\frac{N-4}{2}}\frac{(|x|^{\be_2}-\si^2|x|^{\be_1})}{(\si^2
|x|^{\be_1}+|x|^{\be_2})^{\frac{N}{2}}}\\
&&\cdot\frac{N-2}{2}\de_i^{\frac{N-4}{2}}\frac{(|x-\xi_i|^2-\de_i^2)}{(\de_i^2+|x-\xi_i|^2)^{\frac{N}{2}}}
+O(\si^{\frac{N-2}{2}})\\
&=&o(\frac{1}{\si^2})(\text{and}~o(\frac{1}{\de_i^2})).
\end{eqnarray*}
\end{proof}

\begin{lemma}\label{e50-e51}
For $i=1,2,\dots,k,$ and $j=0,1,\dots,N$, it holds
\begin{eqnarray}\label{equ-Lemma A-e50}
\|P\Psi_i^j-\Psi_i^j\|_{2N/(N-2)}&=&\begin{cases}
O(\de_i^{\frac{N-2}{2}})\quad& \text{if} ~~j=1,2,\dots,N,\\
O(\de_i^{\frac{N-4}{2}})\quad& \text{if} ~~j=0;
\end{cases}\\\label{equ-Lemma A-e51}
\|P\ov{\Psi}-\ov{\Psi}\|_{2N/(N-2)}&=&O(\si^{\frac{N-4}{2}}).
\end{eqnarray}
\end{lemma}

\begin{proof}
(\ref{equ-Lemma A-e50}) can be proved as Lemma B.4 in \cite{MussoPis02IUMJ}. (\ref{equ-Lemma A-e51}) can be obtained similarly by using Proposition \ref{proposition-projection estimate}.
\end{proof}

\begin{lemma}\label{e83-85-Lemma A}
\begin{eqnarray}\label{equ-Lemma A-e83-e84}
&&\|(f'_0(-\sum\limits_{i=1}^{k}PU_{\de_i,\xi_i}+P V_\si)-f'_0(U_{\de_l,\xi_l}))\Psi_l^h\|_{2N/(N+2)}\\\nonumber
&\le&\begin{cases}
O(\si^{\frac{N-2}{2}})+\sum\limits_{i=1}^{k}O(\de_i^{\frac{N-2}{2}})\quad& \text{if} ~~h=1,2,\dots,N,\\
O(\si^{\frac{N-2}{2}})+\sum\limits_{i=1,i\neq l}^{k}O(\de_i^{\frac{N-2}{2}})+O(\de_l^{\frac{N-4}{2}})\quad& \text{if} ~~h=0;
\end{cases}\\\label{equ-Lemma A-e85}
&&\|(f'_0(-\sum\limits_{i=1}^{k}PU_{\de_i,\xi_i}+P V_\si)-f'_0(V_\si))\ov{\Psi}\|_{2N/(N+2)}\\\nonumber
&\le&O(\si^{\frac{N-4}{2}})+\sum\limits_{i=1}^{k}O(\de_i^{\frac{N-2}{2}}).
\end{eqnarray}
\end{lemma}
\begin{proof} We only prove (\ref{equ-Lemma A-e83-e84}).
\begin{eqnarray*}
&&\int_\Om |(f'_0(-\sum\limits_{i=1}^{k}PU_{\de_i,\xi_i}+P V_\si)-f'_0(U_{\de_l,\xi_l}))\Psi_l^h|^{2N/(N+2)}\\
&=&\int_{B(\xi_l,\frac{\eta}{2})} |(f'_0(-\sum\limits_{i=1}^{k}PU_{\de_i,\xi_i}+P V_\si)-f'_0(U_{\de_l,\xi_l}))\Psi_l^h|^{2N/(N+2)}\\
&&+\int_{B(0,\frac{\eta}{2})\bigcup\bigcup\limits_{i=1,i\neq l}^{k} B(\xi_i,\frac{\eta}{2})} |(f'_0(-\sum\limits_{i=1}^{k}PU_{\de_i,\xi_i}+P V_\si)-f'_0(U_{\de_l,\xi_l}))\Psi_l^h|^{2N/(N+2)}\\
&&+\int_{\Om\backslash (B(0,\frac{\eta}{2})\bigcup\bigcup\limits_{i=1}^{k} B(\xi_i,\frac{\eta}{2}))} |(f'_0(-\sum\limits_{i=1}^{k}PU_{\de_i,\xi_i}+P V_\si)-f'_0(U_{\de_l,\xi_l}))\Psi_l^h|^{2N/(N+2)}.
\end{eqnarray*}
First of all, by \eqref{equ-add1-projection estimate}, (\ref{equ-3-projection estimate}),
\begin{eqnarray*}
&&\int_{B(\xi_l,\frac{\eta}{2})} |(f'_0(-\sum\limits_{i=1}^{k}PU_{\de_i,\xi_i}+P V_\si)-f'_0(U_{\de_l,\xi_l}))\Psi_l^h|^{2N/(N+2)}\\
&\le&\int_{B(\xi_l,\frac{\eta}{2})} |(f'_0(PU_{\de_l,\xi_l})-f'_0(U_{\de_l,\xi_l}))\Psi_l^h|^{2N/(N+2)}+O(\si^{\frac{N(N-2)}{N+2}})+\sum\limits_{i=1,i\neq l}^{k}O(\de_i^{\frac{N(N-2)}{N+2}})\\
&\le&O(\si^{\frac{N(N-2)}{N+2}})+\sum\limits_{i=1}^{k}O(\de_i^{\frac{N(N-2)}{N+2}}).
\end{eqnarray*}
For $i\neq l$, we have
\begin{eqnarray*}
&&\int_{B(\xi_i,\frac{\eta}{2})} |(f'_0(-\sum\limits_{i=1}^{k}PU_{\de_i,\xi_i}+P V_\si)-f'_0(U_{\de_l,\xi_l}))\Psi_l^h|^{2N/(N+2)}\\
&=&\int_{B(\xi_i,\frac{\eta}{2})} |(f'_0(PU_{\de_i,\xi_i})+O(\si^{\frac{N-2}{2}})+\sum\limits_{j=1,j\neq i,j\neq l}^{k}O(\de_j^{\frac{N-2}{2}})+O(\de_l^{2}))\Psi_l^h|^{2N/(N+2)}\\
&=&\begin{cases}
O(\de_l^{\frac{N(N-2)}{N+2}})\quad& \text{if} ~~h=1,2,\dots,N,\\
O(\de_l^{\frac{N(N-4)}{N+2}})\quad& \text{if} ~~h=0.
\end{cases}
\end{eqnarray*}
At last,
\begin{eqnarray*}
&&\int_{\Om\backslash B(0,\frac{\eta}{2})\bigcup\limits_{i=1}^{k} B(\xi_i,\frac{\eta}{2})} |(f'_0(-\sum\limits_{i=1}^{k}PU_{\de_i,\xi_i}+P V_\si)-f'_0(U_{\de_l,\xi_l}))\Psi_l^h|^{2N/(N+2)}\\
&\le&\begin{cases}
O(\de_l^{\frac{N(N-2)}{N+2}})(O(\si^{\frac{4N}{N+2}})+\sum\limits_{i=1}^{k}O(\de_i^{\frac{4N}{N+2}}))\quad& \text{if} ~~h=1,2,\dots,N,\\
O(\de_l^{\frac{N(N-4)}{N+2}})(O(\si^{\frac{4N}{N+2}})+\sum\limits_{i=1}^{k}O(\de_i^{\frac{4N}{N+2}}))\quad& \text{if} ~~h=0.
\end{cases}
\end{eqnarray*}
Then (\ref{equ-Lemma A-e83-e84}) follows.
\end{proof}

\begin{lemma}\label{(M17)-Lemma A}
\begin{equation}\label{equ-Lemma A-(M17)}
\|\iota^*(-\sum\limits_{i=1}^{k}f_0(U_{\de_i,\xi_i})+f_0(V_\si))-V_{\eps,\la,\xi}\|_\mu
\le
\sum\limits_{i=1}^{k}O(\mu\de_i)+O((\mu\si^{\frac{N-2}{2}})^{\frac{1}{2}}).
\end{equation}
\end{lemma}
\begin{proof} By Definition \eqref{adjoint operator}, there holds
\begin{eqnarray}\label{equ1-Lemma A-(M17)}
&&\int_\Om \nabla \iota^*(f_0(V_\si))\nabla (\iota^*(f_0(V_\si))-P V_\si)-\mu\int_\Om\frac{\iota^*(f_0(V_\si))(\iota^*(f_0(V_\si))-P V_\si)}{|x|^2}\\\nonumber
&=&\int_\Om f_0(V_\si)(\iota^*(f_0(V_\si))-P V_\si).
\end{eqnarray}
It also has
\[
\begin{cases}
-\De P V_\si=-\De V_\si=\mu\frac{V_\si}{|x|^2}+f_0(V_\si)
\quad& \text{in} ~\Om, \\\
P V_\si=0\quad& \text{on}~
\pa \Om.
\end{cases}
\]
Then
\begin{eqnarray}\label{equ2-Lemma A-(M17)}
&&\int_\Om \nabla P V_\si\nabla (\iota^*(f_0(V_\si))-P V_\si)-\mu\int_\Om\frac{V_\si(\iota^*(f_0(V_\si))-P V_\si)}{|x|^2}\\\nonumber
&=&\int_\Om f_0(V_\si)(\iota^*(f_0(V_\si))-P V_\si).
\end{eqnarray}
Combining with (\ref{equ1-Lemma A-(M17)}) and (\ref{equ2-Lemma A-(M17)}) it holds
\begin{eqnarray}\label{equ3-Lemma A-(M17)}
\int_\Om |\nabla(\iota^*(f_0(V_\si))-P V_\si)|^2=\mu\int_\Om\frac{(\iota^*(f_0(V_\si))-V_\si)(\iota^*(f_0(V_\si))-P V_\si)}{|x|^2}.
\end{eqnarray}
By \eqref{equ-add1-projection estimate},
\begin{eqnarray}\label{equ4-Lemma A-(M17)}
\|\iota^*(f_0(V_\si))-PV_\si\|_\mu=(\mu\int_\Om\frac{|(V_\si -P V_\si)(\iota^*(f_0(V_\si))-P V_\si)|}{|x|^2})^{\frac{1}{2}}\le O((\mu\si^{\frac{N-2}{2}})^{\frac{1}{2}}).
\end{eqnarray}

Similarly to (\ref{equ1-Lemma A-(M17)}), (\ref{equ2-Lemma A-(M17)}), we also have
\begin{eqnarray}\label{equ5-Lemma A-(M17)}
&&\int_\Om \nabla \iota^*(f_0(U_{\de_i,\xi_i}))\nabla (\iota^*(f_0(U_{\de_i,\xi_i}))-P U_{\de_i,\xi_i})\\\nonumber
&&-\mu\int_\Om\frac{\iota^*(f_0(U_{\de_i,\xi_i}))(\iota^*(f_0(U_{\de_i,\xi_i}))-P U_{\de_i,\xi_i})}{|x|^2}\\\nonumber
&=&\int_\Om f_0(U_{\de_i,\xi_i})(\iota^*(f_0(U_{\de_i,\xi_i}))-P U_{\de_i,\xi_i})
\end{eqnarray}
and
\begin{eqnarray}\label{equ6-Lemma A-(M17)}
&&\int_\Om \nabla P U_{\de_i,\xi_i}\nabla (\iota^*(f_0(U_{\de_i,\xi_i}))-PU_{\de_i,\xi_i})=\int_\Om f_0(U_{\de_i,\xi_i})(\iota^*(f_0(U_{\de_i,\xi_i}))-P U_{\de_i,\xi_i}).
\end{eqnarray}
Then
\begin{eqnarray}\label{equ7-Lemma A-(M17)}
\int_\Om |\nabla(\iota^*(f_0(U_{\de_i,\xi_i}))-P U_{\de_i,\xi_i})|^2=\mu\int_\Om\frac{\iota^*(f_0(U_{\de_i,\xi_i}))(\iota^*(f_0(U_{\de_i,\xi_i}))-P U_{\de_i,\xi_i})}{|x|^2}.
\end{eqnarray}
Therefore, by H\"{o}lder's inequality and Hardy's inequality,
\begin{eqnarray*}
\|\iota^*(f_0(U_{\de_i,\xi_i}))-PU_{\de_i,\xi_i}\|_\mu
&=&(\mu\int_\Om\frac{P
U_{\de_i,\xi_i}(\iota^*(f_0(U_{\de_i,\xi_i}))-P
U_{\de_i,\xi_i})}{|x|^2})^{\frac{1}{2}}\\\nonumber
&\le&(\mu(\int_\Om\frac{(P
U_{\de_i,\xi_i})^2}{|x|^2})^{\frac{1}{2}}
(\int_\Om\frac{(\iota^*(f_0(U_{\de_i,\xi_i}))-P
U_{\de_i,\xi_i})^2}{|x|^2})^{\frac{1}{2}})^{\frac{1}{2}}\\\nonumber
&\le&C(\mu\de_i
\|\iota^*(f_0(U_{\de_i,\xi_i}))-PU_{\de_i,\xi_i}\|_\mu)^{\frac{1}{2}},
\end{eqnarray*}
which implies
\begin{eqnarray}\label{equ8-Lemma A-(M17)}
\|\iota^*(f_0(U_{\de_i,\xi_i}))-PU_{\de_i,\xi_i}\|_\mu\le
O(\mu\de_i).
\end{eqnarray}

Hence, (\ref{equ-Lemma A-(M17)}) follows from (\ref{equ4-Lemma
A-(M17)}) and (\ref{equ8-Lemma A-(M17)}).
\end{proof}
\begin{lemma}\label{e53-e55}
\begin{eqnarray}\label{equ1-Lemma A-e53-e55}
\|(f'_\eps(V_{\eps,\la,\xi})-f'_0(V_{\eps,\la,\xi}))\phi\|_{2N/(N+2)}&=&C\eps\|\phi\|_\mu;\\\label{equ2-Lemma A-e53-e55}
\|f_\eps(V_{\eps,\la,\xi})-f_0(V_{\eps,\la,\xi})\|_{2N/(N+2)}&=&C\eps;\\\label{equ3-Lemma A-e53-e55}
\|f_0(V_{\eps,\la,\xi})-(-\sum\limits_{i=1}^{k}f_0(U_{\de_i,\xi_i})+f_0(V_\si))\|_{2N/(N+2)}&=& O(\si^{\frac{N+2}{2}})+\sum\limits_{i=1}^{k}O(\de_i^{\frac{N+2}{2}}).
\end{eqnarray}
\end{lemma}
\begin{proof}
 (\ref{equ1-Lemma A-e53-e55}) and (\ref{equ2-Lemma A-e53-e55}) can be seen in \cite{BarMiPis06CVPDE}.

By (\ref{equ-1-projection estimate}), (\ref{equ-3-projection estimate}),
\begin{eqnarray*}
&&(\int_{B(0,\frac{\eta}{2})} |f_0(V_{\eps,\la,\xi})-(-\sum\limits_{i=1}^{k}f_0(U_{\de_i,\xi_i})+f_0(V_\si))|^{2N/(N+2)})^{(N+2)/2N}\\
&\le& C(\int_{B(0,\frac{\eta}{2})} |(P V_\si)^{2^*-1}-V_\si^{2^*-1}|^{2N/(N+2)})^{(N+2)/2N}+\sum\limits_{i=1}^{k}O(\de_i^{\frac{N+2}{2}})\\
&\le& C\si^{\frac{N-2}{2}}(\int_{B(0,\frac{\eta}{2})} |(V_\si)^{2^*-2}|^{2N/(N+2)})^{(N+2)/2N}+\sum\limits_{i=1}^{k}O(\de_i^{\frac{N+2}{2}})\\
&=&
O(\si^{\frac{N+2}{2}})+\sum\limits_{i=1}^{k}O(\de_i^{\frac{N+2}{2}}),
\end{eqnarray*}
\begin{eqnarray*}
&&(\int_{B(\xi_i,\frac{\eta}{2})} |f_0(V_{\eps,\la,\xi})-(-\sum\limits_{i=1}^{k}f_0(U_{\de_i,\xi_i})+f_0(V_\si))|^{2N/(N+2)})^{(N+2)/2N}\\
&\le& C(\int_{B(\xi_i,\frac{\eta}{2})} |(P U_{\de_i,\xi_i})^{2^*-1}-U_{\de_i,\xi_i}^{2^*-1}|^{2N/(N+2)})^{(N+2)/2N}+O(\si^{\frac{N+2}{2}})+\sum\limits_{j=1,j\neq i}^{k}O(\de_j^{\frac{N+2}{2}})\\
&=& O(\si^{\frac{N+2}{2}})+\sum\limits_{i=1}^{k}O(\de_i^{\frac{N+2}{2}}),
\end{eqnarray*}
\begin{eqnarray*}
&&(\int_{\Om\backslash(B(0,\frac{\eta}{2})\bigcup\limits_{i=1}^{k} B(\xi_i,\frac{\eta}{2}))} |f_0(V_{\eps,\la,\xi})-(-\sum\limits_{i=1}^{k}f_0(U_{\de_i,\xi_i})+f_0(V_\si))|^{2N/(N+2)})^{(N+2)/2N}\\
&=& O(\si^{\frac{N+2}{2}})+\sum\limits_{i=1}^{k}O(\de_i^{\frac{N+2}{2}}),
\end{eqnarray*}
then we deduce (\ref{equ3-Lemma A-e53-e55}).
\end{proof}

\begin{lemma}\label{e63-68}
For $i=1,2,\dots,k$, there hold
\begin{eqnarray}\label{equ1-Lemma A-e63-68}
\|\pa_{\la_i}P\Psi_i^j\|_\mu&=&O(\eps^{\frac{1}{N-2}}\de_i^{-2}),~j=1,2,\dots,N;\\\label{equ2-Lemma A-e63-68}
\|\pa_{(\xi_i)^j}P\Psi_i^j\|_\mu&=&O(\de_i^{-2}),~j=1,2,\dots,N;\\\label{equ3-Lemma A-e63-68}
\|\pa_{(\xi_i)^l}P\Psi_i^j\|_\mu&=&O(\de_i^{-2}),~j,l=1,2,\dots,N,j\neq l;\\\label{equ4-Lemma A-e63-68}
\|\pa_{\la_i}P\Psi_i^0\|_\mu&=&O(\eps^{\frac{1}{N-2}}\de_i^{-2});\\\label{equ5-Lemma A-e63-68}
\|\pa_{(\xi_i)^j}P\Psi_i^0\|_\mu&=&O(\de_i^{-2}),~j=1,2,\dots,N;\\\label{equ6-Lemma A-e63-68}
\|\pa_{\ov{\la}}P\ov{\Psi}\|_\mu&=&O(\eps^{\frac{1}{N-2}}\si^{-2}).
\end{eqnarray}
\end{lemma}
\begin{proof} We only prove (\ref{equ1-Lemma A-e63-68}) and (\ref{equ6-Lemma A-e63-68}) here.
\begin{eqnarray*}
\|\pa_{\la_i}P\Psi_i^j\|_\mu^2&=&\eps^{\frac{2}{N-2}}\|\pa_{\de_i}P\Psi_i^j\|_\mu^2\\
&\le& C\eps^{\frac{2}{N-2}}\int_\Om\nabla\pa_{\de_i}\Psi_i^j\nabla\pa_{\de_i}P\Psi_i^j\\
&=& C\eps^{\frac{2}{N-2}}
\int_\Om((2^*-1)(2^*-2)U_{\de_i,\xi_i}^{2^*-3}\Psi_i^j\Psi_i^0+(2^*-1)U_{\de_i,\xi_i}^{2^*-2}\pa_{\de_i}\Psi_i^j)\pa_{\de_i}P\Psi_i^j\\
&=& O(\eps^{\frac{2}{N-2}}\de_i^{-4}),
\end{eqnarray*}
then (\ref{equ1-Lemma A-e63-68}) is obtained.
\begin{eqnarray*}
\|\pa_{\ov{\la}}P\ov{\Psi}\|_\mu^2&=&\eps^{\frac{2}{N-2}}\|\pa_{\si}P\ov{\Psi}\|_\mu^2\\
&=&\eps^{\frac{2}{N-2}}\int_\Om\nabla\pa_{\si}\ov{\Psi}\nabla\pa_{\si}P\ov{\Psi}-\mu\frac{\pa_{\si}\ov{\Psi}\pa_{\si}P\ov{\Psi}}{|x|^2}
+\mu\frac{\pa_{\si}P\ov{\Psi}(\pa_{\si}\ov{\Psi}-\pa_{\si}P\ov{\Psi})}{|x|^2}\\
&=&\eps^{\frac{2}{N-2}}
(\int_\Om((2^*-1)(2^*-2)V_\si^{2^*-3}(\ov{\Psi})^2+(2^*-1)V_\si^{2^*-2}\pa_{\si}\ov{\Psi})\pa_{\si}P\ov{\Psi}+o(1))\\
&=&O(\eps^{\frac{2}{N-2}}\si^{-4}),
\end{eqnarray*}
which yields  (\ref{equ6-Lemma A-e63-68}).
\end{proof}

\begin{lemma}\label{e12,e24,e30,e34}
As $\mu\to0^+$,
\begin{eqnarray}\label{equ1-Lemma A-e12,e24,e30,e34}
&&\int_{\Om}|\nabla
P V_\si|^2-\mu\frac{|P V_\si|^2}{|x|^2}\\\nonumber
&=&
S_\mu^{\frac{N}{2}}-C_0C_\mu^{2^*-1} H(0,0)
\si^{N-2}
\int_{\R^N} \frac{1}{({|z|^{\be_1}}+|z|^{\be_2})^{\frac{N+2}{2}}}+O(\mu\si^{N-2})+O(\si^N);\\\label{equ2-Lemma A-e12,e24,e30,e34}
&&\int_{\Om}\nabla
P V_\si\nabla PU_{\de_i,\xi_i}-\mu\frac{P V_\si PU_{\de_i,\xi_i}}{|x|^2}\\\nonumber
&=&C_0C_\mu^{2^*-1} \si^{\frac{N-2}{2}}\de_i^{\frac{N-2}{2}}\int_{\R^N}\frac{G(\xi_i,0)}{({|z|^{\be_1}}
+|z|^{\be_2})^{\frac{N+2}{2}}}+O(\mu\si^{\frac{N-2}{2}}\de_i^{\frac{N-2}{2}})+o(\si^{\frac{N-2}{2}}\de_i^{\frac{N-2}{2}});\\\label{equ3-Lemma A-e12,e24,e30,e34}
&&\mu\int_{\Om}\frac{|PU_{\de_i,\xi_i}|^2}{|x|^2}=\mu\frac{C_0^2\de_i^2}{|\xi_i|^2}\int_{\R^N}\frac{1}{(1+|z|^2)^{N-2}}+O(\mu\de_i^4);\\\label{equ4-Lemma A-e12,e24,e30,e34}
&&\mu\int_{\Om}\frac{PU_{\de_i,\xi_i}PU_{\de_j,\xi_j}}{|x|^2}
=O(\mu\de_i^{\frac{N-2}{2}}\de_j^{\frac{N-2}{2}});\\\label{equ--add1-Lemma A-e12,e24,e30,e34}
&&\int_{\Om}|\nabla P U_{\de_i,\xi_i}|^2=S_0^{\frac{N}{2}}-C_0^{2^*}H(\xi_i,\xi_i)\de_i^{N-2}\int_{\R^N} \frac{1}{(1+|z|^2)^{\frac{N+2}{2}}}+o(\de_i^{N-2});\\\label{equ--add2-Lemma A-e12,e24,e30,e34}
&&\int_{\Om}\nabla P U_{\de_i,\xi_i}\nabla P U_{\de_j,\xi_j}=C_0^{2^*}G(\xi_i,\xi_j)\de_i^{\frac{N-2}{2}}\de_j^{\frac{N-2}{2}}
\int_{\R^N}\frac{1}{(1+|z|^2)^{\frac{N+2}{2}}}+o(\de_i^{\frac{N-2}{2}}\de_j^{\frac{N-2}{2}}),
\end{eqnarray}
where $i,j=1,2,\dots,k,i\neq j$.
\end{lemma}
\begin{proof} The proofs of (\ref{equ--add1-Lemma A-e12,e24,e30,e34}) and (\ref{equ--add2-Lemma A-e12,e24,e30,e34}) are from \cite{BaCor88CPAM}. We prove the remaining.

(1). Proof of (\ref{equ1-Lemma A-e12,e24,e30,e34}).

Integration by parts yields
\begin{eqnarray*}
&&\int_{\Om}|\nabla
P V_\si|^2-\mu\frac{|P V_\si|^2}{|x|^2}=\int_{\Om}(-\De V_\si) P V_\si-\mu\frac{|P V_\si|^2}{|x|^2}\\
&=&\int_{\Om}V_\si^{2^*-1} P V_\si+\mu\frac{V_\si P V_\si-|P V_\si|^2}{|x|^2}\\
&=&\int_{\Om}V_\si^{2^*}-\int_\Om V_\si^{2^*-1}\vphi_\si+\mu\int_\Om\frac{\vphi_\si(V_\si-\vphi_\si)}{|x|^2}.
\end{eqnarray*}
As (B.19) in \cite{BaCor88CPAM}, to continue, let us first show the following.
\begin{eqnarray}\label{equ5-Lemma A-e12,e24,e30,e34}
&&\int_\Om V_\si^{2^*-1}H(0,x)\\\nonumber
&=&\int_{B(0,\frac{\eta}{2})} V_\si^{2^*-1}H(0,x)+O(\si^{\frac{N+2}{2}})\\\nonumber
&=&H(0,0)\int_{B(0,\frac{\eta}{2})} V_\si^{2^*-1}+O(\si^{\frac{N+2}{2}})\\\nonumber
&&(\text{expanding}~ H(0,x)~ \text{near}~ x=0)\\\nonumber
&=&H(0,0)C_\mu^{2^*-1}\int_{B(0,\frac{\eta}{2})} \frac{\si^{\frac{N+2}{2}}}{({\si^2|x|^{\be_1}}
+|x|^{\be_2})^{\frac{N+2}{2}}}+O(\si^{\frac{N+2}{2}})\\\nonumber
&=&H(0,0)C_\mu^{2^*-1}\int_{B(0,\frac{\eta}{2}\cdot\si^{-\frac{\sqrt{\ov{\mu}}}{\sqrt{\ov{\mu}-\mu}}})} \frac{\si^{\frac{\ov{\mu}}{\sqrt{\ov{\mu}-\mu}}}}{({|z|^{\be_1}}
+|z|^{\be_2})^{\frac{N+2}{2}}}+O(\si^{\frac{N+2}{2}})~~~~(\si^{-\frac{\sqrt{\ov{\mu}}}{\sqrt{\ov{\mu}-\mu}}}x=z)\\\nonumber
&=&H(0,0)C_\mu^{2^*-1}\si^{\frac{\ov{\mu}}{\sqrt{\ov{\mu}-\mu}}}\int_{\R^N} \frac{1}{({|z|^{\be_1}}
+|z|^{\be_2})^{\frac{N+2}{2}}}+O(\si^{\frac{N+2}{2}}).
\end{eqnarray}
So, by using (\ref{equ-1-projection estimate}),
\begin{eqnarray}\label{equ6-Lemma A-e12,e24,e30,e34}
\int_\Om V_\si^{2^*-1}\vphi_\si&=&C_0C_\mu^{2^*-1}H(0,0)
\si^{\frac{\sqrt{\ov{\mu}}(\sqrt{\ov{\mu}}+\sqrt{\ov{\mu}-\mu})}{\sqrt{\ov{\mu}-\mu}}}
\int_{\R^N} \frac{1}{({|z|^{\be_1}}+|z|^{\be_2})^{\frac{N+2}{2}}}\\\nonumber
&&+O(\mu\si^{\frac{\sqrt{\ov{\mu}}(\sqrt{\ov{\mu}}+\sqrt{\ov{\mu}-\mu})}{\sqrt{\ov{\mu}-\mu}}})+O(\si^N)\\\nonumber
&=&C_0C_\mu^{2^*-1}H(0,0)
\si^{N-2}
\int_{\R^N} \frac{1}{({|z|^{\be_1}}+|z|^{\be_2})^{\frac{N+2}{2}}}+O(\mu\si^{N-2})+O(\si^N).
\end{eqnarray}
On the other hand, since
\begin{eqnarray}\label{equ7-Lemma A-e12,e24,e30,e34}
&&|\int_\Om\frac{\ov{d}^{\sqrt{\ov{\mu}}-\sqrt{\ov{\mu}-\mu}}(x)H(0,x)}{|x|^2}(\frac{1}{({\si^2|x|^{\be_1}}
+|x|^{\be_2})^{\frac{N-2}{2}}}-\frac{1}{|x|^{\sqrt{\ov{\mu}}+\sqrt{\ov{\mu}-\mu}}})|\\\nonumber
&\le&C\int_{B(0,\si^{\frac{\sqrt{\ov{\mu}}}{\sqrt{\ov{\mu}-\mu}}})}
|\frac{1}{|x|^2}(\frac{|x|^{\sqrt{\ov{\mu}}+\sqrt{\ov{\mu}-\mu}}-({\si^2|x|^{\be_1}}
+|x|^{\be_2})^{\frac{N-2}{2}}}{({\si^2|x|^{\be_1}}
+|x|^{\be_2})^{\frac{N-2}{2}}|x|^{\sqrt{\ov{\mu}}+\sqrt{\ov{\mu}-\mu}}})|\\\nonumber
&&+C\int_{\Om\backslash B(0,\si^{\frac{\sqrt{\ov{\mu}}}{\sqrt{\ov{\mu}-\mu}}})}
|\frac{1}{|x|^2}(\frac{|x|^{\sqrt{\ov{\mu}}+\sqrt{\ov{\mu}-\mu}}-({\si^2|x|^{\be_1}}
+|x|^{\be_2})^{\frac{N-2}{2}}}{({\si^2|x|^{\be_1}}
+|x|^{\be_2})^{\frac{N-2}{2}}|x|^{\sqrt{\ov{\mu}}+\sqrt{\ov{\mu}-\mu}}})|\\\nonumber
&\le&C\int_{B(0,\si^{\frac{\sqrt{\ov{\mu}}}{\sqrt{\ov{\mu}-\mu}}})}
|\frac{1}{|x|^2}\frac{\si^{N-2}|x|^{\sqrt{\ov{\mu}}-\sqrt{\ov{\mu}-\mu}}}{\si^{N-2}|x|^{\sqrt{\ov{\mu}}-\sqrt{\ov{\mu}-\mu}}
|x|^{\sqrt{\ov{\mu}}+\sqrt{\ov{\mu}-\mu}}}|\\\nonumber
&&+C\int_{\Om\backslash B(0,\si^{\frac{\sqrt{\ov{\mu}}}{\sqrt{\ov{\mu}-\mu}}})}
|\frac{1}{|x|^2}\frac{|x|^{\sqrt{\ov{\mu}}+\sqrt{\ov{\mu}-\mu}}\si^2|x|^{\frac{-2\sqrt{\ov{\mu}-\mu}}{\sqrt{\ov{\mu}}}}}
{|x|^{\sqrt{\ov{\mu}}+\sqrt{\ov{\mu}-\mu}}
|x|^{\sqrt{\ov{\mu}}+\sqrt{\ov{\mu}-\mu}}}|\\\nonumber
&=&O(\si^{\frac{\sqrt{\ov{\mu}}(\sqrt{\ov{\mu}}-\sqrt{\ov{\mu}-\mu})}{\sqrt{\ov{\mu}-\mu}}})
+O(\si^{2}),
\end{eqnarray}
then
\begin{eqnarray}\label{equ8-Lemma A-e12,e24,e30,e34}
\mu\int_\Om\frac{\vphi_\si V_\si}{|x|^2}&=&(1+o(1))\mu C_\mu^2 \si^{N-2}\int_\Om\frac{\ov{d}^{\sqrt{\ov{\mu}}-\sqrt{\ov{\mu}-\mu}}(x)H(0,x)}{|x|^2|x|^{\sqrt{\ov{\mu}}+\sqrt{\ov{\mu}-\mu}}}\\\nonumber
&=&O(\mu\si^{N-2}).
\end{eqnarray}
It is also easy to see that
\begin{equation}\label{equ9-Lemma A-e12,e24,e30,e34}
\mu\int_{\Om}\frac{\vphi_\si^2}{|x|^2}=O(\mu\si^{N-2}),
\end{equation}
\begin{equation}\label{equ10-Lemma A-e12,e24,e30,e34}
\int_{\Om}V_\si^{2^*}=S_\mu^{\frac{N}{2}}+O(\si^N).
\end{equation}

Hence (\ref{equ6-Lemma A-e12,e24,e30,e34}), (\ref{equ8-Lemma A-e12,e24,e30,e34}), (\ref{equ9-Lemma A-e12,e24,e30,e34}), (\ref{equ10-Lemma A-e12,e24,e30,e34})
yield (\ref{equ1-Lemma A-e12,e24,e30,e34}).

\noindent(2). Proof of (\ref{equ2-Lemma A-e12,e24,e30,e34}).

By using (\ref{equ-3-projection estimate}), integration by parts yields
\begin{eqnarray*}
&&\int_{\Om}\nabla
P V_\si\nabla PU_{\de_i,\xi_i}-\mu\frac{P V_\si PU_{\de_i,\xi_i}}{|x|^2}\\
&=&\int_{\Om}V_\si^{2^*-1}U_{\de_i,\xi_i}-\int_\Om V_\si^{2^*-1}\vphi_{\de_i,\xi_i}+\mu\int_\Om\frac{\vphi_\si(U_{\de_i,\xi_i}-\vphi_{\de_i,\xi_i})}{|x|^2}.
\end{eqnarray*}

%%-------------------------the followings are the proofs of e102, e103
\begin{eqnarray}\label{equ11-Lemma A-e12,e24,e30,e34}
\int_{\Om}V_\si^{2^*-1}U_{\de_i,\xi_i}&=&(\int_{B(0,\frac{\eta}{4})}+\int_{B(\xi_i,\frac{\eta}{4})}+\int_{\Om\backslash B(0,\frac{\eta}{4})\cup B(\xi_i,\frac{\eta}{4})})V_\si^{2^*-1}U_{\de_i,\xi_i}\\\nonumber
&=&(\int_{B(0,\frac{\eta}{4})}+\int_{B(\xi_i,\frac{\eta}{4})})V_\si^{2^*-1}U_{\de_i,\xi_i}+O(\si^{\frac{N+2}{2}}\de_i^{\frac{N-2}{2}}).
\end{eqnarray}
As $\mu\to0^+$,
\begin{eqnarray*}
&&\int_{B(0,\frac{\eta}{4})}V_\si^{2^*-1}U_{\de_i,\xi_i}=C_0C_\mu^{2^*-1} \si^{\frac{N+2}{2}}\de_i^{\frac{N-2}{2}}\int_{B(0,\frac{\eta}{4})}\frac{1}{({\si^2|x|^{\be_1}}
+|x|^{\be_2})^{\frac{N+2}{2}}}\cdot\frac{1}{(\de_i^2+|x-\xi_i|^2)^{\frac{N-2}{2}}}\\
&=&C_0C_\mu^{2^*-1} \si^{\frac{N+2}{2}}\de_i^{\frac{N-2}{2}}\int_{B(0,\frac{\eta}{4})}\frac{1}{({\si^2|x|^{\be_1}}
+|x|^{\be_2})^{\frac{N+2}{2}}}(\frac{1}{(\de_i^2+|\xi_i|^2)^{\frac{N-2}{2}}}+O(|x|^2))\\
&=&C_0C_\mu^{2^*-1} \si^{\frac{\ov{\mu}}{\sqrt{\ov{\mu}-\mu}}}\de_i^{\frac{N-2}{2}}\int_{\R^N}\frac{1}{({|z|^{\be_1}}
+|z|^{\be_2})^{\frac{N+2}{2}}}\frac{1}{|\xi_i|^{N-2}}+o(\si^{\frac{\ov{\mu}}{\sqrt{\ov{\mu}-\mu}}}\de_i^{\frac{N-2}{2}}),
\end{eqnarray*}
\begin{eqnarray*}
&&\int_{B(\xi_i,\frac{\eta}{4})}V_\si^{2^*-1}U_{\de_i,\xi_i}=C_0C_\mu^{2^*-1} \si^{\frac{N+2}{2}}\de_i^{\frac{N-2}{2}}\int_{B(0,\frac{\eta}{4})}\frac{1}{({\si^2|x+\xi_i|^{\be_1}}
+|x+\xi_i|^{\be_2})^{\frac{N+2}{2}}}\cdot\frac{1}{(\de_i^2+|x|^2)^{\frac{N-2}{2}}}\\
&\le&O(\si^{\frac{N+2}{2}}\de_i^{\frac{N-2}{2}}),
\end{eqnarray*}
then
\begin{eqnarray*}
(\ref{equ11-Lemma A-e12,e24,e30,e34})&=&C_0C_\mu^{2^*-1} \si^{\frac{\ov{\mu}}{\sqrt{\ov{\mu}-\mu}}}\de_i^{\frac{N-2}{2}}\int_{\R^N}\frac{1}{({|z|^{\be_1}}
+|z|^{\be_2})^{\frac{N+2}{2}}}\frac{1}{|\xi_i|^{N-2}}+o(\si^{\frac{\ov{\mu}}{\sqrt{\ov{\mu}-\mu}}}\de_i^{\frac{N-2}{2}}).
\end{eqnarray*}
On the other hand,
\begin{eqnarray}\label{equ12-Lemma A-e12,e24,e30,e34}
&&\int_\Om V_\si^{2^*-1}\vphi_{\de_i,\xi_i}\\\nonumber
&=&\int_{B(0,\frac{\eta}{2})}V_\si^{2^*-1}\vphi_{\de_i,\xi_i}+O(\si^{\frac{N+2}{2}}\de_i^{\frac{N-2}{2}})\\\nonumber
&=&C_0C_\mu^{2^*-1} \si^{\frac{N+2}{2}}\de_i^{\frac{N-2}{2}}\int_{B(0,\frac{\eta}{2})}\frac{H(\xi_i,x)}{({\si^2|x|^{\be_1}}
+|x|^{\be_2})^{\frac{N+2}{2}}}+O(\si^{\frac{N+2}{2}}\de_i^{\frac{N-2}{2}})\\\nonumber
&=&C_0C_\mu^{2^*-1} \si^{\frac{N+2}{2}}\de_i^{\frac{N-2}{2}}\int_{B(0,\frac{\eta}{2})}\frac{H(\xi_i,0)}{({\si^2|x|^{\be_1}}
+|x|^{\be_2})^{\frac{N+2}{2}}}+O(\si^{\frac{N+2}{2}}\de_i^{\frac{N-2}{2}})\\\nonumber
&=&C_0C_\mu^{2^*-1} \si^{\frac{\ov{\mu}}{\sqrt{\ov{\mu}-\mu}}}\de_i^{\frac{N-2}{2}}\int_{\R^N}\frac{H(\xi_i,0)}{({|z|^{\be_1}}
+|z|^{\be_2})^{\frac{N+2}{2}}}+o(\si^{\frac{\ov{\mu}}{\sqrt{\ov{\mu}-\mu}}}\de_i^{\frac{N-2}{2}}).
\end{eqnarray}
%%-------------------------the above are the proofs of e102, e103
Similarly to (\ref{equ7-Lemma A-e12,e24,e30,e34}), we have
\begin{eqnarray}\label{equ13-Lemma A-e12,e24,e30,e34}
|\int_\Om\frac{ \ov{d}^{\sqrt{\ov{\mu}}-\sqrt{\ov{\mu}-\mu}}(x)H(0,x)}{|x|^2}(\frac{1}{(\de_i^2+|x-\xi_i|^2)^{\frac{N-2}{2}}}-\frac{1}{(|x-\xi_i|^2)^{\frac{N-2}{2}}})\le O(\de_i^2),
\end{eqnarray}
which yields
\begin{eqnarray}\label{equ14-Lemma A-e12,e24,e30,e34}
&&\mu\int_\Om\frac{\vphi_\si U_{\de_i,\xi_i}}{|x|^2}\\\nonumber
&=&\mu C_\mu C_0\si^{\frac{N-2}{2}}\de_i^{\frac{N-2}{2}}
\int_\Om\frac{\ov{d}^{\sqrt{\ov{\mu}}-\sqrt{\ov{\mu}-\mu}}(x)H(0,x)}{|x|^2|x-\xi_i|^{N-2}}+o(\mu\si^{\frac{N-2}{2}}\de_i^{\frac{N-2}{2}})\\\nonumber
&=&O(\mu\si^{\frac{N-2}{2}}\de_i^{\frac{N-2}{2}}).
\end{eqnarray}
It is also easy to see that
\begin{equation}\label{equ15-Lemma A-e12,e24,e30,e34}
\mu\int_{\Om}\frac{\vphi_\si \vphi_{\de_i,\xi_i}}{|x|^2} =O(\mu\si^{\frac{N-2}{2}}\de_i^{\frac{N-2}{2}}).
\end{equation}

By (\ref{equ11-Lemma A-e12,e24,e30,e34}), (\ref{equ12-Lemma A-e12,e24,e30,e34}), (\ref{equ14-Lemma A-e12,e24,e30,e34}), (\ref{equ15-Lemma A-e12,e24,e30,e34}), we conclude (\ref{equ2-Lemma A-e12,e24,e30,e34}).

\noindent(3). Proof of (\ref{equ3-Lemma A-e12,e24,e30,e34}).
\begin{eqnarray}\label{equ16-Lemma A-e12,e24,e30,e34}
&&\int_{\Om}\frac{U_{\de_i,\xi_i}^2}{|x|^2}\\\nonumber
&=&C_0^2\de_i^{N-2}\int_{\Om}\frac{1}{|x|^2(\de_i^2+|x-\xi_i|^2)^{N-2}}=C_0^2\de_i^{N-2}(C+\int_{B(\xi_i,\frac{\eta}{4})}\frac{1}{|x|^2(\de_i^2+|x-\xi_i|^2)^{N-2}})\\\nonumber
&=&C_0^2\de_i^{N-2}(C+\int_{B(0,\frac{\eta}{4})}\frac{1}{|x+\xi_i|^2(\de_i^2+|x|^2)^{N-2}})
=C_0^2\de_i^{N-2}(C+\int_{B(0,\frac{\eta}{4})}\frac{1+O(|x|^2)}{|\xi_i|^2(\de_i^2+|x|^2)^{N-2}})\\\nonumber
&=&C_0^2\de_i^{N-2}(C+\frac{1}{|\xi_i|^2}(-\int_{\R^N\backslash B(0,\frac{\eta}{4\de_i})}+\int_{\R^N})\frac{\de_i^{4-N}}{(1+|z|^2)^{N-2}}
+\frac{1}{|\xi_i|^2}\int_{B(0,\frac{\eta}{4})}\frac{O(|x|^2)}{(\de_i^2+|x|^2)^{N-2}})\\\nonumber
&=&\frac{C_0^2}{|\xi_i|^2}\de_i^2\int_{\R^N}\frac{1}{(1+|z|^2)^{N-2}}+O(\de_i^4).
\end{eqnarray}
On the other hand,
\begin{eqnarray}\label{equ17-Lemma A-e12,e24,e30,e34}
\int_{\Om}\frac{\vphi_{\de_i,\xi_i}^2}{|x|^2}&=&O(\de_i^{N-2});\\\label{equ18-Lemma A-e12,e24,e30,e34}
\int_{\Om}\frac{\vphi_{\de_i,\xi_i}U_{\de_i,\xi_i}}{|x|^2}&=&O(\de_i^{N-2}).
\end{eqnarray}

Then (\ref{equ16-Lemma A-e12,e24,e30,e34}), (\ref{equ17-Lemma A-e12,e24,e30,e34}), (\ref{equ18-Lemma A-e12,e24,e30,e34}) yield (\ref{equ3-Lemma A-e12,e24,e30,e34}).

\noindent(4). Proof of (\ref{equ4-Lemma A-e12,e24,e30,e34}).

We omit it here since it is similarly to (\ref{equ3-Lemma A-e12,e24,e30,e34}).
\end{proof}

\begin{lemma}\label{e4,e41}
As $\mu\to0^+$,
\begin{eqnarray}\label{equ1-Lemma A-e4,e41}
\int_{\Om}|P V_\si|^{2^*}
&=&S_\mu^{\frac{N}{2}}
-2^*C_0C_\mu^{2^*-1}H(0,0)
\si^{N-2}
\int_{\R^N} \frac{1}{({|z|^{\be_1}}+|z|^{\be_2})^{\frac{N+2}{2}}}\\\nonumber
&&+O(\mu\si^{N-2})+O(\si^N);\\\label{equ1-add-Lemma A-e4,e41}
\int_{\Om}|P U_{\de_i,\xi_i}|^{2^*}
&=&S_0^{\frac{N}{2}}-2^*C_0^{2^*} H(\xi_i,\xi_i)\de_i^{\frac{N-2}{2}}
\int_{\R^N} \frac{1}{(1+|z|^2)^{\frac{N+2}{2}}}+O(\de_i^N);\\\label{equ2-Lemma A-e4,e41}
&&\int_{\Om}|-\sum\limits_{i=1}^{k}PU_{\de_i,\xi_i}+P V_\si|^{2^*}\\\nonumber
&=&S_\mu^{\frac{N}{2}}
-2^*C_0C_\mu^{2^*-1}H(0,0)
\si^{N-2}
\int_{\R^N} \frac{1}{({|z|^{\be_1}}+|z|^{\be_2})^{\frac{N+2}{2}}}\\\nonumber
&&-2^*\sum\limits_{i=1}^{k}C_0C_\mu^{2^*-1} \si^{\frac{N-2}{2}}\de_i^{\frac{N-2}{2}}\int_{\R^N}\frac{G(\xi_i,0)}{({|z|^{\be_1}}
+|z|^{\be_2})^{\frac{N+2}{2}}}\\\nonumber
&&+\sum\limits_{i=1}^{k}[S_0^{\frac{N}{2}}
-2^*C_0^{2^*} H(\xi_i,\xi_i)
\de_i^{N-2}
\int_{\R^N} \frac{1}{(1+|z|^2)^{\frac{N+2}{2}}}\\\nonumber
&&+2^*\sum\limits_{j=1,j\neq i}^{k}C_0^{2^*} \de_i^{\frac{N-2}{2}}\de_j^{\frac{N-2}{2}}\int_{\R^N}\frac{G(\xi_i,\xi_j)}{(1+|z|^2)^{\frac{N+2}{2}}}
\\\nonumber
&&-2^*C_\mu C_0^{2^*-1} \si^{\frac{N-2}{2}}\de_i^{\frac{N-2}{2}}\int_{\R^N}\frac{G(\xi_i,0)}{(1+|z|^2)^{\frac{N+2}{2}}}
]\\\nonumber
&&+\sum\limits_{i=1}^{k}(o(\si^{\frac{N-2}{2}}\de_i^{\frac{N-2}{2}})+\sum\limits_{j=1,j\neq i}^{k}o(\de_j^{\frac{N-2}{2}}\de_i^{\frac{N-2}{2}})+O(\de_i^N))
+O(\mu\si^{N-2})
+O(\si^N).
\end{eqnarray}
\end{lemma}
\begin{proof}
(\ref{equ1-add-Lemma A-e4,e41}) is from \cite{BaCor88CPAM}.

By (\ref{equ6-Lemma A-e12,e24,e30,e34}),
\begin{eqnarray*}
&&\int_{\Om}|P V_\si|^{2^*}=\int_{\Om}V_\si^{2^*}-2^*\int_\Om V_\si^{2^*-1}\vphi_\si+O(\si^N)\\
&=&S_\mu^{\frac{N}{2}}
-2^*C_0C_\mu^{2^*-1}H(0,0)
\si^{\frac{\sqrt{\ov{\mu}}(\sqrt{\ov{\mu}}+\sqrt{\ov{\mu}-\mu})}{\sqrt{\ov{\mu}-\mu}}}
\int_{\R^N} \frac{1}{({|z|^{\be_1}}+|z|^{\be_2})^{\frac{N+2}{2}}}\\\nonumber
&&+O(\mu\si^{\frac{\sqrt{\ov{\mu}}(\sqrt{\ov{\mu}}+\sqrt{\ov{\mu}-\mu})}{\sqrt{\ov{\mu}-\mu}}})+O(\si^N).
\end{eqnarray*}

Now we turn to (\ref{equ2-Lemma A-e4,e41}). By  (\ref{equ-1-projection estimate}), (\ref{equ-3-projection estimate}), (\ref{equ11-Lemma A-e12,e24,e30,e34}), (\ref{equ12-Lemma A-e12,e24,e30,e34}),
\begin{eqnarray}\label{equ3-Lemma A-e4,e41}
&&\int_{B(0,\frac{\eta}{2})}(P V_\si)^{2^*-1}PU_{\de_i,\xi_i}=\int_{B(0,\frac{\eta}{2})}(V_\si^{2^*-1}+O(V_\si^{2^*-2}\vphi_\si))(U_{\de_i,\xi_i}-\vphi_{\de_i,\xi_i})\\\nonumber
&=&\int_{B(0,\frac{\eta}{2})}V_\si^{2^*-1}(U_{\de_i,\xi_i}-\vphi_{\de_i,\xi_i})+O(\si^{\frac{\sqrt{\ov{\mu}}(N-2+\sqrt{\ov{\mu}-\mu})}{\sqrt{\ov{\mu}-\mu}}}\de_i^{\frac{N-2}{2}})\\\nonumber
&=&C_0C_\mu^{2^*-1} \si^{\frac{N-2}{2}}\de_i^{\frac{N-2}{2}}\int_{\R^N}\frac{G(\xi_i,0)}{({|z|^{\be_1}}
+|z|^{\be_2})^{\frac{N+2}{2}}}+o(\si^{\frac{N-2}{2}}\de_i^{\frac{N-2}{2}}).
\end{eqnarray}
Then by (\ref{equ1-Lemma A-e4,e41}),
\begin{eqnarray}\label{equ4-Lemma A-e4,e41}
&&\int_{B(0,\frac{\eta}{2})}|-\sum\limits_{i=1}^{k}PU_{\de_i,\xi_i}+P V_\si|^{2^*}\\\nonumber
&=&\int_{B(0,\frac{\eta}{2})}(P V_\si)^{2^*}-2^*\sum\limits_{i=1}^{k}\int_{B(0,\frac{\eta}{2})}(P V_\si)^{2^*-1}PU_{\de_i,\xi_i}+\sum\limits_{i=1}^{k}O(\int_{B(0,\frac{\eta}{2})}(P V_\si)^{2^*-2}(PU_{\de_i,\xi_i})^2)\\\nonumber
&=&S_\mu^{\frac{N}{2}}
-2^*C_0C_\mu^{2^*-1}H(0,0)
\si^{N-2}\int_{\R^N} \frac{1}{({|z|^{\be_1}}+|z|^{\be_2})^{\frac{N+2}{2}}}\\\nonumber
&&-2^*\sum\limits_{i=1}^{k}C_0C_\mu^{2^*-1} \si^{\frac{N-2}{2}}\de_i^{\frac{N-2}{2}}\int_{\R^N}\frac{G(\xi_i,0)}{({|z|^{\be_1}}
+|z|^{\be_2})^{\frac{N+2}{2}}}\\\nonumber
&&+\sum\limits_{i=1}^{k}o(\si^{\frac{N-2}{2}}\de_i^{\frac{N-2}{2}})
+O(\mu\si^{N-2})+O(\si^N).
\end{eqnarray}
We also have
\begin{eqnarray}\label{equ5-Lemma A-e4,e41}
&&\int_{B(\xi_i,\frac{\eta}{2})}|-\sum\limits_{i=1}^{k}PU_{\de_i,\xi_i}+P
V_\si|^{2^*}\\\nonumber
&=&\int_{B(\xi_i,\frac{\eta}{2})}|PU_{\de_i,\xi_i}+\sum\limits_{j=1,j\neq
i}^{k}PU_{\de_j,\xi_j}-PV_\si|^{2^*}\\\nonumber
&=&\int_{B(\xi_i,\frac{\eta}{2})}(PU_{\de_i,\xi_i})^{2^*}+2^*\sum\limits_{j=1,j\neq
i}^{k}(PU_{\de_i,\xi_i})^{2^*-1}PU_{\de_j,\xi_j}-2^*(PU_{\de_i,\xi_i})^{2^*-1}P
V_\si\\\nonumber
&&+O(\int_{B(\xi_i,\frac{\eta}{2})}(PU_{\de_i,\xi_i})^{2^*-2}(-\sum\limits_{j=1,j\neq
i}^{k}PU_{\de_j,\xi_j}+P V_\si)^2)\\\nonumber
&=&S_0^{\frac{N}{2}} -2^*C_0^{2^*} H(\xi_i,\xi_i) \de_i^{N-2}
\int_{\R^N}
\frac{1}{(1+|z|^2)^{\frac{N+2}{2}}}+O(\de_i^N)\\\nonumber
&&+2^*\sum\limits_{j=1,j\neq i}^{k}C_0^{2^*}
\de_i^{\frac{N-2}{2}}\de_j^{\frac{N-2}{2}}\int_{\R^N}\frac{G(\xi_i,\xi_j)}{(1+|z|^2)^{\frac{N+2}{2}}}
+\sum\limits_{j=1,j\neq
i}^{k}o(\de_j^{\frac{N-2}{2}}\de_i^{\frac{N-2}{2}})\\\nonumber
&&-2^*C_\mu C_0^{2^*-1}
\si^{\frac{N-2}{2}}\de_i^{\frac{N-2}{2}}\int_{\R^N}\frac{G(\xi_i,0)}{(1+|z|^2)^{\frac{N+2}{2}}}
+o(\si^{\frac{N-2}{2}}\de_i^{\frac{N-2}{2}}),
\end{eqnarray}
where the last equality was obtained by the results in \cite{BaCor88CPAM} and
\begin{eqnarray}\label{equ6-Lemma A-e4,e41}
&&\int_{B(\xi_i,\frac{\eta}{2})}(PU_{\de_i,\xi_i})^{2^*-1}P V_\si\\\nonumber
&=&\int_{B(\xi_i,\frac{\eta}{2})}U_{\de_i,\xi_i}^{2^*-1}(V_\si-\vphi_\si)+o(\de_i^{\frac{N-2}{2}}\si^{\frac{N-2}{2}})\\\nonumber
&=&C_\mu C_0^{2^*-1} \si^{\frac{N-2}{2}}\de_i^{\frac{N-2}{2}}\int_{\R^N}\frac{1}{(1+|z|^2)^{\frac{N+2}{2}}}
(\frac{1}{|\xi_i|^{\sqrt{\ov{\mu}}+\sqrt{\ov{\mu}-\mu}}}-\ov{d}^{\sqrt{\ov{\mu}}-\sqrt{\ov{\mu}-\mu}}(x)H(0,\xi_i))
+o(\si^{\frac{N-2}{2}}\de_i^{\frac{N-2}{2}})\\\nonumber
&=&C_\mu C_0^{2^*-1} \si^{\frac{N-2}{2}}\de_i^{\frac{N-2}{2}}\int_{\R^N}\frac{G(\xi_i,0)}{(1+|z|^2)^{\frac{N+2}{2}}}
+o(\si^{\frac{N-2}{2}}\de_i^{\frac{N-2}{2}}).
\end{eqnarray}
At last,
\begin{eqnarray}\label{equ7-Lemma A-e4,e41}
\int_{\Om\backslash B(0,\frac{\eta}{2})\bigcup\bigcup\limits_{i=1}^{k} B(\xi_i,\frac{\eta}{2})}|-\sum\limits_{i=1}^{k}PU_{\de_i,\xi_i}+P V_\si|^{2^*}\le\sum\limits_{i=1}^{k}O(\de_i^N)+O(\si^N).
\end{eqnarray}

Then (\ref{equ4-Lemma A-e4,e41}),(\ref{equ5-Lemma A-e4,e41}),(\ref{equ7-Lemma A-e4,e41}) yield (\ref{equ2-Lemma A-e4,e41}).
\end{proof}

\begin{lemma}\label{e42}
\begin{eqnarray}\label{equ1-Lemma A-e42}
&&\int_{\Om}|-\sum\limits_{i=1}^{k}PU_{\de_i,\xi_i}+P
V_\si|^{2^*}\ln|-\sum\limits_{i=1}^{k}PU_{\de_i,\xi_i}+P
V_\si|\\\nonumber
&=&-\frac{N-2}{2}\ln\si\cdot\int_{\R^N}
V_1^{2^*}-\frac{N-2}{2}\ln(\de_1
\de_2\dots\de_{k})\cdot\int_{\R^N}
U_{1,0}^{2^*}\\\nonumber &&+\int_{\R^N} V_1^{2^*}\ln
V_1+k\int_{\R^N} U_{1,0}^{2^*}\ln U_{1,0}+o(1).
\end{eqnarray}
\end{lemma}
\begin{proof}
Similarly to \cite{PiFelMu03CV},
\begin{eqnarray}\label{equ2-Lemma A-e42}
&&\int_{B(0,\frac{\eta}{2})}|-\sum\limits_{i=1}^{k}PU_{\de_i,\xi_i}+P
V_\si|^{2^*}\ln|-\sum\limits_{i=1}^{k}PU_{\de_i,\xi_i}+P
V_\si|\\\nonumber
&=&-\frac{\ov{\mu}}{\sqrt{\ov{\mu}-\mu}}\ln\si\cdot\int_{\R^N}
V_1^{2^*}+\int_{\R^N} V_1^{2^*}\ln V_1
+o(1)\\\nonumber
 &=&-\frac{N-2}{2}\ln\si\cdot\int_{\R^N}
V_1^{2^*}+\int_{\R^N} V_1^{2^*}\ln V_1
+o(1),\\\label{equ3-Lemma A-e42}
&&\int_{B(\xi_i,\frac{\eta}{2})}|-\sum\limits_{i=1}^{k}PU_{\de_i,\xi_i}+P
V_\si|^{2^*}\ln|-\sum\limits_{i=1}^{k}PU_{\de_i,\xi_i}+P
V_\si|\\\nonumber &=&-\frac{N-2}{2}\ln\de_i\cdot\int_{\R^N}
U_{1,0}^{2^*}+\int_{\R^N} U_{1,0}^{2^*}\ln
U_{1,0}+o(1),\\\label{equ4-Lemma A-e42} &&\int_{\Om\backslash
B(0,\frac{\eta}{2})\bigcup\limits_{i=1}^{k}
B(\xi_i,\frac{\eta}{2})}|-\sum\limits_{i=1}^{k}PU_{\de_i,\xi_i}+P
V_\si|^{2^*}\ln|-\sum\limits_{i=1}^{k}PU_{\de_i,\xi_i}+P
V_\si|\\\nonumber
&=&o(1),
\end{eqnarray}
then we conclude (\ref{equ1-Lemma A-e42}).

\end{proof}

\begin{lemma}\label{e44-e48}
As $\mu\to 0^+$,
\begin{eqnarray}\label{equ3-Lemma A-e44-e48}
\int_{\R^N}V_{1}^{2^*-1}&=&\int_{\R^N}U_{1,0}^{2^*-1}+o(1);\\\label{equ4-Lemma
A-e44-e48}
\int_{\R^N}V_{1}^{2^*}&=&\int_{\R^N}U_{1,0}^{2^*}+o(1);\\\label{equ5-Lemma
A-e44-e48} \int_{\R^N} V_{1}^{2^*}\ln
V_{1}&=&\int_{\R^N} U_{1,0}^{2^*}\ln
U_{1,0}+o(1);\\\label{equ1-Lemma A-e44-e48}
C_\mu&=&C_0-\frac{C_0}{N-2}\mu+O(\mu^2);\\\label{equ2-Lemma
A-e44-e48} S_\mu&=&S_0-\ov{S}\mu+O(\mu^2),
\end{eqnarray}
for some positive constant $\ov{S}$ independent of $\mu$.
\end{lemma}
\begin{proof}
The equalities can be obtained by direct computations.
\end{proof}
%%---------------------------------------------------------------------------------
\section{\textbf{Appendix}\label{Appendix B}}
The lemmas used in Section \ref{Section 4} are listed below.

As Lemma \ref{e56-e62}, we have the following.
\begin{lemma}\label{e56-e62-Lemma B}
For $i,l=1,2,\dots,k,$ and $j,h=0,1,\dots,N$, it holds
\begin{eqnarray}\label{equ1-Lemma B-e56-e62}
(P\ov{\Psi},P\ov{\Psi})&=&\wt{c}_0\frac{1}{\si^2}+o(\frac{1}{\si^2});\\\label{equ2-Lemma
B-e56-e62}
(P\ov{\Psi},P\Psi_i^j)&=&o(\frac{1}{\si^2})(\text{and}~o(\frac{1}{\de_i^2}));\\\label{equ3-Lemma
B-e56-e62}
(P\Psi_i^j,P\Psi_i^j)&=&\wt{c}_{i,j}\frac{1}{\de_i^2}+o(\frac{1}{\de_i^2});\\\label{equ4-Lemma
B-e56-e62}
(P\Psi_i^j,P\Psi_l^h)&=&o(\frac{1}{\de_i^2})(\text{and}~o(\frac{1}{\de_l^2}))~~~
\text{if}~i\neq l~ \text{or}~ j\neq h,
\end{eqnarray}
where $\wt{c}_0>0, \wt{c}_{i,j}>0$ are constants.
\end{lemma}

\begin{lemma}\label{e83-85-Lemma B}
\begin{eqnarray}\label{equ-Lemma B-e83-e84}
\|(f'_0(\sum\limits_{i=1}^{k}(-1)^{i-1}PU_{\de_i,\xi_i}+(-1)^k PV_\si)-f'_0(U_{\de_l,\xi_l}))\Psi_l^h\|_{2N/(N+2)}&=&o(\frac{1}{\de_l^{\frac{2N}{N+2}}}),\\\label{equ-Lemma B-e85}
\|(f'_0(\sum\limits_{i=1}^{k}(-1)^{i-1}PU_{\de_i,\xi_i}+(-1)^k PV_\si)-f'_0(V_\si))\ov{\Psi}\|_{2N/(N+2)}&=&o(\frac{1}{\si^{\frac{2N}{N+2}}}).
\end{eqnarray}
\end{lemma}
\begin{proof} We only prove (\ref{equ-Lemma B-e83-e84}) for $h\neq0$.
\begin{eqnarray}\label{equ1-e83-85-Lemma B}
&&\int_\Om |(f'_0(\sum\limits_{i=1}^{k}(-1)^{i-1}PU_{\de_i,\xi_i}+(-1)^k PV_\si)-f'_0(U_{\de_l,\xi_l}))\Psi_l^h|^{2N/(N+2)}\\\nonumber
&=&\bigcup\limits_{i=1}^{k+1}\int_{A_i} |(f'_0(\sum\limits_{i=1}^{k}(-1)^{i-1}PU_{\de_i,\xi_i}+(-1)^k PV_\si)-f'_0(U_{\de_l,\xi_l}))\Psi_l^h|^{2N/(N+2)}\\\nonumber
&&+\int_{\Om\backslash B(0,\rho)} |(f'_0(\sum\limits_{i=1}^{k}(-1)^{i-1}PU_{\de_i,\xi_i}+(-1)^k PV_\si)-f'_0(U_{\de_l,\xi_l}))\Psi_l^h|^{2N/(N+2)}.
\end{eqnarray}
As Lemma A.3 in \cite{MussoPis10JMPA}, by (\ref{equ-1-projection estimate}), (\ref{equ-3-projection estimate}),
\begin{eqnarray}\label{equ2-e83-85-Lemma B}
&&\int_{A_l} |(f'_0(\sum\limits_{i=1}^{k}(-1)^{i-1}PU_{\de_i,\xi_i}+(-1)^k PV_\si)-f'_0(U_{\de_l,\xi_l}))\Psi_l^h|^{2N/(N+2)}\\\nonumber
&\le&C\int_{A_l} |U_{\de_l,\xi_l}^{2^*-3}\vphi_{\de_l,\xi_l}\Psi_l^h|^{2N/(N+2)}+C\sum\limits_{i\neq l}\int_{A_l} |U_{\de_l,\xi_l}^{2^*-3}U_{\de_i,\xi_i}\Psi_l^h|^{2N/(N+2)}\\\nonumber
&&+C\int_{A_l} |U_{\de_l,\xi_l}^{2^*-3}V_\si\Psi_l^h|^{2N/(N+2)}\\\nonumber
&\le&o(\frac{1}{\de_l^{\frac{2N}{N+2}}}),
\end{eqnarray}
since
\begin{eqnarray*}
&&\int_{A_l} |U_{\de_l,\xi_l}^{2^*-3}\vphi_{\de_l,\xi_l}\Psi_l^h|^{2N/(N+2)}\\
&\le&C\int_{A_l} |\frac{\de_l^{\frac{N+2}{2}}(x^h-\xi_l^h)}{(\de_l^2+|x-\xi_l|^2)^3}|^{2N/(N+2)}=O(\de_l^{\frac{2N(N-3)}{N+2}}),
\end{eqnarray*}
for $i\neq l$,
\begin{eqnarray*}
&&\int_{A_l} |U_{\de_l,\xi_l}^{2^*-3}U_{\de_i,\xi_i}\Psi_l^h|^{2N/(N+2)}\\
&=&C\int_{A_l} |\frac{\de_l^2(x^h-\xi_l^h)}{(\de_l^2+|x-\xi_l|^2)^3}\frac{\de_i^{\frac{N-2}{2}}}{(\de_i^2+|x-\xi_i|^2)^{\frac{N-2}{2}}}|^{2N/(N+2)}\\
&\le&C(\int_{A_l} |\frac{\de_l^2(x^h-\xi_l^h)}{(\de_l^2+|x-\xi_l|^2)^3}|^{\frac{N}{2}})^{\frac{4}{N+2}}
(\int_{A_l}|\frac{\de_i^{\frac{N-2}{2}}}{(\de_i^2+|x-\xi_i|^2)^{\frac{N-2}{2}}}|^{2N/(N-2)})^{\frac{N-2}{N+2}}\\
&=&o(\frac{1}{\de_l^{\frac{2N}{N+2}}}),
\end{eqnarray*}
and similarly,
\begin{eqnarray*}
\int_{A_l} |U_{\de_l,\xi_l}^{2^*-3}V_\si\Psi_l^h|^{2N/(N+2)}=o(\frac{1}{\de_l^{\frac{2N}{N+2}}}).
\end{eqnarray*}
The same arguments as (\ref{equ2-e83-85-Lemma B}) give that, for $i\neq l$,
\begin{eqnarray}\label{equ3-e83-85-Lemma B}
\int_{A_i} |(f'_0(\sum\limits_{i=1}^{k}(-1)^{i-1}PU_{\de_i,\xi_i}+(-1)^k PV_\si)-f'_0(U_{\de_l,\xi_l}))\Psi_l^h|^{2N/(N+2)}
=o(\frac{1}{\de_l^{\frac{2N}{N+2}}}).
\end{eqnarray}
At last,
\begin{eqnarray}\label{equ4-e83-85-Lemma B}
&&\int_{\Om\backslash B(0,\rho)} |(f'_0(\sum\limits_{i=1}^{k}(-1)^{i-1}PU_{\de_i,\xi_i}+(-1)^k PV_\si)-f'_0(U_{\de_l,\xi_l}))\Psi_l^h|^{2N/(N+2)}\\\nonumber
&=&\begin{cases}
O(\de_l^{\frac{N(N-2)}{N+2}})(O(\si^{\frac{4N}{N+2}})+\sum\limits_{i=1}^{k}O(\de_i^{\frac{4N}{N+2}}))\quad& \text{if} ~~h=1,2,\dots,N,\\\nonumber
O(\de_l^{\frac{N(N-4)}{N+2}})(O(\si^{\frac{4N}{N+2}})+\sum\limits_{i=1}^{k}O(\de_i^{\frac{4N}{N+2}}))\quad& \text{if} ~~h=0.
\end{cases}
\end{eqnarray}
Then (\ref{equ-Lemma B-e83-e84}) follows.
\end{proof}

\begin{lemma}\label{(M17)-Lemma B}
\begin{equation}\label{equ-Lemma B-(M17)}
\|\iota^*(\sum\limits_{i=1}^{k}(-1)^{i-1}f_0(U_{\de_i,\xi_i})+(-1)^k
f_0(V_\si))-V_{\eps,\la,\xi}\|_\mu \le
\sum\limits_{i=1}^{k}O(\mu\de_i)+O((\mu\si^{\frac{N-2}{2}})^{\frac{1}{2}}).
\end{equation}
\end{lemma}
\begin{proof}
It is similarly to Lemma \ref{(M17)-Lemma A}.
\end{proof}
\begin{lemma}\label{e53-e55-B}
\begin{eqnarray}\label{equ1-Lemma B-e53-e55}
\|(f'_\eps(V_{\eps,\la,\xi})-f'_0(V_{\eps,\la,\xi}))\phi\|_{2N/(N+2)}&=&C\eps\|\phi\|_\mu;\\\label{equ2-Lemma B-e53-e55}
\|f_\eps(V_{\eps,\la,\xi})-f_0(V_{\eps,\la,\xi})\|_{2N/(N+2)}&=&C\eps;\\\label{equ3-Lemma B-e53-e55}
\|f_0(V_{\eps,\la,\xi})-(\sum\limits_{i=1}^{k}(-1)^{i-1}f_0(U_{\de_i,\xi_i})+(-1)^k f_0(V_\si))\|_{2N/(N+2)}&=& O(\eps^\frac{N+2}{2(N-2)}).
\end{eqnarray}
\end{lemma}
\begin{proof}
The first two, as Lemma \ref{e53-e55}, are from
\cite{BarMiPis06CVPDE}. The last one can be proved as (4.5) in
\cite{MussoPis10JMPA}.
\end{proof}

\begin{lemma}\label{e72-e78} Let $k\ge1$. Then
\begin{eqnarray}\label{equ1-Lemma B-e72-e78}
&&\int_{\Om}|\nabla
P V_\si|^2-\mu\frac{|P V_\si|^2}{|x|^2}=S_\mu^{\frac{N}{2}}+o(\eps);\\\label{equ2-Lemma B-e72-e78}
&&\int_{\Om}\nabla
P V_\si\nabla PU_{\de_i,\xi_i}-\mu\frac{P V_\si PU_{\de_i,\xi_i}}{|x|^2}\\\nonumber
&=&
\begin{cases}
C_0^{2^*}(\frac{\ov{\la}}{\la_k})^{\frac{N-2}{2}}\int_{\R^N}\frac{1}{(1+|y|^2)^{\frac{N+2}{2}}}\frac{1}{(1+|\ze_k|^2)^{\frac{N-2}{2}}}\cdot\eps+o(\eps)
\quad& \text{if} ~~i=k,\\\nonumber
o(\eps)\quad& \text{if} ~~i\neq k;
\end{cases}\\\label{equ3-Lemma B-e72-e78}
&&\mu\int_{\Om}\frac{|PU_{\de_i,\xi_i}|^2}{|x|^2}=\mu C_0^2\int_{\R^N}\frac{1}{|y|^2(1+|y-\ze_i|^2)^{N-2}}+o(\eps);\\\label{equ4-Lemma B-e72-e78}
&&\mu\int_{\Om}\frac{PU_{\de_i,\xi_i}PU_{\de_j,\xi_j}}{|x|^2}
=o(\eps),~~~i\neq j;\\\label{equ5-Lemma B-e72-e78}
&&\int_{\Om}|\nabla P U_{\de_i,\xi_i}|^2\\\nonumber
&=&
\begin{cases}
S_0^{\frac{N}{2}}-C_0^{2^*}H(0,0)\la_1^{N-2}\int_{\R^N} \frac{1}{(1+|z|^2)^{\frac{N+2}{2}}}\cdot\eps+o(\eps)
\quad& \text{if} ~~i=1,\\\nonumber
S_0^{\frac{N}{2}}+o(\eps)\quad& \text{if} ~~i\neq 1;
\end{cases}\\\label{equ6-Lemma B-e72-e78}
&&\int_{\Om}\nabla P U_{\de_i,\xi_i}\nabla P U_{\de_j,\xi_j}\\\nonumber
&=&
\begin{cases}
C_0^{2^*}(\frac{\la_{i+1}}{\la_i})^{\frac{N-2}{2}}\int_{\R^N}\frac{1}{(1+|y|^2)^{\frac{N+2}{2}}}\frac{1}{(1+|\ze_i|^2)^{\frac{N-2}{2}}}\cdot\eps+o(\eps)
\quad& \text{if} ~~j=i+1,\\\nonumber
o(\eps)\quad& \text{otherwise},
\end{cases}
\end{eqnarray}
where we assume, without loss of generality, $1\le i<j\le k$.
\end{lemma}
\begin{proof} (\ref{equ1-Lemma B-e72-e78}) and (\ref{equ5-Lemma B-e72-e78}) can be obtained by (\ref{equ1-Lemma A-e12,e24,e30,e34}) and (\ref{equ--add1-Lemma A-e12,e24,e30,e34}), respectively.

(1). Proof of (\ref{equ2-Lemma B-e72-e78}).

By using (\ref{equ-3-projection estimate}), integration by parts yields
\begin{eqnarray}\label{equ7-Lemma B-e72-e78}
\int_{\Om}\nabla P V_\si\nabla PU_{\de_i,\xi_i}-\mu\frac{P
V_\si PU_{\de_i,\xi_i}}{|x|^2}=
\int_{\Om}V_\si^{2^*-1}(U_{\de_i,\xi_i}-\vphi_{\de_i,\xi_i})+\mu\int_\Om\frac{\vphi_\si (U_{\de_i,\xi_i}-\vphi_{\de_i,\xi_i})}{|x|^2}.
\end{eqnarray}

It is easy to show, by using (\ref{equ-1-projection estimate}) and (\ref{equ-3-projection estimate}), that
\begin{eqnarray}\label{equ8-Lemma B-e72-e78}
\int_{\Om}V_\si^{2^*-1}\vphi_{\de_i,\xi_i}&\le&C\de_i^{\frac{N-2}{2}}\int_{\Om}
\frac{\si^{\frac{N+2}{2}}}{(\si^2|x|^{\be_1}+|x|^{\be_2})^{\frac{N+2}{2}}}\\\nonumber
&\le&C\si^{\frac{\ov{\mu}}{\sqrt{\ov{\mu}-\mu}}}\de_i^{\frac{N-2}{2}}\int_{\R^N}
\frac{1}{(|y|^{\be_1}+|y|^{\be_2})^{\frac{N+2}{2}}}=O(\si^{\frac{N-2}{2}}\de_i^{\frac{N-2}{2}}),
\end{eqnarray}
and
\begin{eqnarray}\label{equ9-Lemma B-e72-e78}
\mu\int_\Om\frac{\vphi_\si (U_{\de_i,\xi_i}-\vphi_{\de_i,\xi_i})}{|x|^2}&\le&\mu\int_\Om\frac{\vphi_\si U_{\de_i,\xi_i}}{|x|^2}\le O(\mu\si^{\frac{N-2}{2}}).
\end{eqnarray}

On the other hand,
\begin{eqnarray}\label{equ10-Lemma B-e72-e78}
\int_{\Om}V_\si^{2^*-1} U_{\de_i,\xi_i}=\bigcup\limits_{j=1}^{k+1}\int_{A_j}V_\si^{2^*-1} U_{\de_i,\xi_i}+O(\si^{\frac{N+2}{2}}\de_i^{\frac{N-2}{2}}).
\end{eqnarray}
If $i=k,j=k+1$,
\begin{eqnarray}\label{equ11-Lemma B-e72-e78}
&&\int_{A_{k+1}}V_\si^{2^*-1} U_{\de_k,\xi_k}\\\nonumber
&=&C_\mu^{2^*-1} C_0\si^{\frac{N+2}{2}}\de_k^{\frac{N-2}{2}}\int_{A_{k+1}}\frac{1}{({\si^2|x|^{\be_1}}
+|x|^{\be_2})^{\frac{N+2}{2}}}\frac{1}{(\de_k^2+|x-\xi_k|^2)^{\frac{N-2}{2}}}\\\nonumber
&=&C_\mu^{2^*-1} C_0\frac{\si^{\frac{\ov{\mu}}{\sqrt{\ov{\mu}-\mu}}}}{\de_k^{\frac{N-2}{2}}}
\int_{\frac{A_{k+1}}{\si^{\frac{\sqrt{\ov{\mu}}}{\sqrt{\ov{\mu}-\mu}}}}}\frac{1}{({|y|^{\be_1}}
+|y|^{\be_2})^{\frac{N+2}{2}}}\frac{1}{(1+|\frac{\si^{\frac{\sqrt{\ov{\mu}}}{\sqrt{\ov{\mu}-\mu}}}}{\de_k}y-\ze_k|^2)^{\frac{N-2}{2}}}\\\nonumber
&=& C_0^{2^*}(\frac{\ov{\la}}{\la_k})^{\frac{N-2}{2}}\int_{\R^N}\frac{1}{(1+|y|^2)^{\frac{N+2}{2}}}\frac{1}{(1+|\ze_k|^2)^{\frac{N-2}{2}}}\cdot\eps+o(\eps)
\end{eqnarray}

If $i\neq k$ or $j\neq k+1$, similar arguments as (\ref{equ11-Lemma B-e72-e78}) give
\begin{eqnarray}\label{equ12-Lemma B-e72-e78}
\int_{A_j}V_\si^{2^*-1} U_{\de_i,\xi_i}=o(\eps).
\end{eqnarray}

Then we conclude by (\ref{equ7-Lemma B-e72-e78})-(\ref{equ12-Lemma B-e72-e78}).

\noindent(2). Proof of (\ref{equ3-Lemma B-e72-e78}).
\begin{eqnarray*}
(\ref{equ3-Lemma B-e72-e78})&=&\mu\int_{\Om}\frac{|U_{\de_i,\xi_i}|^2}{|x|^2}+O(\mu\de_i^{N-2})\\
&=&\mu C_0^2\int_{\frac{\Om}{\de_i}}\frac{1}{|y|^2(1+|y-\ze_i|^2)^{N-2}}+O(\mu\de_i^{N-2})\\
&=&\mu
C_0^2\int_{\R^N}\frac{1}{|y|^2(1+|y-\ze_i|^2)^{N-2}}+o(\eps).
\end{eqnarray*}

Similar arguments give that $(\ref{equ4-Lemma B-e72-e78})=o(\eps)$

\noindent(3). Proof of (\ref{equ6-Lemma B-e72-e78}).

Without loss of generality, let $1\le i<j\le k$. Then as (\ref{equ10-Lemma B-e72-e78}),
\begin{eqnarray}\label{equ13-Lemma B-e72-e78}
(\ref{equ6-Lemma B-e72-e78})&=&\int_{\Om}U_{\de_j,\xi_j}^{2^*-1}U_{\de_i,\xi_i}+o(\eps)\\
&=&
\begin{cases}
C_0^{2^*}(\frac{\la_{i+1}}{\la_i})^{\frac{N-2}{2}}\int_{\R^N}\frac{1}{(1+|y|^2)^{\frac{N+2}{2}}}\frac{1}{(1+|\ze_i|^2)^{\frac{N-2}{2}}}\cdot\eps+o(\eps)
\quad& \text{if} ~~j=i+1,\\\nonumber
o(\eps)\quad& \text{otherwise}.
\end{cases}
\end{eqnarray}
\end{proof}

\begin{lemma}\label{e79}Let $k\ge1$. Then
\begin{eqnarray}\label{equ1-Lemma B-e79}
&&\int_{\Om}|\sum\limits_{i=1}^{k}(-1)^{i-1}PU_{\de_i,\xi_i}+(-1)^k
PV_\si|^{2^*}\\\nonumber &=&k
S_0^{\frac{N}{2}}+S_\mu^{\frac{N}{2}}-2^*C_0^{2^*}H(0,0)\la_1^{N-2}\int_{\R^N}
\frac{1}{(1+|z|^2)^{\frac{N+2}{2}}}\cdot\eps\\\nonumber
&&-2^*C_0^{2^*}\sum\limits_{i=1}^{k}(\frac{\la_{i+1}}{\la_i})^{\frac{N-2}{2}}\int_{\R^N}\frac{1}{|y|^{N-2}(1+|y-\ze_i|^2)^{\frac{N+2}{2}}}\cdot\eps\\\nonumber
&&-2^*C_0^{2^*}\sum\limits_{i=1}^{k}(\frac{\la_{i+1}}{\la_i})^{\frac{N-2}{2}}
\int_{\R^N}\frac{1}{(1+|y|^2)^{\frac{N+2}{2}}}\frac{1}{(1+|\ze_i|^2)^{\frac{N-2}{2}}}\cdot\eps+o(\eps),
\end{eqnarray}
where $\la_{k+1}=\ov{\la}$.
\end{lemma}

\begin{proof}
\begin{eqnarray}\label{equ2-Lemma B-e79}
(\ref{equ1-Lemma B-e79})=\bigcup\limits_{j=1}^{k+1}\int_{A_j}|\sum\limits_{i=1}^{k}(-1)^{i-1}PU_{\de_i,\xi_i}+(-1)^k PV_\si|^{2^*}
+O(\de_1^N).
\end{eqnarray}

First of all,
\begin{eqnarray}\label{equ3-Lemma B-e79}
&&\int_{A_k}V_\si U_{\de_k,\xi_k}^{2^*-1}\\\nonumber
&=&C_\mu C_0^{2^*-1}\si^{\frac{N-2}{2}}\de_k^{\frac{N+2}{2}}\int_{A_k}\frac{1}{({\si^2|x|^{\be_1}}
+|x|^{\be_2})^{\frac{N-2}{2}}}\frac{1}{(\de_k^2+|x-\xi_k|^2)^{\frac{N+2}{2}}}\\\nonumber
&=&C_\mu C_0^{2^*-1}\frac{\si^{\frac{N-2}{2}}}{\de_k^{\sqrt{\ov{\mu}-\mu}}}\int_{\frac{A_k}{\de_k}}\frac{1}{({(\frac{\si}{\de_k})^2|y|^{\be_1}}
+|y|^{\be_2})^{\frac{N-2}{2}}}\frac{1}{(1+|y-\ze_k|^2)^{\frac{N+2}{2}}}\\\nonumber
&=& C_0^{2^*}(\frac{\ov{\la}}{\la_k})^{\frac{N-2}{2}}\int_{\R^N}\frac{1}{|y|^{N-2}(1+|y-\ze_k|^2)^{\frac{N+2}{2}}}\cdot\eps+o(\eps),
\end{eqnarray}
and
\begin{eqnarray}\label{equ3add-Lemma B-e79}
&&\int_{A_j}V_\si U_{\de_i,\xi_i}^{2^*-1}=o(\eps), ~\text{if}~i\neq k,~\text{or}~j\neq k.
\end{eqnarray}

From \cite{MussoPis10JMPA}, we also have, for $1\le i<j\le k$,
\begin{eqnarray}\label{equ4-Lemma B-e79}
&&\int_{A_l}U_{\de_j,\xi_j}U_{\de_i,\xi_i}^{2^*-1}+o(\eps)\\
&=&\begin{cases}
C_0^{2^*}(\frac{\la_{i+1}}{\la_i})^{\frac{N-2}{2}}\int_{\R^N}\frac{1}{|y|^{N-2}(1+|y-\ze_i|^2)^{\frac{N+2}{2}}}\cdot\eps+o(\eps)
\quad& \text{if} ~~j=i+1,i=l,\\\nonumber
o(\eps)\quad& \text{otherwise}.
\end{cases}
\end{eqnarray}

Noticing (\ref{equ11-Lemma B-e72-e78}), (\ref{equ13-Lemma B-e72-e78}) and the above three equalities,   then the proof of (\ref{equ1-Lemma B-e79}) is actually involved by Lemma 6.2 in \cite{MussoPis10JMPA}.
\end{proof}

\begin{lemma}\label{e82}
\begin{eqnarray}\label{equ1-Lemma B-e82}
&&\int_{\Om}|\sum\limits_{i=1}^{k}(-1)^{i-1}PU_{\de_i,\xi_i}+(-1)^k PV_\si|^{2^*}\ln|\sum\limits_{i=1}^{k}(-1)^{i-1}PU_{\de_i,\xi_i}+(-1)^k PV_\si|\\\nonumber
&=&-\frac{N-2}{2}\ln\si\cdot\int_{\R^N} V_1^{2^*}-\frac{N-2}{2}\ln(\de_1 \de_2\dots\de_k)\cdot\int_{\R^N} U_{1,0}^{2^*}\\\nonumber
&&+\int_{\R^N} V_1^{2^*}\ln V_1+k\int_{\R^N} U_{1,0}^{2^*}\ln U_{1,0}+o(1).
\end{eqnarray}
\end{lemma}
\begin{proof}
The proof is similarly to Lemma \ref{e42}.
\end{proof}

%% bibliography--------------------------------------------------------------------

\end{document}